\documentclass[12pt]{amsart}
\usepackage[utf8]{inputenc}
\usepackage[T1]{fontenc}
\usepackage{autobreak}
\usepackage[english]{babel}
\usepackage{amsmath, amssymb, amsthm}            
\usepackage{amstext, amsfonts, a4}
\usepackage{graphicx}
\usepackage{hyperref}
\usepackage{color}
\usepackage{stmaryrd}
\usepackage{enumitem}
\usepackage{tikz}
\usepackage{mathrsfs}
\usepackage{amsaddr}
\usepackage{pgfplots}
\usepackage{multirow}
\pgfplotsset{compat=1.15}
\usetikzlibrary{arrows}
\usepackage{comment}
\usepackage[nameinlink]{cleveref}

\theoremstyle{plain}
	\newtheorem{Theo}{Theorem}[subsection] 
	\newtheorem{Prop}[Theo]{Proposition}        
	\newtheorem{Lem}[Theo]{Lemma}            
	\newtheorem{Cor}[Theo]{Corollary}
	\newtheorem{Conj}[Theo]{Conjecture}

\theoremstyle{definition}
	\newtheorem{Def}[Theo]{Definition}
        
	\newtheorem{Nota}[Theo]{Notation}

\theoremstyle{remark}
	\newtheorem{Rema}[Theo]{Remark}

\def\RR{{\mathbb R}}    
\def\CC{{\mathbb C}}    
\newcommand{\Int}{\mbox{Int}}
\newcommand{\Vol}{\mbox{Vol}}
\newcommand{\KVol}{\mbox{KVol}}

\title[KVol on translation surfaces with multiple singularities]{Algebraic interaction strength for translation surfaces with multiple singularities}
\author{Julien Boulanger}
\address{Centro de Modelamiento matem\'atico,\\ Universidad de Chile \& IRL 2807 - CNRS.}
\email{jboulanger@cmm.uchile.cl}

\begin{document}

\begin{abstract}
We compute the maximal ratio of the algebraic intersection of two closed curves on two families of translation surfaces with multiple singularities. This ratio, called the interaction strength, is difficult to compute for translation surfaces with several singularities as geodesics can change direction at singularities. The main contribution of this paper is to deal with this type of surfaces. Namely, we study the interaction strength of the regular $n-$gons for $n \equiv 2 \pmod 4$ and the Bouw-M\"oller surfaces $S_{m,n}$ with $1 < \gcd(m,n) < n$. This answers a conjecture of the author and it completes the study of the algebraic interaction strength KVol on the regular polygon Veech surfaces. Our results on Bouw-M\"oller surfaces extends results of the author with Pasquinelli. This is also the first exact computation of KVol on translation surfaces with several singularities, and the pairs of curves that achieve the best ratio are singular geodesics made of two saddle connections with different directions.
\end{abstract}

\maketitle


\paragraph{MSc} 51H99

\section{Introduction}
In this paper we are interested in the intersection of closed curves on translation surfaces, and especially those made with (semi-)regular polygons. We aim to study the maximal possible number of intersections for two closed curves of given lengths. This is done by considering the so-called \emph{algebraic interaction strength} which is defined for a surface $X$ with a (say, Riemannian possibly with singularities) metric $g$ as
\[ \KVol(X,g) := \Vol(X) \cdot \sup_{\alpha,\beta} \frac{\Int(\alpha,\beta)}{l(\alpha)l(\beta)} \]
where the supremum is taken over pairs of closed curves on $X$, and $\Int(\alpha,\beta)$ represents the algebraic intersection of the closed curves $\alpha$ and $\beta$. This quantity has seen a recent surge of interest, whether it be for flat surfaces \cite{CKM, CKMcras, BLM22, Bou23, Bou23b, BP24}, hyperbolic surfaces\footnote{The former paper studies a similar quantity which is obtained by considering the geometric intersection instead of the algebraic intersection, see the discussion at the end of the introduction.} \cite{Torkaman, Jiang_Pan}, or in general \cite{MM}. In particular, $\KVol(X,g)$ is notoriously difficult to compute explicitely for a given surface $(X,g)$ and, with the purpose of finding examples, Cheboui-Kessi-Massart \cite{CKM, CKMcras} initiated the study of KVol on translation surfaces. In \cite{BLM22}, the author, E.~Lanneau and D.~Massart compute KVol on the double regular $n$-gon ($n$ odd) and its Teichm\"uller disk. Then, in \cite{Bou23}, we compute KVol on the Teichm\"uller disks of the regular $n$-gon $X_n$ for $n \equiv 0 \mod 4$, but we only provide non-sharp estimates in the case $n \equiv 2 \mod 4$. Namely,

\begin{Theo}{\cite[Theorems 1.2 and 1.6]{Bou23}}\label{theo:l_0}
For $n \geq 8$ even, we have
\[ \KVol(X_n) \leq \frac{n}{4 \tan (\pi /n)}, \]
and the bound is sharp if and only if $n \equiv 0 \mod 4$.
\end{Theo}

This is due to the fact that the regular $n$-gon for $n \equiv 2 \mod 4$ (and $n \geq 10$) is a translation surface with two distinct singularities, and the sides of the regular $n$-gon are not closed curves; contrary to the case $n \equiv 0 \mod 4$ for which sides are closed curves and pairs of sides intersecting at the singularity achieve the maximum in the definition of KVol. In this paper we compute the exact value of KVol on the regular $n$-gon for $n \equiv 2 \mod 4$, thus finishing the study of KVol in regular polygons. An interesting (and new) feature of this family of surfaces is that KVol is not achieved by (closed) saddle connections, but rather unions of non-closed saddle connections with different directions. Namely,

\begin{Theo}\label{theo:decagon}
Let $n \geq 10$, $n \equiv 2 \mod 4$. For every pair of closed curves $\gamma, \delta$ on the regular $n$-gon, we have:

\begin{equation}\label{eq:main}
\frac{\Int(\gamma,\delta)}{l(\gamma)l(\delta)} \leq \frac{1}{2 l_0^2} 
\end{equation}
where $l_0$ is the side-length of the $n$-gon.\newline

Further, equality is achieved if and only if $\gamma$ and $\delta$ are both made of two sides of the $n$-gon, and intersect at both singularities with the same sign. 
\end{Theo}

And in particular,
\begin{Cor}
For $n \geq 10$ with $n \equiv 2 \mod 4$, we have
\[ \KVol(X_n) = \frac{n}{8 \tan (\pi /n)}. \]
\end{Cor}

The proof of Theorem \ref{theo:decagon} relies on a subdivision method similar to the one used in \cite{Bou23} for the case $n \equiv 0 \mod 4$. Namely, given a closed geodesic $\alpha$ which is made of saddle connections, we will cut $\alpha$ every time it crosses a side of the regular $n$-gon, and we will group the obtained segments in order to control both the length of each group of segments and the (non-singular) intersections with other groups of segments obtained from the decomposition of a second closed geodesic $\beta$. We will build upon the intersection count done in \cite{Bou23} while refining both the length and intersections estimates. This will be obtained with a careful partition into sequences of segments as well as distinguishing several types of saddle connections. \newline 

In fact, our method extends to another family of surfaces, the \emph{Bouw-M\"oller surfaces}. These surfaces, made from semi-regular polygons, were discovered by Bouw and M\"oller \cite{BM10} in an algebraic setting and were described geometrically by Hooper \cite{Hooper}. Given two integers $m,n \geq 2$ with $(m,n) \neq (2,2)$ there is an associated Bouw-M\"oller surface made of a chain of $m$ semi-regular polygons, each of which has $2n$ sides (apart from two of them, which have $n$ sides). The algebraic interaction strength of these surfaces has been studied by the author and Pasquinelli in \cite{BP24} where $\KVol(S_{m,n})$ is computed for coprime $m,n$. Still, we were unable to compute the interaction strength for non-coprime entries. This paper provides the exact value of $\KVol(S_{m,n})$ for a large family of Bouw-M\"oller surfaces with non-coprime entries. Namely,

\begin{Theo}\label{theo:BM}
Let $m, n \geq 8$ such that $\gcd(m,n) \neq 1$. Then, for any pair of closed curves $\alpha,\beta$ on $S_{m,n}$, we have:
\[
\frac{\Int(\alpha,\beta)}{l(\alpha)l(\beta)} \leq \frac{1}{2l_0^2}
\] 
where $l_0 = \sin \frac{\pi}{m}$ is the length of the smallest sides in the standard polygonal decomposition of $S_{m,n}$. \newline

Further, when $\gcd(m,n) \neq n$, the equality is achieved by a pair of closed curves $\alpha,\beta$ which are both made of two sides of length $l_0$ and which intersect twice. 
\end{Theo}  
The assumption $m, n \geq 8$ is mostly there for simplicity. In fact, the same result should hold for any $m,n$ non-coprime, but the adequate length estimates are more difficult to obtain for small values of $m$ and $n$, and there are more cases to consider. Furthermore, Theorem \ref{theo:BM} also holds for $m = 2$ and $n \geq 10$ (and in fact, the assumption $n \geq 8$ is sufficient): In this case, the proof can be directly adapted from the $4m+2$-gon case, as we are still dealing only with regular polygons. \newline

Finally, the equality case is not achieved when $\gcd(m,n) = n$. This is because the construction of the corresponding closed curves $\alpha$ and $\beta$ fail in this case. In fact, we conjecture:

\begin{Conj}
Let $m,n \geq 3$ with $(m,n) \neq (3,3)$ such that $m$ is a multiple of $n$ (equivalently $gcd(m,n) = n$). Then for any pair of closed curves $(\alpha, \beta)$ on $S_{m,n}$, we have:
\begin{equation}
\frac{\Int(\alpha, \beta)}{l(\alpha)l(\beta)} \leq \frac{1}{4l_0^2}
\end{equation}
where $l_0 = \sin \frac{\pi}{m}$ is the length of the smallest sides in the standard polygonal decomposition of $S_{m,n}$. \newline
Further, equality is achieved for two closed curves $\alpha, \beta$ both made of two sides of length $l_0$, and which intersect once.
\end{Conj}
\begin{Rema}
When $(m,n)=(3,3)$, the resulting Bouw-M\"oller surface has genus one and we know from \cite{MM} that for any pair of closed curves $\alpha,\beta$ on $S_{3,3}$
\begin{equation}
\frac{\Int(\alpha, \beta)}{l(\alpha)l(\beta)} \leq \frac{1}{2\sqrt{3}l_0^2}
\end{equation}
and the inequality is optimal.
\end{Rema}

\subsection{History and motivations.}
The study of the quantity KVol goes back to \cite{these_massart}, although the name \emph{interaction strength} is much more recent as it comes from \cite{Torkaman}, where a similar quantity is studied (namely, the algebraic intersection is replaced by the geometric intersection). The original motivation for the study of $\KVol(X)$ lies in its relation with norms defined in the homology $H_1(X,\RR)$ of the surface $X$, and for which $\KVol(X)$ can be alternatively seen as the best comparison constant between the (equivalent) $L^2$ norm and the stable norm (up to a normalisation, see \cite{MM} for definitions and for a precise statement), or the maximal symplectic area of a parallelogram inscribed in the unit ball of the stable norm.\newline

As a consequence, the study of the interaction strength not only allows to get an estimate on the maximal possible (algebraic) intersection of two closed curves, but it also gives information on these two norms, and especially on the stable norm (which is still mysterious, see e.g. \cite{these_massart, Mas97b, Babenko_Balacheff, Balacheff_Massart, Montealegre24}). Then, one of the first questions one can ask concerns its explicit computation on a given surface. This first question turned out to be far from trivial, and this is the reason why Cheboui, Kessi and Massart \cite{CKM, CKMcras} initiated the study of KVol on translation surfaces, and more specifically on some square-tiled (or arithmetic) translation surfaces. 
Even in this case, the computation of KVol is not straightforward, and for example we do not even know the infimum value of KVol on translation surfaces of genus two with a single singularity (this is expected to be two, see \cite{CKMcras}).

In \cite{BLM22, Bou23} and \cite{BP24}, E.~Lanneau, D.~Massart, I.~Pasquinelli and the author develop a method to provide an upper bound on KVol for a large class of (half-)translation surfaces which are made out of convex polygons having obtuse angles, the estimate being sharp on Bouw-M\"oller surfaces with one singularity \cite{BP24}, and on regular $n$-gons for $n \equiv 0 \mod 4$, see \cite{Bou23}. This estimate also allows to compute KVol on some examples of translation surfaces with several distinct singularities, for example the covers of a regular $n$-gon for odd $n \geq 5$ presented in Figure \ref{fig:pentagon_cover}. However, the estimate of \cite[Theorem 1.3]{BP24} can only be sharp for translation surfaces where the value of KVol is achieved by closed saddle connections, and it also requires that sides of the same polygon are not identified together. In particular, it never allows to compute KVol on translation surfaces of genus two with two disctint singularities, as for singularities of angle $4\pi$, the fact that there is a closed saddle connection and that the surface has a polygonal decomposition made of polygons with only obtuse or right angles (as in \cite{BP24}) imply a self-identification on the sides of a polygon. Our computation of KVol on the regular decagon thus provides the first example of computation on a translation surface with two singularities of angle $4\pi$, 
and we also show that KVol is \emph{not} achieved by closed saddle connections but rather unions of two (non-closed) saddle connections in disctinct directions\footnote{One can always show that if KVol is a maximum, then it must be achieved as a union of at most $s$ saddle connections, where $s$ is the number of singularities of the surface, see Section \ref{sec:preliminaries}.}.

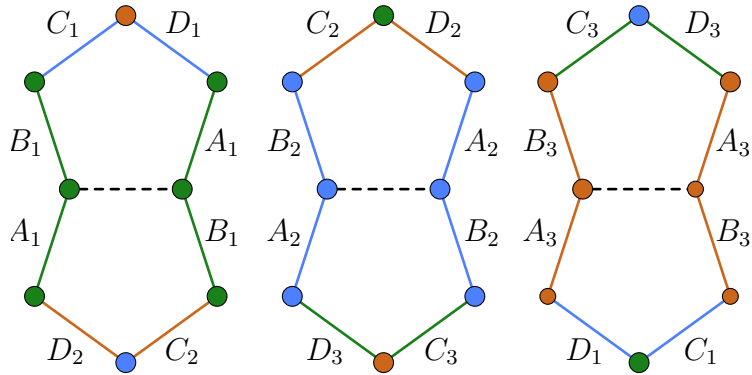
\begin{figure}
\center
\definecolor{qqwuqq}{rgb}{0.1,0.5,0.1}
\definecolor{qqqqff}{rgb}{0.3,0.5,1}
\definecolor{ccqqqq}{rgb}{0.8,0.4,0.1}
\begin{tikzpicture}[line cap=round,line join=round,>=triangle 45,x=1cm,y=1cm,rotate=90,scale =1.5]
\clip(-2,-0.5579655471739381) rectangle (2,6.0656283788948455);
\draw [line width=1pt,dash pattern=on 3pt off 3pt] (0,0)-- (0,1);
\draw [line width=1pt,color=ccqqqq] (0,1)-- (-0.9510565162951532,1.3090169943749475);
\draw [line width=1pt,color=qqqqff] (-0.9510565162951532,1.3090169943749475)-- (-1.5388417685876266,0.5);
\draw [line width=1pt,color=qqqqff] (-1.5388417685876266,0.5)-- (-0.9510565162951536,-0.30901699437494734);
\draw [line width=1pt,color=ccqqqq] (-0.9510565162951536,-0.30901699437494734)-- (0,0);
\draw [line width=1pt,color=ccqqqq] (0,0)-- (0.9510565162951532,-0.30901699437494745);
\draw [line width=1pt,color=qqwuqq] (0.9510565162951532,-0.30901699437494745)-- (1.5388417685876266,0.5);
\draw [line width=1pt,color=qqwuqq] (1.5388417685876266,0.5)-- (0.9510565162951536,1.3090169943749475);
\draw [line width=1pt,color=ccqqqq] (0.9510565162951536,1.3090169943749475)-- (0,1);
\draw [line width=1pt,dash pattern=on 3pt off 3pt] (0,2.2644593560306294)-- (0,3.2644593560306294);
\draw [line width=1pt,color=qqqqff] (0,3.2644593560306294)-- (-0.9510565162951532,3.573476350405577);
\draw [line width=1pt,color=qqwuqq] (-0.9510565162951532,3.573476350405577)-- (-1.5388417685876266,2.7644593560306294);
\draw [line width=1pt,color=qqwuqq] (-1.5388417685876266,2.7644593560306294)-- (-0.9510565162951536,1.955442361655682);
\draw [line width=1pt,color=qqqqff] (-0.9510565162951536,1.955442361655682)-- (0,2.2644593560306294);
\draw [line width=1pt,color=qqqqff] (0,2.2644593560306294)-- (0.9510565162951532,1.955442361655682);
\draw [line width=1pt,color=ccqqqq] (0.9510565162951532,1.955442361655682)-- (1.5388417685876266,2.7644593560306294);
\draw [line width=1pt,color=ccqqqq] (1.5388417685876266,2.7644593560306294)-- (0.9510565162951536,3.573476350405577);
\draw [line width=1pt,color=qqqqff] (0.9510565162951536,3.573476350405577)-- (0,3.2644593560306294);
\draw [line width=1pt,dash pattern=on 3pt off 3pt] (0,4.548977995466573)-- (0,5.548977995466573);
\draw [line width=1pt,color=qqwuqq] (0,5.548977995466573)-- (-0.9510565162951532,5.8579949898415205);
\draw [line width=1pt,color=ccqqqq] (-0.9510565162951532,5.8579949898415205)-- (-1.5388417685876266,5.048977995466573);
\draw [line width=1pt,color=ccqqqq] (-1.5388417685876266,5.048977995466573)-- (-0.9510565162951536,4.2399610010916255);
\draw [line width=1pt,color=qqwuqq] (-0.9510565162951536,4.2399610010916255)-- (0,4.548977995466573);
\draw [line width=1pt,color=qqwuqq] (0,4.548977995466573)-- (0.9510565162951532,4.2399610010916255);
\draw [line width=1pt,color=qqqqff] (0.9510565162951532,4.2399610010916255)-- (1.5388417685876266,5.048977995466573);
\draw [line width=1pt,color=qqqqff] (1.5388417685876266,5.048977995466573)-- (0.9510565162951536,5.8579949898415205);
\draw [line width=1pt,color=qqwuqq] (0.9510565162951536,5.8579949898415205)-- (0,5.548977995466573);
\draw (-0.6,6.2) node[anchor=south west] {$A_1$};
\draw (0.2,4.45) node[anchor=south west] {$A_1$};
\draw (0.2,6.2) node[anchor=south west] {$B_1$};
\draw (-0.6,4.45) node[anchor=south west] {$B_1$};
\draw (-0.6,3.9) node[anchor=south west] {$A_2$};
\draw (0.2,3.9) node[anchor=south west] {$B_2$};
\draw (0.2,2.15) node[anchor=south west] {$A_2$};
\draw (-0.6,2.15) node[anchor=south west] {$B_2$};
\draw (-0.6,1.63) node[anchor=south west] {$A_3$};
\draw (0.2,1.63) node[anchor=south west] {$B_3$};
\draw (0.2,-0.08) node[anchor=south west] {$A_3$};
\draw (-0.6,-0.08) node[anchor=south west] {$B_3$};
\draw (1.25,5.85) node[anchor=south west] {$C_1$};
\draw (1.25,4.8) node[anchor=south west] {$D_1$};
\draw (1.25,3.55) node[anchor=south west] {$C_2$};
\draw (1.25,1.25) node[anchor=south west] {$C_3$};
\draw (1.25,2.5) node[anchor=south west] {$D_2$};
\draw (1.25,0.2) node[anchor=south west] {$D_3$};
\draw (-1.65,0.2) node[anchor=south west] {$C_1$};
\draw (-1.65,1.25) node[anchor=south west] {$D_1$};
\draw (-1.65,4.8) node[anchor=south west] {$C_2$};
\draw (-1.65,5.85) node[anchor=south west] {$D_2$};
\draw (-1.65,3.55) node[anchor=south west] {$D_3$};
\draw (-1.65,2.5) node[anchor=south west] {$C_3$};
\begin{scriptsize}
\draw [fill=ccqqqq] (0,0) circle (2pt);
\draw [fill=ccqqqq] (0,1) circle (2.5pt);
\draw [fill=ccqqqq] (-0.9510565162951532,1.3090169943749475) circle (2pt);
\draw [fill=qqwuqq] (-1.5388417685876266,0.5) circle (2.5pt);
\draw [fill=ccqqqq] (-0.9510565162951536,-0.30901699437494734) circle (2pt);
\draw [fill=ccqqqq] (0.9510565162951532,-0.30901699437494745) circle (2.5pt);
\draw [fill=qqqqff] (1.5388417685876266,0.5) circle (2.5pt);
\draw [fill=ccqqqq] (0.9510565162951536,1.3090169943749475) circle (2.5pt);
\draw [fill=qqqqff] (0,2.2644593560306294) circle (2.5pt);
\draw [fill=qqqqff] (0,3.2644593560306294) circle (2.5pt);
\draw [fill=qqqqff] (-0.9510565162951532,3.573476350405577) circle (2.5pt);
\draw [fill=ccqqqq] (-1.5388417685876266,2.7644593560306294) circle (2.5pt);
\draw [fill=qqqqff] (-0.9510565162951536,1.955442361655682) circle (2.5pt);
\draw [fill=qqqqff] (0.9510565162951532,1.955442361655682) circle (2.5pt);
\draw [fill=qqwuqq] (1.5388417685876266,2.7644593560306294) circle (2.5pt);
\draw [fill=qqqqff] (0.9510565162951536,3.573476350405577) circle (2.5pt);
\draw [fill=qqwuqq] (0,4.548977995466573) circle (2.5pt);
\draw [fill=qqwuqq] (0,5.548977995466573) circle (2.5pt);
\draw [fill=qqwuqq] (-0.9510565162951532,5.8579949898415205) circle (2.5pt);
\draw [fill=qqqqff] (-1.5388417685876266,5.048977995466573) circle (2.5pt);
\draw [fill=qqwuqq] (-0.9510565162951536,4.2399610010916255) circle (2.5pt);
\draw [fill=qqwuqq] (0.9510565162951532,4.2399610010916255) circle (2.5pt);
\draw [fill=ccqqqq] (1.5388417685876266,5.048977995466573) circle (2.5pt);
\draw [fill=qqwuqq] (0.9510565162951536,5.8579949898415205) circle (2.5pt);
\end{scriptsize}
\end{tikzpicture}

\caption{A triple cover of the double regular pentagon with three singularities (of angle $6\pi$) and for which the algebraic interaction strenght can be computed from \cite{BP24} - and is $\frac{3n}{4\tan(\pi /n)}$ -. The same picture can be extended to any number of singularities.}
\label{fig:pentagon_cover}
\end{figure}

\subsection{Questions and conjectures.}
In light of Theorem \ref{theo:decagon} one could wonder as it is studied in \cite{BLM22, Bou23, BP24} how KVol behaves not only on the regular $n$-gon but also on its $SL_2(\RR)-$orbit. However, the main argument to compute of KVol in the orbit of the surfaces considered in \cite{BLM22} and \cite{Bou23} relies on the fact that the surfaces have a single singularity and hence the supremum in the definition of KVol can be considered as a supremum over saddle connections, which is not the case on the regular $n$-gon for $n \equiv 2 \mod 4$, or for Bouw-M\"oller surfaces with non-coprime entries. In fact, even the simpler question of the computation of KVol on the staircase model of the regular $n$-gon is still open for $n \equiv 2 \mod 4$, and it is not clear whether a subdivision method as in the present paper would help.


In another direction, one should notice that for the regular $n$-gon, the supremum in the definition of KVol is actually a maximum. This is consistent with Remark 1.6 of \cite{BLM22}, where a rough argument is given to support that KVol should be a maximum for Veech surfaces of genus at least two.

\subsection{Algebraic vs. geometric intersection.}
We conclude this introduction with a remark on the nature of the intersection used here. As we have seen, a motivation for considering the algebraic intersection comes from the relations with norms in homology, but one may similarly wonder what happens for the \emph{geometric interaction strength} which is defined as its algebraic counterpart but replacing the algebraic intersection with the geometric intersection. These two quantities are not the same in general, even for translation surfaces: for example, the surfaces $L(n,n)$ considered in \cite{CKMcras} have algebraic interaction strength going to two as $n$ goes to infinity while it is easily shown that their geometric interaction strength go to infinity with $n$. Another significative difference is that the algebraic intersection is defined at the homology level whereas the geometric intersection is defined at the free homotopy level, and in particular if one can assume the curves to be simple when studying the algebraic interaction strength this is not the case anymore for the geometric interaction strength.

However, several of the methods used to study one of them also gives information on the other. This is for example the case of all the subdivision methods used in \cite{BLM22, Bou23, BP24} and in the present paper, whose estimates on intersections rely on a count of the number of intersection points for saddle connections, assuming that, in the worst case, all the signs are the same. As a consequence, all the results of these papers can be extended to the geometric interaction strength, and in particular the algebraic and geometric interaction strength coincide on regular polygons (\cite{Bou23} and the present paper) and on all the examples of Bouw-M\"oller surfaces $S_{m,n}$ for which $\KVol$ is known (\cite{BLM22, BP24} and the present paper).

\subsection*{Outline.}
We first recall in Section~\ref{sec:preliminaries} useful background on translation surfaces, as well as the decomposition of a saddle connection into \emph{sandwiched} and \emph{non-sandwiched} segments from \cite{BLM22, Bou23}. Then, Section \ref{sec:decagon} is devoted to the proof of Theorem \ref{theo:decagon}: in Section \ref{sec:study_lengths} we study the length of saddle connections depending on the number of segments in their sandwiched/non-sandwiched decomposition and on four types of saddle connections. In Section \ref{sec:study_intersections} we study the non-singular intersections of pairs of saddle connections depending on their type and the number of segments, and we finally use both estimates to prove Theorem \ref{theo:decagon} in Section \ref{sec:conclusion}. Finally, Section \ref{sec:Bouw-Moller} is devoted to the study of Bouw-M\"oller surfaces and the proof of Theorem \ref{theo:BM}.

\subsection*{Acknowledgments.}
This work was supported by Centro de Modelamiento Matem\'atico (CMM) BASAL fund FB210005 for center of excellence from ANID-Chile. The author thanks Rodolfo Guti\'errez--Romo, Erwan Lanneau and Daniel Massart for discussions and comments related to this paper.


\section{Preliminaries on the regular \texorpdfstring{$n-$}{n}gon}\label{sec:preliminaries}
In this section we recall useful background on the regular $n$-gon translation surface, and we describe the decomposition of a saddle connection into \emph{sandwiched} and \emph{non-sandwiched} segments, following \cite{Bou23}. Then, we provide a general strategy for the proof of Theorems \ref{theo:decagon} and \ref{theo:BM}. 

\subsection{The regular \texorpdfstring{$n$-}{n-}gon translation surface}
A \emph{translation surface} $(X,\omega)$ is a real compact genus $g$ surface $X$ with an atlas $\omega$ such that all transition functions are translations except on a finite set of singularities $\Sigma$, along with a distinguished direction. In fact, it can also be seen as a surface obtained from a finite collection of polygons embedded in $\CC$ by identifying pairs of parallel opposite sides by translation. The resulting surface has a flat metric and a finite number of conical singularities.

For $n \geq 4$ even, one can construct a translation surface by identifying the opposite sides of a regular $n$-gon. The resulting surface has a well defined atlas of charts to $\RR^2$ where the transition functions are translations, except, for $n \geq 8$, at the vertices, which are then singularities. One can check that if $n \equiv 0 \mod 4$ then all vertices are identified to the same point whereas if $n \equiv 2 \mod 4$ the vertices split in two classes, and the resulting surface has two distinct singularities.

As a consequence of the definition, geodesics on translation surfaces are piecewise straight line segments, and they can only change direction at a singularity. Geodesic segments from a singularity to a singularity are referred to as \emph{saddle connections}, and such geodesics play a central role in the study of translation surfaces: every closed curve is freely homotopic (hence homologous) to the union of saddle connections. More, every free homotopy class has a representative element of minimal length which is the union of saddle connections. The reason for this is that a non-singular closed geodesic determines a flat \emph{cylinder} of freely homotopic trajectories, as in Figure \ref{fig:cylinders}, and the boundary of the cylinder is made of saddle connections. For example, the regular $4m+2$-gon is decomposed into $m$ cylinders in the direction of every side, see Figure \ref{fig:cylinders}.

As a consequence of the above, when studying the interaction strength on translation surfaces one can restrict to closed curves made of saddle connections. In fact, if one studies the algebraic interaction strength, a convexity argument shows that one can further restrict to the curves that are simple (see \cite{these_massart}). Although this is not necessary for our argument, it will simplify the exposition to restrict as a first step to closed curves passing at most once through each singularity of the surface.

\begin{figure}
\center
\definecolor{ccqqqq}{rgb}{0.8,0,0}
\definecolor{qqqqff}{rgb}{0,0,1}
\definecolor{cxvqqq}{rgb}{0.7803921568627451,0.3137254901960784,0}
\definecolor{qqwuqq}{rgb}{0,0.39215686274509803,0}
\definecolor{uuuuuu}{rgb}{0.26666666666666666,0.26666666666666666,0.26666666666666666}
\begin{tikzpicture}[line cap=round,line join=round,>=triangle 45,x=1cm,y=1cm,scale = 1.3]
\clip(-1.9,-0.8) rectangle (8.527784548244595,4.8);
\fill[line width=1pt,color=qqwuqq,fill=qqwuqq,fill opacity=0.1] (-0.9009688679024195,0.4338837391175585) -- (1.9009688679024186,0.43388373911755784) -- (1,0) -- (0,0) -- cycle;
\fill[line width=1pt,color=qqwuqq,fill=qqwuqq,fill opacity=0.1] (-0.9009688679024184,3.9474025284172645) -- (0,4.381286267534822) -- (1,4.381286267534822) -- (1.900968867902419,3.947402528417264) -- cycle;
\fill[line width=1pt,color=cxvqqq,fill=cxvqqq,fill opacity=0.1] (-1.5244586697611524,3.165571045949235) -- (-0.9009688679024184,3.9474025284172645) -- (1.900968867902419,3.947402528417264) -- (2.5244586697611524,3.165571045949234) -- cycle;
\fill[line width=1pt,color=cxvqqq,fill=cxvqqq,fill opacity=0.1] (-1.5244586697611524,1.2157152215855882) -- (2.5244586697611524,1.2157152215855875) -- (1.9009688679024186,0.43388373911755784) -- (-0.9009688679024195,0.4338837391175585) -- cycle;
\fill[line width=1pt,color=qqqqff,fill=qqqqff,fill opacity=0.1] (-1.7469796037174667,2.1906431337674115) -- (-1.5244586697611524,3.165571045949235) -- (2.5244586697611524,3.165571045949234) -- (2.7469796037174667,2.1906431337674106) -- (2.5244586697611524,1.2157152215855875) -- (-1.5244586697611524,1.2157152215855882) -- cycle;
\fill[line width=1pt,fill=black,fill opacity=0.1] (3.588984210083918,1.9582800271035101) -- (4.489953077986337,2.392163766221068) -- (8.291890813791175,2.392163766221067) -- (7.390921945888756,1.9582800271035095) -- cycle;
\draw [line width=1pt] (0,0)-- (1,0);
\draw [line width=1pt] (1,4.381286267534822)-- (0,4.381286267534822);
\draw [line width=1pt,dash pattern=on 3pt off 3pt,color=qqwuqq] (-0.9009688679024195,0.4338837391175585)-- (1.9009688679024186,0.43388373911755784);
\draw [line width=1pt,color=qqwuqq] (1.9009688679024186,0.43388373911755784)-- (1,0);
\draw [line width=1pt,color=qqwuqq] (0,0)-- (-0.9009688679024195,0.4338837391175585);
\draw [line width=1pt,color=qqwuqq] (-0.9009688679024184,3.9474025284172645)-- (0,4.381286267534822);
\draw [line width=1pt,color=qqwuqq] (1,4.381286267534822)-- (1.900968867902419,3.947402528417264);
\draw [line width=1pt,dash pattern=on 3pt off 3pt,color=qqwuqq] (1.900968867902419,3.947402528417264)-- (-0.9009688679024184,3.9474025284172645);
\draw [line width=1pt,color=cxvqqq] (-1.5244586697611524,3.165571045949235)-- (-0.9009688679024184,3.9474025284172645);
\draw [line width=1pt,color=cxvqqq] (1.900968867902419,3.947402528417264)-- (2.5244586697611524,3.165571045949234);
\draw [line width=1pt,dash pattern=on 3pt off 3pt,color=cxvqqq] (2.5244586697611524,3.165571045949234)-- (-1.5244586697611524,3.165571045949235);
\draw [line width=1pt,dash pattern=on 3pt off 3pt,color=cxvqqq] (-1.5244586697611524,1.2157152215855882)-- (2.5244586697611524,1.2157152215855875);
\draw [line width=1pt,color=cxvqqq] (2.5244586697611524,1.2157152215855875)-- (1.9009688679024186,0.43388373911755784);
\draw [line width=1pt,color=cxvqqq] (-0.9009688679024195,0.4338837391175585)-- (-1.5244586697611524,1.2157152215855882);
\draw [line width=1pt,color=qqqqff] (-1.7469796037174667,2.1906431337674115)-- (-1.5244586697611524,3.165571045949235);
\draw [line width=1pt,color=qqqqff] (2.5244586697611524,3.165571045949234)-- (2.7469796037174667,2.1906431337674106);
\draw [line width=1pt,color=qqqqff] (2.7469796037174667,2.1906431337674106)-- (2.5244586697611524,1.2157152215855875);
\draw [line width=1pt,color=qqqqff] (-1.5244586697611524,1.2157152215855882)-- (-1.7469796037174667,2.1906431337674115);
\draw [line width=1pt,dash pattern=on 3pt off 3pt,color=qqqqff] (-1.7469796037174667,2.1906431337674115)-- (2.7469796037174667,2.1906431337674106);
\draw (0.21,0.43) node[anchor=north west] {$C_1$};
\draw (0.21,1.1) node[anchor=north west] {$C_2$};
\draw (0.21,1.95) node[anchor=north west] {$C_3$};
\draw (0.21,2.95) node[anchor=north west] {$C_3$};
\draw (0.21,3.85) node[anchor=north west] {$C_2$};
\draw (0.21,4.42) node[anchor=north west] {$C_1$};
\draw [line width=1pt,color=qqwuqq] (3.588984210083918,1.9582800271035101)-- (4.489953077986337,2.392163766221068);
\draw [line width=1pt,color=qqwuqq] (7.390921945888756,1.9582800271035095)-- (8.291890813791175,2.392163766221067);
\draw [line width=1pt,dash pattern=on 3pt off 3pt,color=qqwuqq] (5.489953077986337,2.392163766221068)-- (8.291890813791175,2.392163766221067);
\draw [line width=1pt,dash pattern=on 3pt off 3pt,color=qqwuqq] (6.390921945888756,1.9582800271035095)-- (3.588984210083918,1.9582800271035101);
\draw [line width=1pt] (4.489953077986337,2.392163766221068)-- (5.489953077986337,2.392163766221068);
\draw [line width=1pt] (6.390921945888756,1.9582800271035095)-- (7.390921945888756,1.9582800271035095);
\draw [line width=1pt,dash pattern=on 3pt off 3pt,color=qqwuqq] (5.489953077986337,2.392163766221068)-- (6.390921945888756,1.9582800271035095);
\draw (0.3,4.85) node[anchor=north west] {$1$};
\draw (-0.8,4.6) node[anchor=north west] {$2$};
\draw (1.45,4.6) node[anchor=north west] {$3$};
\draw (0.3,0) node[anchor=north west] {$1$};
\draw (1.45,0.3) node[anchor=north west] {$2$};
\draw (-0.8,0.3) node[anchor=north west] {$3$};
\draw (4.75,2.8) node[anchor=north west] {$1$};
\draw (3.75,2.6) node[anchor=north west] {$2$};
\draw (5.7,2.2) node[anchor=north west] {$3$};
\draw (6.75,2) node[anchor=north west] {$1$};
\draw (8,2.3) node[anchor=north west] {$2$};
\draw [line width=1pt,color=ccqqqq] (4.146457218316958,2.226744878538822)-- (7.941947584372851,2.223639988909662);
\draw [color=ccqqqq](6.75,2.25) node[anchor=north west] {$\alpha$};
\begin{scriptsize}
\draw [color=uuuuuu] (0,0)-- ++(-2pt,-2pt) -- ++(4pt,4pt) ++(-4pt,0) -- ++(4pt,-4pt);
\draw [fill=black] (1,0) circle (2pt);
\draw [color=uuuuuu] (1.9009688679024186,0.43388373911755784)-- ++(-2pt,-2pt) -- ++(4pt,4pt) ++(-4pt,0) -- ++(4pt,-4pt);
\draw [fill=uuuuuu] (2.5244586697611524,1.2157152215855875) circle (2pt);
\draw [color=uuuuuu] (2.7469796037174667,2.1906431337674106)-- ++(-2pt,-2pt) -- ++(4pt,4pt) ++(-4pt,0) -- ++(4pt,-4pt);
\draw [fill=uuuuuu] (2.5244586697611524,3.165571045949234) circle (2pt);
\draw [color=uuuuuu] (1.900968867902419,3.947402528417264)-- ++(-2pt,-2pt) -- ++(4pt,4pt) ++(-4pt,0) -- ++(4pt,-4pt);
\draw [fill=uuuuuu] (1,4.381286267534822) circle (2pt);
\draw [color=uuuuuu] (0,4.381286267534822)-- ++(-2pt,-2pt) -- ++(4pt,4pt) ++(-4pt,0) -- ++(4pt,-4pt);
\draw [fill=uuuuuu] (-0.9009688679024184,3.9474025284172645) circle (2pt);
\draw [color=uuuuuu] (-1.5244586697611524,3.165571045949235)-- ++(-2pt,-2pt) -- ++(4pt,4pt) ++(-4pt,0) -- ++(4pt,-4pt);
\draw [fill=uuuuuu] (-1.7469796037174667,2.1906431337674115) circle (2pt);
\draw [color=uuuuuu] (-1.5244586697611524,1.2157152215855882)-- ++(-2pt,-2pt) -- ++(4pt,4pt) ++(-4pt,0) -- ++(4pt,-4pt);
\draw [fill=uuuuuu] (-0.9009688679024195,0.4338837391175585) circle (2pt);
\draw [fill=black] (3.588984210083918,1.9582800271035101) circle (2pt);
\draw [color=uuuuuu] (4.489953077986337,2.392163766221068)-- ++(-2pt,-2pt) -- ++(4pt,4pt) ++(-4pt,0) -- ++(4pt,-4pt);
\draw [fill=uuuuuu] (5.489953077986337,2.392163766221068) circle (2pt);
\draw [color=uuuuuu] (6.390921945888756,1.9582800271035095)-- ++(-2pt,-2pt) -- ++(4pt,4pt) ++(-4pt,0) -- ++(4pt,-4pt);
\draw [fill=uuuuuu] (7.390921945888756,1.9582800271035095) circle (2pt);
\draw [color=uuuuuu] (8.291890813791175,2.392163766221067)-- ++(-2pt,-2pt) -- ++(4pt,4pt) ++(-4pt,0) -- ++(4pt,-4pt);
\draw [fill=black] (7.39092194588874,1.9582800271035055) circle (2pt);
\draw [fill=ccqqqq,shift={(4.146457218316958,2.226744878538822)}] (0,0) ++(0 pt,3.75pt) -- ++(3.2475952641916446pt,-5.625pt)--++(-6.495190528383289pt,0 pt) -- ++(3.2475952641916446pt,5.625pt);
\draw [fill=ccqqqq,shift={(7.941947584372851,2.223639988909662)}] (0,0) ++(0 pt,3.75pt) -- ++(3.2475952641916446pt,-5.625pt)--++(-6.495190528383289pt,0 pt) -- ++(3.2475952641916446pt,5.625pt);
\end{scriptsize}
\end{tikzpicture}
\caption{The decomposition into three horizontal cylinders on the $14$-gon. On the right, the unfolding of the smallest cylinder $C_1$. The closed geodesic $\alpha$ is homologous to the union of two saddle connections.}
\label{fig:cylinders}
\end{figure}
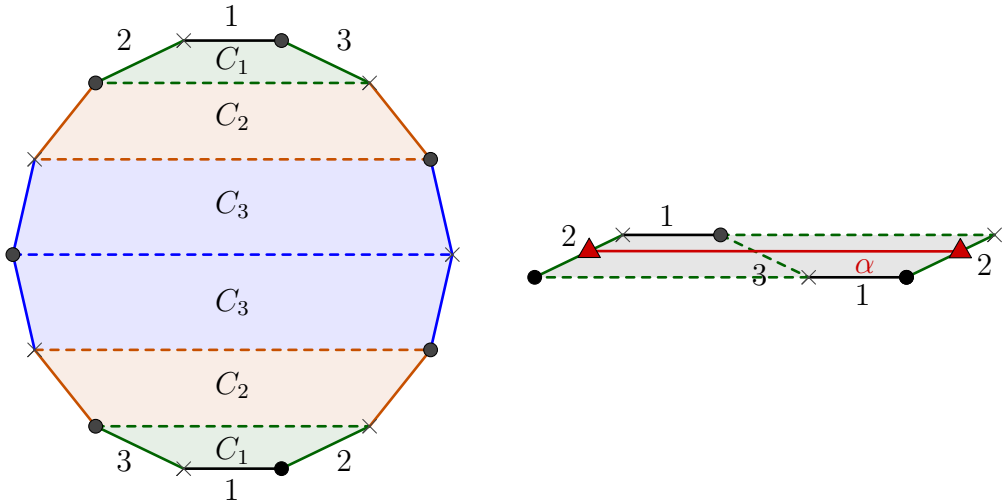

\subsection{Algebraic intersection in translation surfaces}
\begin{figure}
\center
\definecolor{ccqqqq}{rgb}{0.8,0,0}
\definecolor{qqwuqq}{rgb}{0,0.39215686274509803,0}
\begin{tikzpicture}[line cap=round,line join=round,>=triangle 45,x=1cm,y=1cm]
\clip(-0.2,0) rectangle (6.2,3.2);
\draw [line width=1pt,color=qqwuqq] (0,1)-- (1,2);
\draw [line width=1pt,color=qqwuqq,-to] (1,2)-- (2,3);
\draw [line width=1pt,color=ccqqqq] (2,1)-- (1,2);
\draw [line width=1pt,color=ccqqqq,-to] (1,2)-- (0,3);
\draw [line width=1pt,color=qqwuqq,-to] (5,2)-- (6,3);
\draw [line width=1pt,color=ccqqqq,to-] (6,1)-- (5,2);
\draw [line width=1pt,color=ccqqqq] (5,2)-- (4,3);
\draw [line width=1pt,color=qqwuqq] (4,1)-- (5,2);
\draw (1.7035162747177381,2.756789768387328) node[anchor=north west] {$\alpha$};
\draw (1.2,2.2) node[anchor=north west] {$P$};
\draw (5.2,2.2) node[anchor=north west] {$P$};
\draw (0.25,3.2) node[anchor=north west] {$\beta$};
\draw (5.69958264519532,2.78493108085548) node[anchor=north west] {$\alpha$};
\draw (5.720688629546434,1.6803845664805104) node[anchor=north west] {$\beta$};
\draw (-0.3,0.93463978607448) node[anchor=north west] {$\varepsilon_P(\alpha,\beta) = +1$};
\draw (3.55,0.906498473606328) node[anchor=north west] {$\varepsilon_P(\alpha,\beta) = -1$};
\begin{scriptsize}
\draw [fill=black] (1,2) circle (2.5pt);
\draw [fill=black] (5,2) circle (2.5pt);
\end{scriptsize}
\end{tikzpicture}
\caption{The sign of a transverse intersection.}
\label{fig:algebraic_intersection}
\end{figure}
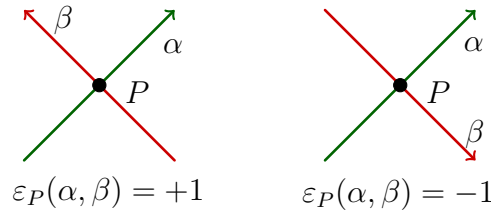
Given two oriented closed curves $\alpha, \beta$ on a smooth surface, which we assume to be in transverse position, the algebraic intersection $\Int(\alpha,\beta)$ is defined as the sum of signs at each intersection points, where the sign at an intersection point $P$ is set to $+1$ if turning counter-clockwise around the point $P$, one sees $\alpha$ going away from $P$, then $\beta$ going away from $P$, then $\alpha$ again, going towards $P$, and then $\beta$ going towards $P$. See Figure \ref{fig:algebraic_intersection}. In the specific case of translation surfaces, given a plane template (at the regular $n$-gon for even $n$), the sign of intersections is easily computed outside the singularities, but in order to compute the sign at singularities one as to look at identifications, and construct a neighborhood of the singularity which makes clear the cyclic order of the curves. In the case of the regular $n$-gon, if we label the sides of the regular $n$-gon $1,2, \dots, \frac{n}{2}$ in cyclic clockwise order, then one can check that the cyclic counter-clockwise order of the sides we see around each singularity is the natural (cyclic) order on $\{1,2, \dots, \frac{n}{2}\}$, see Figure \ref{fig:cyclic_order} for the decagon. In particular, if we choose $\gamma$ (resp. $\delta$) to be the union of the two sides of label $s_1$ and $s_2$ (resp. $t_1$ and $t_2$) with $1 \leq s_1 < t_1 < s_2 < t_2 \leq \frac{n}{2}$ (orienting consistently the sides so that $\gamma$ and $\delta$ have a well defined orientation) then $\gamma$ and $\delta$ intersect at both singularities with the same sign, and $|\Int(\gamma,\delta)| = 2$. From Theorem \ref{theo:decagon}, these are pairs of curves $\gamma$ and $\delta$ achieving KVol.

\begin{Rema}
For these pairs of closed curves, the geometric intersection is also two. More generally, the geometric intersection and the algebraic intersection coincide when all the intersection points have the same sign. As already said in the introduction, our proof of Theorem~\ref{theo:decagon} will rely on a count on intersection points for saddle connections, assuming that, in the worst cases, all intersections have the same signs\footnote{For saddle connections, all non-singular intersections have the same sign and this sign can only change at singularities.}. Hence, Theorem~\ref{theo:decagon} also holds when the algebraic intersection is replaced with the geometric intersection, with the same equality cases.
\end{Rema}

\begin{figure}
\center
\definecolor{qqwuqq}{rgb}{0,0.39215686274509803,0}
\definecolor{ccqqqq}{rgb}{0.8,0,0}
\begin{tikzpicture}[line cap=round,line join=round,>=triangle 45,x=1cm,y=1cm,scale=1.4]
\clip(-1.392653000757567,-0.5) rectangle (7.0574632628342675,3.5);
\draw [line width=1pt,color=qqwuqq] (0,0)-- (1,0);
\draw [line width=1pt,color=ccqqqq] (1,0)-- (1.809016994374947,0.5877852522924729);
\draw [line width=1pt] (1.809016994374947,0.5877852522924729)-- (2.118033988749895,1.5388417685876261);
\draw [line width=1pt,color=qqwuqq] (2.118033988749895,1.5388417685876261)-- (1.8090169943749475,2.4898982848827793);
\draw [line width=1pt,color=ccqqqq] (1.8090169943749475,2.4898982848827793)-- (1,3.0776835371752527);
\draw [line width=1pt,color=qqwuqq] (1,3.0776835371752527)-- (0,3.077683537175253);
\draw [line width=1pt,color=ccqqqq] (0,3.077683537175253)-- (-0.809016994374947,2.4898982848827798);
\draw [line width=1pt] (-0.809016994374947,2.4898982848827798)-- (-1.1180339887498945,1.5388417685876268);
\draw [line width=1pt,color=qqwuqq] (-1.1180339887498945,1.5388417685876268)-- (-0.8090169943749475,0.5877852522924734);
\draw [line width=1pt,color=ccqqqq] (-0.8090169943749475,0.5877852522924734)-- (0,0);
\draw (0.3,0) node[anchor=north west] {$1$};
\draw (-0.75,0.4) node[anchor=north west] {$2$};
\draw (1.4,0.4) node[anchor=north west] {$5$};
\draw (-1.4,1.2) node[anchor=north west] {$3$};
\draw (1.95,1.2) node[anchor=north west] {$4$};
\draw (-1.4,2.3) node[anchor=north west] {$4$};
\draw (1.95,2.3) node[anchor=north west] {$3$};
\draw (-0.75,3.1) node[anchor=north west] {$5$};
\draw (1.4,3.1) node[anchor=north west] {$2$};
\draw (0.3,3.5) node[anchor=north west] {$1$};
\draw (4.15,1.7) node[anchor=north west] {$1$};
\draw (3.55,0.85) node[anchor=north west] {$5$};
\draw (2.55,1.2) node[anchor=north west] {$4$};
\draw (2.55,2.2) node[anchor=north west] {$3$};
\draw (3.55,2.55) node[anchor=north west] {$2$};
\draw [line width=1pt,color=qqwuqq,-to] (3.5,1.5)-- (4.194486122028071,1.4977033742653243);
\draw [line width=1pt,color=ccqqqq,-to] (3.5,1.5)-- (3.720461381342655,2.158568923891677);
\draw [line width=1pt,color=qqwuqq] (3.5,1.5)-- (2.945459193349563,1.911963291209176);
\draw [line width=1pt] (3.5,1.5)-- (2.9405062404997118,1.098687078767865);
\draw [line width=1pt,color=ccqqqq] (3.5,1.5)-- (3.7124473352869205,0.8426603699198562);
\draw (5,1.7) node[anchor=north west] {$1$};
\draw (5.5,2.55) node[anchor=north west] {$5$};
\draw (6.5,2.2) node[anchor=north west] {$4$};
\draw (6.5,1.2) node[anchor=north west] {$3$};
\draw (5.5,0.85) node[anchor=north west] {$2$};
\draw [line width=1pt,color=qqwuqq] (6,1.5)-- (5.305513877971929,1.4977033742653243);
\draw [line width=1pt,color=ccqqqq,-to] (6,1.5)-- (5.7795386186573445,2.158568923891677);
\draw [line width=1pt] (6,1.5)-- (6.5545408066504365,1.9119632912091762);
\draw [line width=1pt,color=qqwuqq,-to] (6,1.5)-- (6.559493759500288,1.0986870787678653);
\draw [line width=1pt,color=ccqqqq] (6,1.5)-- (5.787552664713079,0.8426603699198563);
\begin{scriptsize}
\draw [fill=black] (0,0) circle (3pt);
\draw [color=black] (1,0)-- ++(-3pt,-3pt) -- ++(6pt,6pt) ++(-6pt,0) -- ++(6pt,-6pt);
\draw [fill=black] (1.809016994374947,0.5877852522924729) circle (3pt);
\draw [color=black] (2.118033988749895,1.5388417685876261)-- ++(-3pt,-3pt) -- ++(6pt,6pt) ++(-6pt,0) -- ++(6pt,-6pt);
\draw [fill=black] (1.8090169943749475,2.4898982848827793) circle (3pt);
\draw [color=black] (1,3.0776835371752527)-- ++(-3pt,-3pt) -- ++(6pt,6pt) ++(-6pt,0) -- ++(6pt,-6pt);
\draw [fill=black] (0,3.077683537175253) circle (3pt);
\draw [color=black] (-0.809016994374947,2.4898982848827798)-- ++(-3pt,-3pt) -- ++(6pt,6pt) ++(-6pt,0) -- ++(6pt,-6pt);
\draw [fill=black] (-1.1180339887498945,1.5388417685876268) circle (3pt);
\draw [color=black] (-0.8090169943749475,0.5877852522924734)-- ++(-3pt,-3pt) -- ++(6pt,6pt) ++(-6pt,0) -- ++(6pt,-6pt);
\draw [fill=black] (3.5,1.5) circle (3pt);
\draw [color=black] (6,1.5)-- ++(-3pt,-3pt) -- ++(6pt,6pt) ++(-6pt,0) -- ++(6pt,-6pt);
\end{scriptsize}
\end{tikzpicture}
\caption{The cyclic order of the sides at the singularties. The segments $\gamma = 1 \cup 3$ and $\delta = 2 \cup 5$ intersect twice.}
\label{fig:cyclic_order}
\end{figure}
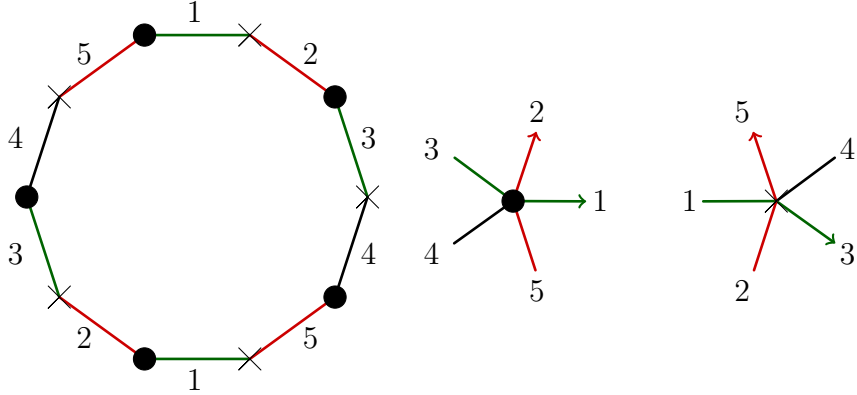

\subsection{Sectors, transition diagrams, and subdivisions}
We now explain how to subdivide saddle connections into smaller \emph{sandwiched} and \emph{non-sandwiched} segments in order to be able to control both their length and the number of intersections, as done in \cite{BLM22, Bou23}. The main goal of this paper is to go one step further in the study of this decomposition in order to improve the length estimates and therefore deal with the case where there are two singularities.\newline

Let $\alpha$ be a saddle connection which we assume is neither a side nor a diagonal of the $n$-gon. We let $\theta_{\alpha} \in (0, \pi)$ be the angle between the direction of $\alpha$ and the horizontal (which is set as the direction of a side), and we partition $(0,\pi)$ into $n$ sub-intervals $\Sigma_{i} = (\frac{i \pi}{n}, \frac{(i+1)\pi}{n})$ for $i \in \{0, \cdots, n-1\}$.

\begin{Def}
We will say that $\alpha$ has \emph{sector} $\Sigma_i$ if $\theta_\alpha \in \Sigma_i$.
\end{Def}

As used in \cite{SU, Bou23}, the sector determines a possible \emph{transition diagram} of the cutting sequence of sides crossed by $\alpha$. More precisely, the transition diagram associated to the sector $\Sigma_i$ is the graph whose vertices corresponds to the sides of the regular $n$-gon (up to identification) and where there is an edge from the sides of label $A$ to the sides of label $B$ if there is a segment going from a side of label $A$ to a side of label $B$ whose direction belong to $\Sigma_i$. In particular, if we label the sides of the regular $n$-gon by numbers $1, 2, \dots, n/2$ (there are only $n/2$ sides up to identification), then one can check that the transition diagram corresponding to the sector $\Sigma_i$ has the form
\[ \sigma_i(1) \leftrightharpoons \sigma_i(2) \leftrightharpoons \sigma_i(3) \cdots \leftrightharpoons \sigma_i(n/2) \circlearrowright \]
where $\sigma_i \in \mathfrak{S}_{n/2}$ is a permutation (see Figure~\ref{fig:sectors_decagon}).

\begin{Nota}
We will write $\sigma_{\alpha}$ to refer to the permutation $\sigma_i$ associated to the sector $\Sigma_i$ corresponding to $\alpha$.
\end{Nota}

\begin{figure}
\center
\definecolor{qqwuqq}{rgb}{0,0.39215686274509803,0}
\definecolor{uuuuuu}{rgb}{0.26666666666666666,0.26666666666666666,0.26666666666666666}
\begin{tikzpicture}[line cap=round,line join=round,>=triangle 45,x=1cm,y=1cm, scale = 1.7]
\clip(-1.8,-0.35) rectangle (2.8,3.35);
\fill[line width=1pt,fill=black,fill opacity=0.05] (0,0) -- (1,0) -- (1.809016994374947,0.5877852522924729) -- (2.118033988749895,1.5388417685876261) -- (1.8090169943749475,2.4898982848827793) -- (1,3.0776835371752527) -- (0,3.077683537175253) -- (-0.809016994374947,2.4898982848827798) -- (-1.1180339887498945,1.5388417685876268) -- (-0.8090169943749475,0.5877852522924734) -- cycle;
\draw [shift={(0,0)},line width=1pt,color=qqwuqq,fill=qqwuqq,fill opacity=0.10000000149011612] (0,0) -- (0:0.5664025228310235) arc (0:18:0.5664025228310235) -- cycle;
\draw [line width=1pt] (0,0)-- (1,0);
\draw [line width=1pt] (1,0)-- (1.809016994374947,0.5877852522924729);
\draw [line width=1pt] (1.809016994374947,0.5877852522924729)-- (2.118033988749895,1.5388417685876261);
\draw [line width=1pt] (2.118033988749895,1.5388417685876261)-- (1.8090169943749475,2.4898982848827793);
\draw [line width=1pt] (1.8090169943749475,2.4898982848827793)-- (1,3.0776835371752527);
\draw [line width=1pt] (1,3.0776835371752527)-- (0,3.077683537175253);
\draw [line width=1pt] (0,3.077683537175253)-- (-0.809016994374947,2.4898982848827798);
\draw [line width=1pt] (-0.809016994374947,2.4898982848827798)-- (-1.1180339887498945,1.5388417685876268);
\draw [line width=1pt] (-1.1180339887498945,1.5388417685876268)-- (-0.8090169943749475,0.5877852522924734);
\draw [line width=1pt] (-0.8090169943749475,0.5877852522924734)-- (0,0);
\draw [line width=1pt,dash pattern=on 3pt off 3pt] (0,0)-- (1.809016994374947,0.5877852522924729);
\draw [line width=1pt,dash pattern=on 3pt off 3pt] (1.809016994374947,0.5877852522924729)-- (-0.8090169943749475,0.5877852522924734);
\draw [line width=1pt,dash pattern=on 3pt off 3pt] (-0.8090169943749475,0.5877852522924734)-- (2.118033988749895,1.5388417685876261);
\draw [line width=1pt,dash pattern=on 3pt off 3pt] (2.118033988749895,1.5388417685876261)-- (-1.1180339887498945,1.5388417685876268);
\draw [line width=1pt,dash pattern=on 3pt off 3pt] (-1.1180339887498945,1.5388417685876268)-- (1.8090169943749475,2.4898982848827793);
\draw [line width=1pt,dash pattern=on 3pt off 3pt] (1.8090169943749475,2.4898982848827793)-- (-0.809016994374947,2.4898982848827798);
\draw [line width=1pt,dash pattern=on 3pt off 3pt] (-0.809016994374947,2.4898982848827798)-- (1,3.0776835371752527);
\draw (0.2,0) node[anchor=north west] {$\sigma_0(1)$};
\draw (1.4,0.45) node[anchor=north west] {$\sigma_0(2)$};
\draw (-1,0.45) node[anchor=north west] {$\sigma_0(3)$};
\draw (1.9,1.2) node[anchor=north west] {$\sigma_0(4)$};
\draw (-1.6,1.2) node[anchor=north west] {$\sigma_0(5)$};
\draw (1.9,2.25) node[anchor=north west] {$\sigma_0(5)$};
\draw (-1.6,2.25) node[anchor=north west] {$\sigma_0(4)$};
\draw (1.4,3.15) node[anchor=north west] {$\sigma_0(3)$};
\draw (-1,3.15) node[anchor=north west] {$\sigma_0(2)$};
\draw (0.2,3.4) node[anchor=north west] {$\sigma_0(1)$};
\draw (0.55,0.3) node[anchor=north west] {$\Sigma_0$};
\begin{scriptsize}
\draw [fill=uuuuuu] (0,0) circle (2pt);
\draw [color=black] (1,0)-- ++(-2pt,-2pt) -- ++(4pt,4pt) ++(-4pt,0) -- ++(4pt,-4pt);
\draw [fill=uuuuuu] (1.809016994374947,0.5877852522924729) circle (2pt);
\draw [color=uuuuuu] (2.118033988749895,1.5388417685876261)-- ++(-2pt,-2pt) -- ++(4pt,4pt) ++(-4pt,0) -- ++(4pt,-4pt);
\draw [fill=uuuuuu] (1.8090169943749475,2.4898982848827793) circle (2pt);
\draw [color=uuuuuu] (1,3.0776835371752527)-- ++(-2pt,-2pt) -- ++(4pt,4pt) ++(-4pt,0) -- ++(4pt,-4pt);
\draw [fill=uuuuuu] (0,3.077683537175253) circle (2pt);
\draw [color=uuuuuu] (-0.809016994374947,2.4898982848827798)-- ++(-2pt,-2pt) -- ++(4pt,4pt) ++(-4pt,0) -- ++(4pt,-4pt);
\draw [fill=uuuuuu] (-1.1180339887498945,1.5388417685876268) circle (2pt);
\draw [color=uuuuuu] (-0.8090169943749475,0.5877852522924734)-- ++(-2pt,-2pt) -- ++(4pt,4pt) ++(-4pt,0) -- ++(4pt,-4pt);
\end{scriptsize}
\end{tikzpicture}
\caption{The permutation $\sigma_0$ associated to the transitition diagram in sector $\Sigma_0$ for $n=10$.}
\label{fig:sectors_decagon}
\end{figure}
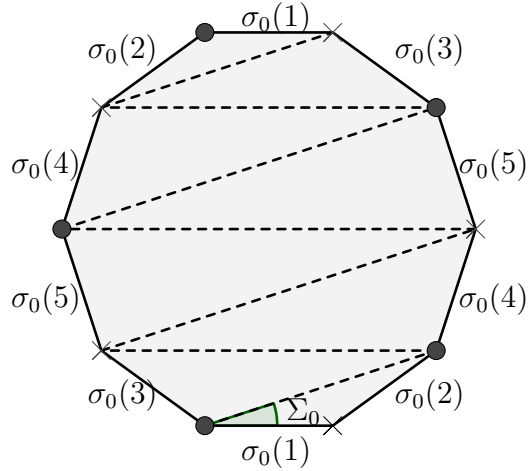

Now, let us subdivide $\alpha$ into shorter segments by cutting it every time it crosses a side of the regular $n$-gon. In order to obtain length estimates, one can consider two types of segments: the \emph{adjacent segments} which go from the interior of a side of the regular $n$-gon to the interior of an adjacent side of the $n$-gon, and the other segments, referred to as \emph{non-adjacent segments}. One of the main ideas of \cite{BLM22}, and then \cite{Bou23}, is to notice that non-adjacent segments have length at least $l_0$, the side-length of the regular $n$-gon, and that, although we cannot control the length of adjacent segments separately, the total length of two consecutive adjacent segments is always greater than $l_0$. More, adjacent segments correspond to segments for which one of the endpoints lies on a side of label $\sigma_\alpha(1)$, and as a consequence they always come in pairs: grouping adjacent segments in pairs, we obtain a so called \emph{sandwiched segment}, whereas the other segments (i.e. the non-adjacent segments) are refereed to as \emph{non-sandwiched segments}.

\begin{Rema}
This idea was later generalized in \cite{BP24} under general hypotheses on polygons and identifications. In this general context adjacent segments do not always come in pairs and one has to be very careful in order to obtain intersection estimates. For example, this behavior appears for Bouw-M\"oller surfaces considered in Section \ref{sec:Bouw-Moller}.
\end{Rema}

Another way to define the decomposition is to first remark that the side $\sigma_i(1)$ can only be preceded and followed by a single side, $\sigma_i(2)$, adjacent to $\sigma_i(1)$. For this reason we refer to $\sigma_i(1)$ as the \emph{sandwiched} side in the sector $\Sigma_i$, whereas the other sides will be referred to as \emph{non-sandwiched} in the sector $\Sigma_i$. Now, given a saddle connection $\alpha$ which is neither a side nor a diagonal, we subdivide $\alpha$ into smaller (non-closed) segments by cutting $\alpha$ every time it crosses a non-sandwiched side (in the sector corresponding to $\alpha$). This determines a decomposition $\alpha = \alpha_1 \cup \cdots \cup \alpha_k$, where $k \geq 2$ and each segment is either:
\begin{itemize}
\item A \emph{non-sandwiched segment} which goes from a side of the $n$-gon to another non-adjacent side of the $n$-gon
\item A \emph{sandwiched segment} which goes from the interior of a side of label $\sigma_i(2)$ to the interior of the other side of label $\sigma_i(2)$ and intersects a side of label $\sigma_i(1)$.
\item An initial or terminal segment $\alpha_1$ or $\alpha_k$. These segments will be considered as non-sandwiched segments.
\end{itemize}

By convention, if $\alpha$ is a side or a diagonal, we will say that the decomposition of $\alpha$ is $\alpha = \alpha_1$, made of a single segment.

\begin{Nota}\label{nota:nalpha1}
We will denote by $n_{\alpha}$ the number of segments in the above decomposition of $\alpha$.
\end{Nota}


\subsection{Length and intersection estimates}
As already hinted, the main reason for considering this specific decomposition of a saddle connection is twofold: first, the construction ensures that every segment of the decomposition has length at least $l_0=1$, the side length of the regular $n-$gon (which we will assume to be 1). This means in particular that the total length of a saddle connection $\alpha$ is at least $n_\alpha$. Second, it is easily seen that two non-sandwiched segments can only intersect once, and in fact one can control the number of intersections in terms of the total number of segments, as done in \cite[Section 3.3]{Bou23}. For this, one should distinguish the following class of saddle connections:

\begin{Def}\label{def:big_cyl}
The saddle connections which are \emph{strictly contained inside a big cylinder} are the longest diagonals and the saddle connections $\alpha$ for which the only sides appearing in the cutting sequence of $\alpha$ are $\sigma_\alpha(n/2-1)$ and $\sigma_\alpha(n/2)$, and the endpoints of $\alpha$ are the common endpoints of sides with label $\sigma_\alpha(n/2-1)$ and $\sigma_\alpha(n/2)$.
\end{Def}

See an example on the right of Figure \ref{fig:ex_type_2}. With this definition, we have the following intersection estimates:

\begin{Prop}{\cite[Section 3.3]{Bou23}}\label{prop:intersection_BOU23}
Given two saddle connections $\alpha$ and $\beta$ on the regular $n$-gon ($n \geq 8$ even), the number $|\alpha \cap \beta|$ of non-singular intersections between $\alpha$ and $\beta$ satisfies
\[ |\alpha \cap \beta| \leq n_{\alpha} n_{\beta} \]
unless one of $\alpha$ or $\beta$ (say $\alpha$) is strictly contained inside a big cylinder, and in this case we have:
\[ |\alpha \cap \beta| \leq n_{\alpha}(n_{\beta} + q_{\beta}) \]
where $q_{\beta}$ is the number of sandwiched segments of $\beta$.
\end{Prop}

The second estimate of Proposition \ref{prop:intersection_BOU23} comes directly from the fact that each non-sandwiched segment of $\alpha$ can only intersect once a non-sandwiched segment, and twice a sandwiched segment. The first estimate requires more work. 

\subsection{General strategy.} The main idea of \cite{Bou23} in the proof of the Theorem \ref{theo:l_0} is to use the above bounds to obtain
\[ 
\frac{\Int(\alpha, \beta)}{l(\alpha)l(\beta)} \lessapprox  \frac{n_\alpha n_\beta}{\ n_\alpha  n_\beta} = 1,
\]
although in reality there are two issues to deal with:
\begin{itemize}
\item Since the intersection estimate $|\alpha \cap \beta| \leq n_{\alpha} n_{\beta}$ fails for saddle connections which are strictly contained inside a big cylinder, these saddle connections have to be considered separately, and their length estimates have to be improved;
\item the intersection estimate do not take into account the singular intersections. Here again, a small margin is required in the length estimates, for saddle connections which are not sides, to compensate for a possible singular intersection.
\end{itemize}

In the present paper, and both for the proofs of Theorems \ref{theo:decagon} and \ref{theo:BM}, the general idea is similar: we roughly want to show that in general the length of a saddle connection $\alpha$ is at least $l(\alpha) \geq \sqrt2 n_\alpha$, which, using the same intersection estimate\footnote{In the case of Bouw-M\"oller surfaces discussed in Section \ref{sec:Bouw-Moller}, a similar intersection estimate holds, see Proposition \ref{prop:intersections_BM}.}, would give:
\begin{equation*}
\frac{\Int(\alpha, \beta)}{l(\alpha)l(\beta)} \lessapprox \frac{n_\alpha n_\beta}{\sqrt{2} n_\alpha \sqrt{2} n_\beta} = \frac{1}{2}.
\end{equation*}
Here again, the above two issues have to be dealt with\footnote{For Bouw-M\"oller surfaces, there is no need to consider an analog of saddle connections strictly contained in a big cylinder, and only the second -- and the third -- issue has to be dealt with}, but there is actually a third issue coming from the fact that saddle connections may not be closed: here $\alpha$ and $\beta$ have to be considered as unions of possibly several saddle connections. Although we will work separately with saddle connections in a first step, a large part of the argument will be devoted to analyze what happens for specific unions of saddle connections.

\section{Proof of Theorem \ref{theo:decagon}}\label{sec:decagon}

\subsection{Outline.}
In order to overcome the three above issues we will distinguish four types of saddle connections and provide specific length and intersection estimates in each of the cases. These are the following:

\begin{enumerate}
\item sides,
\item the saddle connection $\Delta$ in the left of Figure \ref{fig:ex_type_2}, starting from the bottom left vertex of the regular $n$-gon with direction $(2+\cos(2\pi/n),\sin(2\pi/n))$, and its symmetric images by the action of the dihedral group.
\item The saddle connections which are strictly contained inside the big cylinder associated to their sector.
\item all other saddle connections.
\end{enumerate}

\begin{Rema}
Saddle connections of type (1), (2) and (3) are always non-closed, whereas saddle connections of type (4) can either be closed or non-closed.
\end{Rema}

\begin{figure}
\center
\definecolor{ccqqqq}{rgb}{0.8,0,0}
\definecolor{uuuuuu}{rgb}{0.26666666666666666,0.26666666666666666,0.26666666666666666}
\begin{tikzpicture}[line cap=round,line join=round,>=triangle 45,x=1cm,y=1cm,scale = 1.7]
\clip(-1.5,-0.35) rectangle (2.5,3.5);
\draw [line width=1pt] (0,0)-- (1,0);
\draw [line width=1pt] (1,0)-- (1.809016994374947,0.5877852522924729);
\draw [line width=1pt] (1.809016994374947,0.5877852522924729)-- (2.118033988749895,1.5388417685876261);
\draw [line width=1pt] (2.118033988749895,1.5388417685876261)-- (1.8090169943749475,2.4898982848827793);
\draw [line width=1pt] (1.8090169943749475,2.4898982848827793)-- (1,3.0776835371752527);
\draw [line width=1pt] (1,3.0776835371752527)-- (0,3.077683537175253);
\draw [line width=1pt] (0,3.077683537175253)-- (-0.809016994374947,2.4898982848827798);
\draw [line width=1pt] (-0.809016994374947,2.4898982848827798)-- (-1.1180339887498945,1.5388417685876268);
\draw [line width=1pt] (-1.1180339887498945,1.5388417685876268)-- (-0.8090169943749475,0.5877852522924734);
\draw [line width=1pt] (-0.8090169943749475,0.5877852522924734)-- (0,0);
\draw [line width=1pt,color=ccqqqq] (0,0)-- (1.4045084971874735,0.29389262614623646);
\draw [line width=1pt,color=ccqqqq] (-0.4045084971874733,2.7837909110290164)-- (1,3.0776835371752527);
\draw [color=ccqqqq](0.5993081196669593,0.5) node[anchor=north west] {$\Delta$};
\begin{scriptsize}
\draw [color=uuuuuu] (0,0)-- ++(-2.5pt,-2.5pt) -- ++(5pt,5pt) ++(-5pt,0) -- ++(5pt,-5pt);
\draw [fill=black] (1,0) circle (2pt);
\draw [color=uuuuuu] (1.809016994374947,0.5877852522924729)-- ++(-2.5pt,-2.5pt) -- ++(5pt,5pt) ++(-5pt,0) -- ++(5pt,-5pt);
\draw [fill=black] (2.118033988749895,1.5388417685876261) circle (2pt);
\draw [color=uuuuuu] (1.8090169943749475,2.4898982848827793)-- ++(-2.5pt,-2.5pt) -- ++(5pt,5pt) ++(-5pt,0) -- ++(5pt,-5pt);
\draw [fill=black] (1,3.0776835371752527) circle (2pt);
\draw [color=uuuuuu] (0,3.077683537175253)-- ++(-2.5pt,-2.5pt) -- ++(5pt,5pt) ++(-5pt,0) -- ++(5pt,-5pt);
\draw [fill=black] (-0.809016994374947,2.4898982848827798) circle (2pt);
\draw [color=uuuuuu] (-1.1180339887498945,1.5388417685876268)-- ++(-2.5pt,-2.5pt) -- ++(5pt,5pt) ++(-5pt,0) -- ++(5pt,-5pt);
\draw [fill=black] (-0.8090169943749475,0.5877852522924734) circle (2pt);
\draw [fill=ccqqqq] (1.4045084971874735,0.29389262614623646) circle (2pt);
\draw [fill=ccqqqq] (-0.4045084971874733,2.7837909110290164) circle (2pt);
\end{scriptsize}
\end{tikzpicture}
\definecolor{ccqqqq}{rgb}{0.8,0,0}
\definecolor{qqwuqq}{rgb}{0,0.39215686274509803,0}
\definecolor{zzttqq}{rgb}{0.6,0.2,0}
\definecolor{uuuuuu}{rgb}{0.26666666666666666,0.26666666666666666,0.26666666666666666}
\definecolor{cxvqqq}{rgb}{0.7803921568627451,0.3137254901960784,0}
\begin{tikzpicture}[line cap=round,line join=round,>=triangle 45,x=1cm,y=1cm,scale = 1.7]
\clip(-1.5,-0.35) rectangle (2.5,3.5);
\fill[line width=1pt,color=zzttqq,fill=zzttqq,fill opacity=0.10000000149011612] (0,3.077683537175253) -- (1,3.0776835371752527) -- (1.8090169943749475,2.4898982848827793) -- (1,0) -- (0,0) -- (-0.8090169943749475,0.5877852522924734) -- cycle;
\draw [line width=1pt] (0,0)-- (1,0);
\draw [line width=1pt] (1,0)-- (1.809016994374947,0.5877852522924729);
\draw [line width=1pt] (1.809016994374947,0.5877852522924729)-- (2.118033988749895,1.5388417685876261);
\draw [line width=1pt] (2.118033988749895,1.5388417685876261)-- (1.8090169943749475,2.4898982848827793);
\draw [line width=1pt] (1.8090169943749475,2.4898982848827793)-- (1,3.0776835371752527);
\draw [line width=1pt] (1,3.0776835371752527)-- (0,3.077683537175253);
\draw [line width=1pt] (0,3.077683537175253)-- (-0.809016994374947,2.4898982848827798);
\draw [line width=1pt] (-0.809016994374947,2.4898982848827798)-- (-1.1180339887498945,1.5388417685876268);
\draw [line width=1pt] (-1.1180339887498945,1.5388417685876268)-- (-0.8090169943749475,0.5877852522924734);
\draw [line width=1pt] (-0.8090169943749475,0.5877852522924734)-- (0,0);
\draw [line width=1pt,dash pattern=on 3pt off 3pt,color=zzttqq] (1.8090169943749475,2.4898982848827793)-- (1,0);
\draw [line width=1pt,dash pattern=on 3pt off 3pt,color=zzttqq] (-0.8090169943749475,0.5877852522924734)-- (0,3.077683537175253);
\draw [line width=1pt,color=qqwuqq] (0,0)-- (0.7376263079770541,3.077683537175253);
\draw [line width=1pt,color=qqwuqq] (0.26298784023064553,0)-- (1,3.0776835371752527);
\draw [line width=1pt,color=qqwuqq] (0.7376263079770539,0)-- (1.4047700305072512,2.7836008959497085);
\draw [line width=1pt,color=qqwuqq] (-0.4042469638676962,0.2937026110669292)-- (0.2629878402306459,3.077683537175253);
\draw [color=ccqqqq](1.0524739036123054,3.3380649073760162) node[anchor=north west] {$A'$};
\draw [color=ccqqqq](-0.18866494745482557,-0.0690543816145054) node[anchor=north west] {$A$};
\begin{scriptsize}
\draw [color=cxvqqq] (0,0)-- ++(-3pt,-3pt) -- ++(6pt,6pt) ++(-6pt,0) -- ++(6pt,-6pt);
\draw [fill=black] (1,0) circle (2pt);
\draw [color=uuuuuu] (1.809016994374947,0.5877852522924729)-- ++(-2.5pt,-2.5pt) -- ++(5pt,5pt) ++(-5pt,0) -- ++(5pt,-5pt);
\draw [fill=black] (2.118033988749895,1.5388417685876261) circle (2pt);
\draw [color=uuuuuu] (1.8090169943749475,2.4898982848827793)-- ++(-2.5pt,-2.5pt) -- ++(5pt,5pt) ++(-5pt,0) -- ++(5pt,-5pt);
\draw [fill=cxvqqq] (1,3.0776835371752527) circle (2.5pt);
\draw [color=uuuuuu] (0,3.077683537175253)-- ++(-2.5pt,-2.5pt) -- ++(5pt,5pt) ++(-5pt,0) -- ++(5pt,-5pt);
\draw [fill=black] (-0.809016994374947,2.4898982848827798) circle (2pt);
\draw [color=uuuuuu] (-1.1180339887498945,1.5388417685876268)-- ++(-2.5pt,-2.5pt) -- ++(5pt,5pt) ++(-5pt,0) -- ++(5pt,-5pt);
\draw [fill=black] (-0.8090169943749475,0.5877852522924734) circle (2pt);
\draw [fill=qqwuqq] (0.7376263079770541,3.077683537175253) ++(-2.5pt,0 pt) -- ++(2.5pt,2.5pt)--++(2.5pt,-2.5pt)--++(-2.5pt,-2.5pt)--++(-2.5pt,2.5pt);
\draw [fill=qqwuqq] (0.7376263079770539,0) ++(-2pt,0 pt) -- ++(2pt,2pt)--++(2pt,-2pt)--++(-2pt,-2pt)--++(-2pt,2pt);
\draw [fill=qqwuqq] (1.4047700305072512,2.7836008959497085) ++(-2pt,0 pt) -- ++(2pt,2pt)--++(2pt,-2pt)--++(-2pt,-2pt)--++(-2pt,2pt);
\draw [fill=qqwuqq] (-0.4042469638676962,0.2937026110669292) ++(-2pt,0 pt) -- ++(2pt,2pt)--++(2pt,-2pt)--++(-2pt,-2pt)--++(-2pt,2pt);
\draw [fill=qqwuqq] (0.2629878402306459,3.077683537175253) ++(-2pt,0 pt) -- ++(2pt,2pt)--++(2pt,-2pt)--++(-2pt,-2pt)--++(-2pt,2pt);
\draw [fill=qqwuqq] (0.26298784023064553,0) ++(-2pt,0 pt) -- ++(2pt,2pt)--++(2pt,-2pt)--++(-2pt,-2pt)--++(-2pt,2pt);
\end{scriptsize}
\end{tikzpicture}
\caption{On the left, a saddle connection of type (2). Any saddle connection of type (2) can be obtained from $\Delta$ by a symmetry or a rotation. On the right, a saddle connection of type (3), that is staying strictly inside the big cylinder associated to its diagram. Its endpoints $A$ and $A'$ are opposite vertices of the decagon, hence representing different singularities.}
\label{fig:ex_type_2}
\end{figure}

With the above distinction into types, we show:
\begin{Prop}[Length estimates]\label{prop:study_lengths}
Let $\alpha$ be a saddle connection on the regular $n$-gon, where $n \geq 10$, is even. We assume that $l_0=1$, that is the $n$-gon has unit side. Then, according to the type of $\alpha$ as above, we have:
\begin{enumerate}
\item $l(\alpha) = 1$,
\item $l(\alpha) = \sqrt{5 + 4 \cos \left(\frac{2\pi}{n} \right)} > 2 \sqrt{2}$. 
\item $l(\alpha) \geq 2 \sqrt{2} n_{\alpha} + \varepsilon_1,$ where $\varepsilon_1 := \sqrt{5} + 1 -2\sqrt{2} \simeq 0.408\dots$, and with equality if and only if $n_{\alpha} = 1$ and $n = 4m+2 = 10$, that is $\alpha$ is a long diagonal of the decagon.
\item In all the other cases, we have
\begin{equation*}
l(\alpha) \geq \eta_0 n_{\alpha} + \varepsilon_0
\end{equation*}
where $\eta_0 := 2/(2+\sqrt{5}-2\sqrt{2}) = 2/(1+\varepsilon_1) \simeq 1,421\dots$ and $\varepsilon_0 := 2 \cos(\pi/10)-\eta_0 \simeq 0,481\dots$, and with equality if and only if $\alpha$ is a short diagonal of the decagon.
\end{enumerate}
\end{Prop}
We provide a proof of this Proposition in Section \ref{sec:study_lengths}. The constants $\varepsilon_1, \eta_0$ and $\varepsilon_0$ may seem intriguing at first sight, but the additional length will help us compensate for the possible additional singular intersections. There is actually a slight margin in the choice of constants, and the key relations that we will need are
\begin{itemize}
    \item $\eta_0 \geq \sqrt{2}$;
    \item $(1+\varepsilon_1 )\eta_0 \geq 2$ (here, it is $2$ by definition);
    \item $1+\varepsilon_0 \geq \sqrt{2}$;
    \item and $3\sqrt{2}\varepsilon_0 + \varepsilon_0^2 \geq  2$.
\end{itemize}

\begin{Rema}
The length estimate $l(\alpha) \geq \eta_0 n_\alpha + \varepsilon_0$ does not hold for the saddle connection $\Delta$ (and therefore, any saddle connection of type (2)). This is the main reason why these saddle connections have to be considered separately.
\end{Rema}

Concerning the intersections, we show in Section \ref{sec:study_intersections}:
\begin{Prop}[Intersection estimates]\label{prop:study_intersections}
Given two saddle connections $\alpha$ and $\beta$ on the $(4m+2)$-gon, the following (symmetric) table gives an upper bound on the number $|\alpha \cap \beta|$ of non-singular intersection points between $\alpha$ and $\beta$ in terms of $n_{\alpha}, n_{\beta}$.
\begin{center}
\begin{tabular}{ |c|c||c|c|c|c| } 
\hline
& & \multicolumn{4}{|c|}{Type of $\alpha$} \\ 
\hline
& & (1) & (2) & (3) & (4)\\
 \hline
\multirow{4}{4em}{Type of $\beta$}& (1) & 0 & 1 & $n_{\alpha}-1$ & $n_{\alpha}-1$ \\ 

& (2) & $\star$ & $2$ & $2n_{\alpha}$ & $2n_{\alpha}-1$\\ 
& (3) & $\star$ & $\star$ & $n_{\alpha} n_{\beta}$ & $2n_{\alpha}n_{\beta} - 1$ \\ 
& (4) & $\star$ & $\star$ & $\star$ & $n_{\alpha} n_{\beta}$ \\
\hline
\end{tabular}
\end{center}
\end{Prop}

Finally, in Section \ref{sec:conclusion} we use Propositions \ref{prop:study_lengths} and \ref{prop:study_intersections} to analyze all possible cases of pairs of closed curves which are made of one closed saddle connection or two non-closed saddle connections.  

\subsection{Study of the lengths}\label{sec:study_lengths}
We now prove Proposition \ref{prop:study_lengths}. Since case 1. is trivial and case 2. is a direct computation, it remains to study cases 3. and 4.

\subsubsection*{Case 3.} Let us first assume that $\alpha$ is a saddle connection strictly contained inside a big cylinder. We start with the following lemma, whose proof is left to the reader.

\begin{Lem}[Diagonals of the decagon]\label{lem:length_diagonals}
Let $\varphi_{10} = 2 \cos \frac{\pi}{10}$. The lengths of the diagonals of a regular, unit-sided decagon are given by, in increasing order $\varphi_{10}, \varphi_{10}^2-1$, $\varphi_{10}^3 - 2 \varphi_{10}$ and $\varphi_{10}^4 - 3 \varphi_{10}^2 + 1$.\newline

Further
\[ \varphi_{10}^4 - 3 \varphi_{10}^2 + 1 = \sqrt{5}+1 = 2 \sqrt{2} + \varepsilon_1. \]
\end{Lem}

We can now estimate the length of saddle connections of type 3:
\begin{Lem}\label{lem:case3}
Assume that $\alpha$ is a saddle connection of type 3. Then
\begin{equation}\label{eq:length_big_cylinder}
 l(\alpha) \geq 2\sqrt{2} n_{\alpha} + \varepsilon_1
\end{equation}
with equality if and only if $\alpha$ is a long diagonal of the decagon.
\end{Lem}

\begin{proof}
Given a saddle connection $\alpha$ of type 3, we unfold $\alpha$ to obtain a chain of $n$-gons in which $\alpha$ is now a straight line, as in Figure \ref{fig:unfolding_big_cylinder}. We will assume as in the figure that the direction of the cylinder is vertical, and $\alpha$ starts from the bottom singularity $X$. Notice that, by definition, $\alpha$ must end at the top singularity $\bigcirc$. In particular, the length of $\alpha$ can be estimated as follows:
\begin{itemize}
\item The first segment accounts for a vertical length equal to the length of the longest diagonal $L_n$ of the regular $n$-gon, and we have:
\[ L_n \geq L_{10} = 2\sqrt{2} + \varepsilon_1\]
\item Each additional segment accounts for an additional vertical length equal to the length $L_n$ of the longest diagonal minus $\sin(\pi /n)$. Further,
\[ L_n - \sin \left(\frac{\pi}{n}\right) \geq L_{10} - \sin\left(\frac{\pi}{10}\right) \simeq 2.9336\dots > 2\sqrt{2}. \]
\end{itemize}

In particular, the vertical length of $\alpha$ is at least $2 \sqrt{2} n_{\alpha} + \varepsilon_1$ and therefore the same holds for the total length of $\alpha$. This proves Lemma \ref{lem:case3}.
\end{proof}
\begin{figure}[h]
\center
\definecolor{qqwuqq}{rgb}{0,0.39215686274509803,0}
\definecolor{zzttqq}{rgb}{0.6,0.2,0}
\definecolor{uuuuuu}{rgb}{0.26666666666666666,0.26666666666666666,0.26666666666666666}
\begin{tikzpicture}[line cap=round,line join=round,>=triangle 45,x=1cm,y=1cm]
\clip(-4,-1) rectangle (9,10);
\fill[line width=1pt,color=zzttqq,fill=zzttqq,fill opacity=0.10000000149011612] (-1,2.7527638409423463) -- (-1,0) -- (0,-0.3249196962329063) -- (1,0) -- (1,2.7527638409423463) -- (0,3.0776835371752527) -- cycle;
\fill[line width=1pt,color=zzttqq,fill=zzttqq,fill opacity=0.10000000149011612] (0,3.0776835371752527) -- (0,5.830447378117599) -- (-1,6.155367074350505) -- (-2,5.830447378117599) -- (-2,3.0776835371752527) -- (-1,2.7527638409423463) -- cycle;
\fill[line width=1pt,color=zzttqq,fill=zzttqq,fill opacity=0.10000000149011612] (0,5.830447378117599) -- (1,6.155367074350505) -- (2,5.830447378117599) -- (2,3.0776835371752527) -- (1,2.7527638409423463) -- (0,3.0776835371752527) -- cycle;
\fill[line width=1pt,color=zzttqq,fill=zzttqq,fill opacity=0.10000000149011612] (-1,6.155367074350505) -- (-2,5.830447378117599) -- (-3,6.155367074350505) -- (-3,8.908130915292851) -- (-2,9.233050611525758) -- (-1,8.908130915292851) -- (0,9.233050611525758) -- (1,8.908130915292851) -- (2,9.233050611525758) -- (3,8.908130915292851) -- (3,6.155367074350505) -- (2,5.830447378117599) -- (1,6.155367074350505) -- (0,5.830447378117599) -- cycle;
\draw [line width=1pt] (0,-0.3249196962329063)-- (1,0);
\draw [line width=1pt] (1,0)-- (1.6180339887498947,0.8506508083520394);
\draw [line width=1pt] (1.6180339887498947,0.8506508083520394)-- (1.6180339887498947,1.9021130325903064);
\draw [line width=1pt] (1.6180339887498947,1.9021130325903064)-- (1,2.7527638409423463);
\draw [line width=1pt] (1,2.7527638409423463)-- (0,3.0776835371752527);
\draw [line width=1pt] (0,3.0776835371752527)-- (-1,2.7527638409423463);
\draw [line width=1pt] (-1,2.7527638409423463)-- (-1.6180339887498942,1.902113032590307);
\draw [line width=1pt] (-1.6180339887498942,1.902113032590307)-- (-1.6180339887498942,0.8506508083520402);
\draw [line width=1pt] (-1.6180339887498942,0.8506508083520402)-- (-1,0);
\draw [line width=1pt] (-1,0)-- (0,-0.3249196962329063);
\draw [line width=1pt,dash pattern=on 3pt off 3pt,domain=-4.671695023164216:5.5] plot(\x,{(--3.0776835371752522-0*\x)/1});
\draw [line width=1pt,dash pattern=on 3pt off 3pt,color=zzttqq] (0,3.0776835371752527)-- (0,5.830447378117599);
\draw [line width=1pt,color=zzttqq] (0,5.830447378117599)-- (-1,6.155367074350505);
\draw [line width=1pt,color=zzttqq] (-1,6.155367074350505)-- (-2,5.830447378117599);
\draw [line width=1pt,color=zzttqq] (-2,5.830447378117599)-- (-2,3.0776835371752527);
\draw [line width=1pt,color=zzttqq] (-2,3.0776835371752527)-- (-1,2.7527638409423463);
\draw [line width=1pt,color=zzttqq] (-1,2.7527638409423463)-- (0,3.0776835371752527);
\draw [line width=1pt,color=zzttqq] (1,6.155367074350505)-- (2,5.830447378117599);
\draw [line width=1pt,color=zzttqq] (2,5.830447378117599)-- (2,3.0776835371752527);
\draw [line width=1pt,color=zzttqq] (2,3.0776835371752527)-- (1,2.7527638409423463);
\draw [line width=1pt,dash pattern=on 3pt off 3pt,domain=-4.671695023164216:5.5] plot(\x,{(--12.310734148701009-0*\x)/2});
\draw [line width=1pt,color=zzttqq] (-1,6.155367074350505)-- (-2,5.830447378117599);
\draw [line width=1pt,color=zzttqq] (-2,5.830447378117599)-- (-3,6.155367074350505);
\draw [line width=1pt,color=zzttqq] (-3,6.155367074350505)-- (-3,8.908130915292851);
\draw [line width=1pt,color=zzttqq] (-3,8.908130915292851)-- (-2,9.233050611525758);
\draw [line width=1pt,color=zzttqq] (-2,9.233050611525758)-- (-1,8.908130915292851);
\draw [line width=1pt,color=zzttqq] (-1,8.908130915292851)-- (0,9.233050611525758);
\draw [line width=1pt,color=zzttqq] (0,9.233050611525758)-- (1,8.908130915292851);
\draw [line width=1pt,color=zzttqq] (1,8.908130915292851)-- (2,9.233050611525758);
\draw [line width=1pt,color=zzttqq] (2,9.233050611525758)-- (3,8.908130915292851);
\draw [line width=1pt,color=zzttqq] (3,8.908130915292851)-- (3,6.155367074350505);
\draw [line width=1pt,color=zzttqq] (3,6.155367074350505)-- (2,5.830447378117599);
\draw [line width=1pt,color=zzttqq] (2,5.830447378117599)-- (1,6.155367074350505);
\draw [line width=1pt,color=zzttqq] (1,6.155367074350505)-- (0,5.830447378117599);
\draw [line width=1pt,color=zzttqq] (0,5.830447378117599)-- (-1,6.155367074350505);
\draw [line width=1pt,dash pattern=on 3pt off 3pt,domain=-4.671695023164216:5.5] plot(\x,{(--36.93220244610303-0*\x)/4});
\draw [line width=1pt,dash pattern=on 3pt off 3pt,color=zzttqq] (-1,6.155367074350505)-- (-1,8.908130915292851);
\draw [line width=1pt,dash pattern=on 3pt off 3pt,color=zzttqq] (1,8.908130915292851)-- (1,6.155367074350505);
\draw [line width=1pt,dash pattern=on 3pt off 3pt] (-4, -0.3249196962329063)--(5.5, -0.3249196962329063);
\draw [line width=1pt,to-to] (3.9949607643140217,3.0776835371752527)-- (4.036901292738877,-0.3249196962329063);
\draw [line width=1pt,to-to] (3.9949607643140217,3.0776835371752527)-- (4,6.155367074350505);
\draw [line width=1pt,to-to] (4,6.155367074350505)-- (3.9949607643140217,9.233050611525758);
\draw [line width=1pt,color=qqwuqq] (0,-0.3249196962329063)-- (1,6.155367074350505);
\draw [line width=1pt,color=qqwuqq] (0,-0.3249196962329063)-- (-1,6.155367074350505);
\draw [line width=1pt,color=qqwuqq] (0,-0.3249196962329063)-- (2,9.233050611525758);
\draw [line width=1pt,color=qqwuqq] (0,-0.3249196962329063)-- (-2,9.233050611525758);
\draw (4.3,1.9) node[anchor=north west] {$L_{10} = 2 \sqrt{2} + \varepsilon_1$};
\draw (4.3,4.9) node[anchor=north west] {$L_{10} - \sin \left(\frac{\pi}{10}\right) > 2 \sqrt{2}$};
\draw [line width=1pt,color=qqwuqq] (0,-0.3249196962329063)-- (0,3.0776835371752527);
\draw (4.3,8.1) node[anchor=north west] {$L_{10} - \sin \left(\frac{\pi}{10}\right) > 2 \sqrt{2}$};
\begin{scriptsize}
\draw [fill=black] (1,0) circle (2pt);
\draw [color=uuuuuu] (0,-0.3249196962329063)-- ++(-2pt,-2pt) -- ++(4pt,4pt) ++(-4pt,0) -- ++(4pt,-4pt);
\draw [color=uuuuuu] (1.6180339887498947,0.8506508083520394)-- ++(-2pt,-2pt) -- ++(4pt,4pt) ++(-4pt,0) -- ++(4pt,-4pt);
\draw [fill=uuuuuu] (1.6180339887498947,1.9021130325903064) circle (2pt);
\draw [color=uuuuuu] (1,2.7527638409423463)-- ++(-2pt,-2pt) -- ++(4pt,4pt) ++(-4pt,0) -- ++(4pt,-4pt);
\draw [fill=uuuuuu] (0,3.0776835371752527) circle (2pt);
\draw [color=uuuuuu] (-1,2.7527638409423463)-- ++(-2pt,-2pt) -- ++(4pt,4pt) ++(-4pt,0) -- ++(4pt,-4pt);
\draw [fill=uuuuuu] (-1.6180339887498942,1.902113032590307) circle (2pt);
\draw [color=uuuuuu] (-1.6180339887498942,0.8506508083520402)-- ++(-2pt,-2pt) -- ++(4pt,4pt) ++(-4pt,0) -- ++(4pt,-4pt);
\draw [fill=uuuuuu] (-1,0) circle (2pt);
\draw [color=uuuuuu] (0,5.830447378117599)-- ++(-2pt,-2pt) -- ++(4pt,4pt) ++(-4pt,0) -- ++(4pt,-4pt);
\draw [fill=uuuuuu] (-1,6.155367074350505) circle (2pt);
\draw [fill=uuuuuu] (1,6.155367074350505) circle (2pt);
\draw [fill=uuuuuu] (-2,3.0776835371752527) circle (2pt);
\draw [fill=uuuuuu] (2,3.0776835371752527) circle (2pt);
\draw [color=uuuuuu] (-2,5.830447378117599)-- ++(-2pt,-2pt) -- ++(4pt,4pt) ++(-4pt,0) -- ++(4pt,-4pt);
\draw [color=uuuuuu] (2,5.830447378117599)-- ++(-2pt,-2pt) -- ++(4pt,4pt) ++(-4pt,0) -- ++(4pt,-4pt);
\draw [fill=uuuuuu] (0,9.233050611525758) circle (2pt);
\draw [color=uuuuuu] (1,8.908130915292851)-- ++(-2pt,-2pt) -- ++(4pt,4pt) ++(-4pt,0) -- ++(4pt,-4pt);
\draw [color=uuuuuu] (-1,8.908130915292851)-- ++(-2pt,-2pt) -- ++(4pt,4pt) ++(-4pt,0) -- ++(4pt,-4pt);
\draw [fill=uuuuuu] (-2,9.233050611525758) circle (2pt);
\draw [fill=uuuuuu] (2,9.233050611525758) circle (2pt);
\draw [fill=uuuuuu] (-3,6.155367074350505) circle (2pt);
\draw [fill=uuuuuu] (3,6.155367074350505) circle (2pt);
\draw [color=uuuuuu] (-3,8.908130915292851)-- ++(-2pt,-2pt) -- ++(4pt,4pt) ++(-4pt,0) -- ++(4pt,-4pt);
\draw [color=uuuuuu] (3,8.908130915292851)-- ++(-2pt,-2pt) -- ++(4pt,4pt) ++(-4pt,0) -- ++(4pt,-4pt);
\end{scriptsize}
\end{tikzpicture}
\caption{The unfolding of the saddle connections staying inside the big cylinder with $n_{\alpha} \leq 3$ (and for $n = 10$). Here, $L_{10} = \sqrt{5} +1 = 2\sqrt{2} + \varepsilon_1$ is the length of the longest diagonal of the unit-sided regular decagon.}
\label{fig:unfolding_big_cylinder}
\end{figure}

\subsubsection*{Case 4.}
We now consider a saddle connection $\alpha$ of type 4. We recall that we want to show $l(\alpha) \geq \eta_0 n_\alpha + \varepsilon_0$, with equality if and only if $\alpha$ is a short diagonal and $n=10$.
The main idea is to distinguish \emph{short} and \emph{long} segments in the decomposition of $\alpha$. More precisely, one should notice that:

\begin{Lem}\label{lem:long_segments}
If $\alpha_i$ is a non-sandwiched segment whose endpoints lies on sides of label $\sigma_\alpha(j)$ for $j \geq 3$, then $l(\alpha_i) \geq 2 \cos \left(\frac{\pi}{n}\right) l_0$, with equality if and only if $\alpha = \alpha_i$ is a short diagonal of the decagon.
\end{Lem}
See the left part of Figure \ref{fig:short_long_segments}. In particular, such segments have length at least $ 2 \cos \pi / 10 = \eta_0+ \varepsilon_0$. Such segments will therefore be referred to as \emph{long non-adjacent segments} while the other segments will be called \emph{short segments} and are either sandwiched segments or non-sandwiched segments with an endpoint on the side $\sigma_\alpha(2)$. Geometrically speaking the short segments are those contained in a short cylinder in the direction of the side $\sigma_{\alpha}(1)$, see the right of Figure \ref{fig:short_long_segments} in the case where $\Sigma_\alpha = \Sigma_0$. This comes from the fact that the sides of label $\sigma_\alpha(1)$ are adjacent to both a side of label $\sigma_\alpha(2)$ and $\sigma_\alpha(3)$. In particular, although on their own short segments may have length as close as $l_0$ as possible, it turns out that considering a \emph{maximal trip through a short cylinder} (in the direction of $\sigma_\alpha(1)$) yields suitable length estimates:

\begin{Def}
A \emph{maximal trip through a short cylinder} is a sequence of segments $\alpha_{i_0}, \dots, \alpha_{i_0+p}$ of $\alpha$ such that:
\begin{itemize}
\item The segments $\alpha_{i_0 + 1}, \dots, \alpha_{i_0 + p-1}$ are either sandwiched segments or non-sandwiched segments with an endpoint on the side $\sigma_\alpha(2)$;
\item The segments $\alpha_{i_0}$ and $\alpha_{i_0+p}$ are either initial (resp. terminal) segments or have their endpoints on sides of label $\sigma_\alpha(3)$ and $\sigma_\alpha(4)$.
\end{itemize}
\end{Def}

\begin{figure}
\center
\definecolor{qqwuqq}{rgb}{0,0.39215686274509803,0}
\definecolor{uuuuuu}{rgb}{0.26666666666666666,0.26666666666666666,0.26666666666666666}
\begin{tikzpicture}[line cap=round,line join=round,>=triangle 45,x=1cm,y=1cm,scale=1.5]
\clip(-1.8,-0.35) rectangle (2.8,3.35);
\draw [shift={(0,0)},line width=1pt,color=qqwuqq,fill=qqwuqq,fill opacity=0.10000000149011612] (0,0) -- (0:0.5664025228310235) arc (0:18:0.5664025228310235) -- cycle;
\draw [line width=1pt] (0,0)-- (1,0);
\draw [line width=1pt] (1,0)-- (1.809016994374947,0.5877852522924729);
\draw [line width=1pt] (1.809016994374947,0.5877852522924729)-- (2.118033988749895,1.5388417685876261);
\draw [line width=1pt] (2.118033988749895,1.5388417685876261)-- (1.8090169943749475,2.4898982848827793);
\draw [line width=1pt] (1.8090169943749475,2.4898982848827793)-- (1,3.0776835371752527);
\draw [line width=1pt] (1,3.0776835371752527)-- (0,3.077683537175253);
\draw [line width=1pt] (0,3.077683537175253)-- (-0.809016994374947,2.4898982848827798);
\draw [line width=1pt] (-0.809016994374947,2.4898982848827798)-- (-1.1180339887498945,1.5388417685876268);
\draw [line width=1pt] (-1.1180339887498945,1.5388417685876268)-- (-0.8090169943749475,0.5877852522924734);
\draw [line width=1pt] (-0.8090169943749475,0.5877852522924734)-- (0,0);
\draw [line width=1pt,dash pattern=on 3pt off 3pt] (0,0)-- (1.809016994374947,0.5877852522924729);
\draw [line width=1pt,dash pattern=on 3pt off 3pt] (-0.809016994374947,2.4898982848827798)-- (1,3.0776835371752527);
\draw [line width=1pt,color=qqwuqq] (-0.3236683994496628,0.23515885717160706)-- (2.0326289836979248,1.2759921905468086);
\draw [color=qqwuqq](0.513467927240458,1.1089158875870346) node[anchor=north west] {$\alpha_i$};
\begin{scriptsize}
\draw (0.3,0) node[anchor=north west] {$\sigma_0(1)$};
\draw (1.3,0.35) node[anchor=north west] {$\sigma_0(2)$};
\draw (-1,0.35) node[anchor=north west] {$\sigma_0(3)$};
\draw (1.95,1.1) node[anchor=north west] {$\sigma_0(4)$};
\draw (-1.55,1.1) node[anchor=north west] {$\sigma_0(5)$};
\draw (1.95,2.25) node[anchor=north west] {$\sigma_0(5)$};
\draw (-1.55,2.25) node[anchor=north west] {$\sigma_0(4)$};
\draw (1.3,3.05) node[anchor=north west] {$\sigma_0(3)$};
\draw (-1,3.05) node[anchor=north west] {$\sigma_0(2)$};
\draw (0.2,3.4) node[anchor=north west] {$\sigma_0(1)$};
\draw (0.55,0.25) node[anchor=north west] {$\Sigma_0$};
\draw [fill=uuuuuu] (0,0) circle (2pt);
\draw [color=black] (1,0)-- ++(-2pt,-2pt) -- ++(4pt,4pt) ++(-4pt,0) -- ++(4pt,-4pt);
\draw [fill=uuuuuu] (1.809016994374947,0.5877852522924729) circle (2pt);
\draw [color=uuuuuu] (2.118033988749895,1.5388417685876261)-- ++(-2pt,-2pt) -- ++(4pt,4pt) ++(-4pt,0) -- ++(4pt,-4pt);
\draw [fill=uuuuuu] (1.8090169943749475,2.4898982848827793) circle (2pt);
\draw [color=uuuuuu] (1,3.0776835371752527)-- ++(-2pt,-2pt) -- ++(4pt,4pt) ++(-4pt,0) -- ++(4pt,-4pt);
\draw [fill=uuuuuu] (0,3.077683537175253) circle (2pt);
\draw [color=uuuuuu] (-0.809016994374947,2.4898982848827798)-- ++(-2pt,-2pt) -- ++(4pt,4pt) ++(-4pt,0) -- ++(4pt,-4pt);
\draw [fill=uuuuuu] (-1.1180339887498945,1.5388417685876268) circle (2pt);
\draw [color=uuuuuu] (-0.8090169943749475,0.5877852522924734)-- ++(-2pt,-2pt) -- ++(4pt,4pt) ++(-4pt,0) -- ++(4pt,-4pt);
\draw [fill=qqwuqq,shift={(-0.3236683994496628,0.23515885717160706)}] (0,0) ++(0 pt,3.75pt) -- ++(3.2475952641916446pt,-5.625pt)--++(-6.495190528383289pt,0 pt) -- ++(3.2475952641916446pt,5.625pt);
\draw [fill=qqwuqq] (2.0326289836979248,1.2759921905468086) ++(-2.5pt,0 pt) -- ++(2.5pt,2.5pt)--++(2.5pt,-2.5pt)--++(-2.5pt,-2.5pt)--++(-2.5pt,2.5pt);
\end{scriptsize}
\end{tikzpicture}
\definecolor{qqwuqq}{rgb}{0,0.39215686274509803,0}
\definecolor{zzttqq}{rgb}{0.6,0.2,0}
\definecolor{uuuuuu}{rgb}{0.26666666666666666,0.26666666666666666,0.26666666666666666}
\begin{tikzpicture}[line cap=round,line join=round,>=triangle 45,x=1cm,y=1cm,scale=1.5]
\clip(-1.8,-0.35) rectangle (2.8,3.35);
\fill[line width=1pt,color=zzttqq,fill=zzttqq,fill opacity=0.10000000149011612] (1.8090169943749475,2.4898982848827793) -- (-0.809016994374947,2.4898982848827798) -- (0,3.077683537175253) -- (1,3.0776835371752527) -- cycle;
\fill[line width=1pt,color=zzttqq,fill=zzttqq,fill opacity=0.10000000149011612] (-0.8090169943749475,0.5877852522924734) -- (1.809016994374947,0.5877852522924729) -- (1,0) -- (0,0) -- cycle;
\draw [line width=1pt] (0,0)-- (1,0);
\draw [line width=1pt] (1,0)-- (1.809016994374947,0.5877852522924729);
\draw [line width=1pt] (1.809016994374947,0.5877852522924729)-- (2.118033988749895,1.5388417685876261);
\draw [line width=1pt] (2.118033988749895,1.5388417685876261)-- (1.8090169943749475,2.4898982848827793);
\draw [line width=1pt] (1.8090169943749475,2.4898982848827793)-- (1,3.0776835371752527);
\draw [line width=1pt] (1,3.0776835371752527)-- (0,3.077683537175253);
\draw [line width=1pt] (0,3.077683537175253)-- (-0.809016994374947,2.4898982848827798);
\draw [line width=1pt] (-0.809016994374947,2.4898982848827798)-- (-1.1180339887498945,1.5388417685876268);
\draw [line width=1pt] (-1.1180339887498945,1.5388417685876268)-- (-0.8090169943749475,0.5877852522924734);
\draw [line width=1pt] (-0.8090169943749475,0.5877852522924734)-- (0,0);
\draw [line width=1pt,dash pattern=on 3pt off 3pt] (-0.809016994374947,2.4898982848827798)-- (1.8090169943749475,2.4898982848827793);
\draw [line width=1pt,dash pattern=on 3pt off 3pt] (-0.8090169943749475,0.5877852522924734)-- (1.809016994374947,0.5877852522924729);
\draw [line width=1pt,color=qqwuqq] (-0.988364539638482,1.9379232973924037)-- (1.7059850249516277,2.56475539241297);
\draw [line width=1pt,color=qqwuqq] (-0.10303196942331971,0.0748571075301907)-- (1.6711419803071608,0.4876131910228888);
\draw [line width=1pt,color=qqwuqq] (-0.13787501406778624,2.9775114759056684)-- (0.29270049699781087,3.077683537175253);
\draw [line width=1pt,color=qqwuqq] (0.2927004969978105,0)-- (1.3331696731874536,0.24206193661233252);
\draw [line width=1pt,color=qqwuqq] (-0.4758473211874934,2.7319602214951124)-- (1,3.0776835371752527);
\draw [color=qqwuqq](0.05702767380216179,2.23) node[anchor=north west] {$\alpha_{i_0}$};
\draw [color=qqwuqq](0,0.55) node[anchor=north west] {$\alpha_{i_0+1}$};
\draw [color=qqwuqq](0,2.9) node[anchor=north west] {$\alpha_{i_0+p-1}$};
\draw (1,3.4) node[anchor=north west] {$A'$};
\begin{scriptsize}
\draw (0.3,0) node[anchor=north west] {$\sigma_0(1)$};
\draw (1.45,0.35) node[anchor=north west] {$\sigma_0(2)$};
\draw (-1,0.35) node[anchor=north west] {$\sigma_0(3)$};
\draw (1.95,1.1) node[anchor=north west] {$\sigma_0(4)$};
\draw (-1.55,1.1) node[anchor=north west] {$\sigma_0(5)$};
\draw (1.95,2.25) node[anchor=north west] {$\sigma_0(5)$};
\draw (-1.55,2.25) node[anchor=north west] {$\sigma_0(4)$};
\draw (1.3,3.05) node[anchor=north west] {$\sigma_0(3)$};
\draw (-1,3.05) node[anchor=north west] {$\sigma_0(2)$};
\draw (0.2,3.4) node[anchor=north west] {$\sigma_0(1)$};
\draw [fill=uuuuuu] (0,0) circle (2pt);
\draw [color=black] (1,0)-- ++(-2pt,-2pt) -- ++(4pt,4pt) ++(-4pt,0) -- ++(4pt,-4pt);
\draw [fill=uuuuuu] (1.809016994374947,0.5877852522924729) circle (2pt);
\draw [color=uuuuuu] (2.118033988749895,1.5388417685876261)-- ++(-2pt,-2pt) -- ++(4pt,4pt) ++(-4pt,0) -- ++(4pt,-4pt);
\draw [fill=uuuuuu] (1.8090169943749475,2.4898982848827793) circle (2pt);
\draw [color=uuuuuu] (1,3.0776835371752527)-- ++(-2pt,-2pt) -- ++(4pt,4pt) ++(-4pt,0) -- ++(4pt,-4pt);
\draw [fill=uuuuuu] (0,3.077683537175253) circle (2pt);
\draw [color=uuuuuu] (-0.809016994374947,2.4898982848827798)-- ++(-2pt,-2pt) -- ++(4pt,4pt) ++(-4pt,0) -- ++(4pt,-4pt);
\draw [fill=uuuuuu] (-1.1180339887498945,1.5388417685876268) circle (2pt);
\draw [color=uuuuuu] (-0.8090169943749475,0.5877852522924734)-- ++(-2pt,-2pt) -- ++(4pt,4pt) ++(-4pt,0) -- ++(4pt,-4pt);
\draw [fill=qqwuqq] (-0.988364539638482,1.9379232973924037) circle (1pt);
\draw [fill=qqwuqq] (1.7059850249516277,2.56475539241297) circle (1pt);
\draw [fill=qqwuqq] (-0.10303196942331971,0.0748571075301907) circle (1pt);
\draw [fill=qqwuqq] (1.6711419803071608,0.4876131910228888) circle (1pt);
\draw [fill=qqwuqq] (-0.13787501406778624,2.9775114759056684) circle (1pt);
\draw [fill=qqwuqq] (0.29270049699781087,3.077683537175253) circle (1pt);
\draw [fill=qqwuqq] (0.2927004969978105,0) circle (1pt);
\draw [fill=qqwuqq] (1.3331696731874536,0.24206193661233252) circle (1pt);
\draw [fill=qqwuqq] (-0.4758473211874934,2.7319602214951124) circle (1pt);
\end{scriptsize}
\end{tikzpicture}
\caption{On the left, a segment whose endpoints on sides of label $\sigma_0(i)$ for $i \geq 3$. This segment is longer than the short dashed diagonals. On the right, a maximal trip through the short horizontal cylinder. Here, the first segment $\alpha_{i_0}$ is a long segment whereas the last segment $\alpha_{i_0+p-1}$ ($p=5$) is the terminal segment and goes to the vertex $A'$.}
\label{fig:short_long_segments}
\end{figure}

In particular, the segments $\alpha_{i_0}$ and $\alpha_{i_0 + p}$ may be long non-adjacent segments. However, we will count them with the sequence of short segments as this will help us to obtain suitable length estimates for a maximal trip through the short cylinder. Let us point out that maximal trips through small cylinders cannot overlap, since the first (resp. last) segment is either initial (resp. terminal) or has an endpoint in the interior of a side of label $\sigma_\alpha(4)$, which means that the preceding (resp. following) segment is also a long segment. With this definition, we have:

\begin{Lem}\label{lem:length_consecutive_short}
A maximal trip through the short horizontal cylinder made of $p$ segments, and among them $q$ sandwiched segments, has length at least
\begin{equation}\label{eq:second_shortest}
\sqrt{\left(p+(p-1)\cos\left(\frac{2\pi}{n}\right)\right)^2 + \left((q+1) \sin\left(\frac{2\pi}{n}\right)\right)^2}
\end{equation}
Further, unless the saddle connection $\alpha$ is made of a single maximal trip through the short horizontal cylinder and does not contain any long segments (that is, the endpoints of $\alpha$ are the points $A$ and $A'$ of Figure \ref{fig:second_shortest_cylinder}), we can slightly improve the estimate to:
\begin{equation}\label{eq:second_shortest_2}
l(\alpha) \geq \sqrt{\left(p+p\cos\left(\frac{2\pi}{n}\right)\right)^2 + \left((q+1) \sin\left(\frac{2\pi}{n}\right)\right)^2}
\end{equation}
\end{Lem}

\begin{proof}
Let $\alpha$ be a saddle connection whose direction lies in $\Sigma_0$, and let $\alpha_{i_0} \cup \cdots \cup \alpha_{i_0+p-1}$ be a maximal trip through the short horizontal cylinder made of $p$ segments, and among them $q$ sandwiched segments. In order to estimate its length, we proceed similarly to case (3) by unfolding the trajectory, as in Figure~\ref{fig:ex_unfolding_short_cylinder}. Now, we provide a lower bound on the length of $\alpha_{i_0} \cup \cdots \cup \alpha_{i_0+p-1}$ using a segment-by-segment estimate. More precisely, the horizontal (resp. vertical) length of each segment $\alpha_j$ is counted from the next singularity on the right (resp. above) the left endpoint $\alpha_j^-$ of the segment $\alpha_j$ up to the next singularity on the right (resp. above) the right endpoint $\alpha_j^+$ of $\alpha_j$, except for the last segment which we will count up to the last singularity (unless it ends to a singularity). The reason for this count is that then we only have to estimate distances between singularities. In particular, we obtain the following estimates:

\begin{figure}[h]
\center
\definecolor{qqwuqq}{rgb}{0,0.39215686274509803,0}
\definecolor{zzttqq}{rgb}{0.6,0.2,0}
\definecolor{uuuuuu}{rgb}{0.26666666666666666,0.26666666666666666,0.26666666666666666}
\begin{tikzpicture}[line cap=round,line join=round,>=triangle 45,x=1cm,y=1cm]
\clip(-2.15,-0.75) rectangle (13,2);
\fill[line width=1pt,color=zzttqq,fill=zzttqq,fill opacity=0.10000000149011612] (-0.8090169943749475,0.5877852522924734) -- (1.809016994374947,0.5877852522924729) -- (1,0) -- (0,0) -- cycle;
\fill[line width=1pt,color=zzttqq,fill=zzttqq,fill opacity=0.10000000149011612] (1,0) -- (1.809016994374947,0.5877852522924729) -- (2.8090169943749466,0.5877852522924729) -- (3.618033988749893,0) -- cycle;
\fill[line width=1pt,color=zzttqq,fill=zzttqq,fill opacity=0.10000000149011612] (2.8090169943749466,0.5877852522924729) -- (3.618033988749893,0) -- (4.618033988749891,0) -- (5.427050983124839,0.5877852522924714) -- cycle;
\fill[line width=1pt,color=zzttqq,fill=zzttqq,fill opacity=0.10000000149011612] (4.618033988749893,1.1755705045849454) -- (5.427050983124839,0.5877852522924714) -- (6.427050983124838,0.5877852522924699) -- (7.236067977499785,1.1755705045849414) -- cycle;
\fill[line width=1pt,color=zzttqq,fill=zzttqq,fill opacity=0.10000000149011612] (4.618033988749891,0) -- (5.427050983124839,0.5877852522924714) -- (6.427050983124838,0.5877852522924699) -- (7.236067977499784,0) -- cycle;
\fill[line width=1pt,color=zzttqq,fill=zzttqq,fill opacity=0.10000000149011612] (6.427050983124838,0.5877852522924699) -- (7.236067977499785,1.1755705045849414) -- (8.236067977499784,1.1755705045849396) -- (9.04508497187473,0.5877852522924654) -- cycle;
\fill[line width=1pt,color=zzttqq,fill=zzttqq,fill opacity=0.10000000149011612] (8.236067977499784,1.1755705045849396) -- (9.04508497187473,0.5877852522924654) -- (10.045084971874727,0.5877852522924623) -- (10.854101966249676,1.175570504584932) -- cycle;
\fill[line width=1pt,color=zzttqq,fill=zzttqq,fill opacity=0.1] (10.854101966249676,1.175570504584932) -- (11.854101966249674,1.1755705045849278) -- (12.663118960624619,0.5877852522924523) -- (10.045084971874727,0.5877852522924623) -- cycle;
\draw [line width=1pt] (-0.8090169943749475,0.5877852522924734)-- (0,0);
\draw [line width=1pt] (0,0)-- (1,0);
\draw [line width=1pt] (1,0)-- (1.809016994374947,0.5877852522924729);
\draw [line width=1pt,to-to] (1.76,-0.3)-- (3.56,-0.3);
\draw [line width=1pt,color=qqwuqq] (0,0)-- (11.854101966249674,1.1755705045849278);
\draw [line width=1pt] (1,0)-- (1.809016994374947,0.5877852522924729);
\draw [line width=1pt] (1.809016994374947,0.5877852522924729)-- (2.8090169943749466,0.5877852522924729);
\draw [line width=1pt] (2.8090169943749466,0.5877852522924729)-- (3.618033988749893,0);
\draw [line width=1pt] (2.8090169943749466,0.5877852522924729)-- (3.618033988749893,0);
\draw [line width=1pt] (3.618033988749893,0)-- (4.618033988749891,0);
\draw [line width=1pt] (4.618033988749891,0)-- (5.427050983124839,0.5877852522924714);
\draw [line width=1pt,dash pattern=on 3pt off 3pt] (4.618033988749893,1.1755705045849454)-- (5.427050983124839,0.5877852522924714);
\draw [line width=1pt] (5.427050983124839,0.5877852522924714)-- (6.427050983124838,0.5877852522924699);
\draw [line width=1pt] (6.427050983124838,0.5877852522924699)-- (7.236067977499785,1.1755705045849414);
\draw [line width=1pt] (4.618033988749891,0)-- (5.427050983124839,0.5877852522924714);
\draw [line width=1pt] (5.427050983124839,0.5877852522924714)-- (6.427050983124838,0.5877852522924699);
\draw [line width=1pt,dash pattern=on 3pt off 3pt] (6.427050983124838,0.5877852522924699)-- (7.236067977499784,0);
\draw [line width=1pt] (6.427050983124838,0.5877852522924699)-- (7.236067977499785,1.1755705045849414);
\draw [line width=1pt] (7.236067977499785,1.1755705045849414)-- (8.236067977499784,1.1755705045849396);
\draw [line width=1pt] (8.236067977499784,1.1755705045849396)-- (9.04508497187473,0.5877852522924654);
\draw [line width=1pt] (8.236067977499784,1.1755705045849396)-- (9.04508497187473,0.5877852522924654);
\draw [line width=1pt] (9.04508497187473,0.5877852522924654)-- (10.045084971874727,0.5877852522924623);
\draw [line width=1pt] (10.045084971874727,0.5877852522924623)-- (10.854101966249676,1.175570504584932);
\draw [line width=1pt] (10.854101966249676,1.175570504584932)-- (11.854101966249674,1.1755705045849278);
\draw [line width=1pt] (11.854101966249674,1.1755705045849278)-- (12.663118960624619,0.5877852522924523);
\draw [line width=1pt] (10.045084971874727,0.5877852522924623)-- (10.854101966249676,1.175570504584932);
\draw [line width=1pt,to-to] (3.59,-0.3)-- (5.39,-0.3);
\draw [line width=1pt,to-to] (0,-0.3)-- (1.73,-0.3);
\draw [line width=1pt,to-to] (5.43,1.4)-- (7.25,1.4);
\draw [line width=1pt,to-to] (7.25,1.4)-- (9.05,1.4);
\draw [line width=1pt,to-to] (9.05,1.4)-- (10.85,1.4);
\draw [line width=1pt,to-to] (4.35,1.18)-- (4.35,0.6);
\draw [line width=1pt,to-to] (10.85,1.4)-- (11.85,1.4);
\draw [line width=1pt,to-to] (-1,0.6)-- (-1,0);
\draw (11.1,2.05) node[anchor=north west] {$1$};
\begin{scriptsize}
\draw (0,-0.25) node[anchor=north west] {$1+\cos\left(\frac{\pi}{5}\right)$};
\draw (1.8,-0.25) node[anchor=north west] {$1+\cos\left(\frac{\pi}{5}\right)$};
\draw (3.7,-0.25) node[anchor=north west] {$1+\cos\left(\frac{\pi}{5}\right)$};
\draw (5.5,2.05) node[anchor=north west] {$1+\cos\left(\frac{\pi}{5}\right)$};
\draw (7.3,2.05) node[anchor=north west] {$1+\cos\left(\frac{\pi}{5}\right)$};
\draw (9.1,2.05) node[anchor=north west] {$1+\cos\left(\frac{\pi}{5}\right)$};
\draw (-2.2,0.55) node[anchor=north west] {$\sin \left(\frac{\pi}{5}\right)$};
\draw (3.2,1.15) node[anchor=north west] {$\sin \left(\frac{\pi}{5}\right)$};
\draw [color=black] (0,0)-- ++(-2.5pt,-2.5pt) -- ++(5pt,5pt) ++(-5pt,0) -- ++(5pt,-5pt);
\draw [fill=black] (1,0) circle (2.5pt);
\draw [color=uuuuuu] (1.809016994374947,0.5877852522924729)-- ++(-2.5pt,-2.5pt) -- ++(5pt,5pt) ++(-5pt,0) -- ++(5pt,-5pt);
\draw [fill=black] (-0.8090169943749475,0.5877852522924734) circle (2.5pt);
\draw [color=uuuuuu] (3.618033988749893,0)-- ++(-2.5pt,-2.5pt) -- ++(5pt,5pt) ++(-5pt,0) -- ++(5pt,-5pt);
\draw [fill=uuuuuu] (2.8090169943749466,0.5877852522924729) circle (2.5pt);
\draw [fill=uuuuuu] (4.618033988749891,0) circle (2.5pt);
\draw [color=uuuuuu] (5.427050983124839,0.5877852522924714)-- ++(-2.5pt,-2.5pt) -- ++(5pt,5pt) ++(-5pt,0) -- ++(5pt,-5pt);
\draw [color=uuuuuu] (7.236067977499784,0)-- ++(-2.5pt,-2.5pt) -- ++(5pt,5pt) ++(-5pt,0) -- ++(5pt,-5pt);
\draw [fill=uuuuuu] (6.427050983124838,0.5877852522924699) circle (2.5pt);
\draw [color=uuuuuu] (7.236067977499785,1.1755705045849414)-- ++(-2.5pt,-2.5pt) -- ++(5pt,5pt) ++(-5pt,0) -- ++(5pt,-5pt);
\draw [fill=uuuuuu] (4.618033988749893,1.1755705045849454) circle (2.5pt);
\draw [color=uuuuuu] (9.04508497187473,0.5877852522924654)-- ++(-2.5pt,-2.5pt) -- ++(5pt,5pt) ++(-5pt,0) -- ++(5pt,-5pt);
\draw [fill=uuuuuu] (8.236067977499784,1.1755705045849396) circle (2.5pt);
\draw [fill=uuuuuu] (10.045084971874727,0.5877852522924623) circle (2.5pt);
\draw [color=uuuuuu] (10.854101966249676,1.175570504584932)-- ++(-2.5pt,-2.5pt) -- ++(5pt,5pt) ++(-5pt,0) -- ++(5pt,-5pt);
\draw [color=uuuuuu] (12.663118960624619,0.5877852522924523)-- ++(-2.5pt,-2.5pt) -- ++(5pt,5pt) ++(-5pt,0) -- ++(5pt,-5pt);
\draw [fill=uuuuuu] (11.854101966249674,1.1755705045849278) circle (2.5pt);
\end{scriptsize}
\end{tikzpicture}
\caption{Example of unfolding of a maximal trip in the shortest horizontal cylinder (here the saddle connection $\alpha$ starts at $A$ and ends at $A'$, and is completely contained in the short horizontal cylinder). Each segment (sandwiched or non-sandwiched) accounts for a horizontal length at least $1 + \cos(\frac{\pi}{5})$, except the last one which may account for only $1$. Further, the vertical length is at least $\sin \left(\frac{\pi}{5}\right)$, and each sandwiched segment accounts for an additional vertical length of $\sin(\frac{\pi}{5})$.}
\label{fig:ex_unfolding_short_cylinder}
\end{figure}
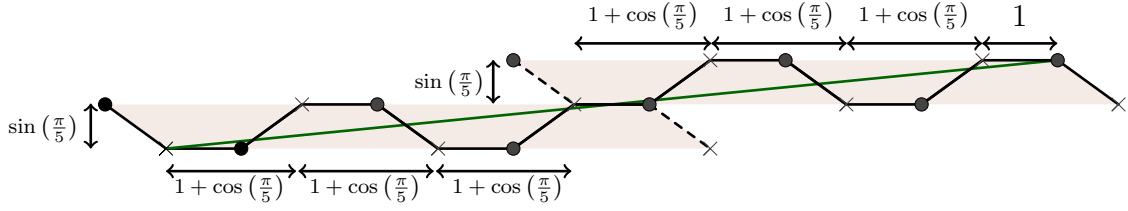

\begin{itemize}
\item Since $\theta_\alpha \in \Sigma_0$, the initial endpoint $\alpha_{i_0}^-$ of the first segment of the sequence $\alpha_{i_0}$ is either the singularity $A$ or it lies on the segment $CB$ (possibly on $B$), see Figure~\ref{fig:second_shortest_cylinder} and the top-left of Figure \ref{fig:cases_additional_length}. In particular, if $\alpha_{i_0}^- = A$ then the segment $\alpha_{i_0}$ accounts for a horizontal length $1+\cos\left(\frac{2\pi}{n}\right)$, and otherwise it accounts for a horizontal length of at least $1 + 2\cos\left(\frac{2\pi}{n}\right)$. Further, this segment also accounts for a vertical length of $\sin\left(\frac{2\pi}{n}\right)$.

\begin{figure}[h]
\center
\definecolor{ccqqqq}{rgb}{0.8,0,0}
\definecolor{zzttqq}{rgb}{0.6,0.2,0}
\definecolor{uuuuuu}{rgb}{0.26666666666666666,0.26666666666666666,0.26666666666666666}
\definecolor{qqwuqq}{rgb}{0,0.39215686274509803,0}
\begin{tikzpicture}[line cap=round,line join=round,>=triangle 45,x=1cm,y=1cm,scale=1.5]
\clip(-1.7845282725827565,-0.4232152711615752) rectangle (5.406398970426983,3.294928211444511);
\fill[line width=1pt,color=zzttqq,fill=zzttqq,fill opacity=0.10000000149011612] (-0.8090169943749475,0.5877852522924734) -- (1.809016994374947,0.5877852522924729) -- (1,0) -- (0,0) -- cycle;
\fill[line width=1pt,color=zzttqq,fill=zzttqq,fill opacity=0.10000000149011612] (-0.809016994374947,2.4898982848827798) -- (0,3.077683537175253) -- (1,3.0776835371752527) -- (1.8090169943749475,2.4898982848827793) -- cycle;
\draw [shift={(4,1)},line width=1pt,color=qqwuqq,fill=qqwuqq,fill opacity=0.10000000149011612] (0,0) -- (0:0.4718456196200616) arc (0:18:0.4718456196200616) -- cycle;
\draw [line width=1pt] (0,0)-- (1,0);
\draw [line width=1pt] (1,0)-- (1.809016994374947,0.5877852522924729);
\draw [line width=1pt] (1.809016994374947,0.5877852522924729)-- (2.118033988749895,1.5388417685876261);
\draw [line width=1pt] (2.118033988749895,1.5388417685876261)-- (1.8090169943749475,2.4898982848827793);
\draw [line width=1pt] (1.8090169943749475,2.4898982848827793)-- (1,3.0776835371752527);
\draw [line width=1pt] (1,3.0776835371752527)-- (0,3.077683537175253);
\draw [line width=1pt] (0,3.077683537175253)-- (-0.809016994374947,2.4898982848827798);
\draw [line width=1pt] (-0.809016994374947,2.4898982848827798)-- (-1.1180339887498945,1.5388417685876268);
\draw [line width=1pt] (-1.1180339887498945,1.5388417685876268)-- (-0.8090169943749475,0.5877852522924734);
\draw [line width=1pt] (-0.8090169943749475,0.5877852522924734)-- (0,0);
\draw [line width=1pt,dash pattern=on 3pt off 3pt,color=zzttqq] (-0.8090169943749475,0.5877852522924734)-- (1.809016994374947,0.5877852522924729);
\draw [line width=1pt,dash pattern=on 3pt off 3pt,color=zzttqq] (1.8090169943749475,2.4898982848827793)-- (-0.809016994374947,2.4898982848827798);
\draw [line width=1pt,dash pattern=on 3pt off 3pt] (5,1)-- (4,1);
\draw [line width=1pt,dash pattern=on 3pt off 3pt] (4,1)-- (3,1);
\draw [line width=1pt,dash pattern=on 3pt off 3pt] (4,1)-- (4.951056516295154,1.3090169943749475);
\draw [line width=1pt,dash pattern=on 3pt off 3pt] (4,1)-- (4.8090169943749475,1.5877852522924731);
\draw [line width=1pt,dash pattern=on 3pt off 3pt] (4,1)-- (4.587785252292473,1.8090169943749475);
\draw [line width=1pt,dash pattern=on 3pt off 3pt] (4,1)-- (4.3090169943749475,1.9510565162951536);
\draw [line width=1pt,dash pattern=on 3pt off 3pt] (4,1)-- (4,2);
\draw [line width=1pt,dash pattern=on 3pt off 3pt] (4,1)-- (3.6909830056250525,1.9510565162951536);
\draw [line width=1pt,dash pattern=on 3pt off 3pt] (4,1)-- (3.4122147477075266,1.8090169943749475);
\draw [line width=1pt,dash pattern=on 3pt off 3pt] (4,1)-- (3.1909830056250525,1.5877852522924731);
\draw [line width=1pt,dash pattern=on 3pt off 3pt] (4,1)-- (3.0489434837048464,1.3090169943749475);
\draw [color=qqwuqq](4.962864087984125,1.35) node[anchor=north west] {$\Sigma_0$};
\draw (-0.4,0) node[anchor=north west] {$A$};
\draw (-1.5,1.9) node[anchor=north west] {$C$};
\draw (1.1,3.4) node[anchor=north west] {$A'$};
\draw (2.1,1.9) node[anchor=north west] {$C'$};
\draw (-1.06,2.92) node[anchor=north west] {$B$};
\draw(1.8,0.7) node[anchor=north west] {$B'$};
\draw [line width=1pt,dash pattern=on 3pt off 3pt,color=qqwuqq] (0,0)-- (1.36,0.27);
\draw [line width=1pt,dash pattern=on 3pt off 3pt,color=qqwuqq] (-0.95,2.1)-- (1.45,2.76);
\begin{scriptsize}
\draw [color=qqwuqq] (0,0)-- ++(-2.5pt,-2.5pt) -- ++(5pt,5pt) ++(-5pt,0) -- ++(5pt,-5pt);
\draw [fill=black] (1,0) circle (2.5pt);
\draw [color=uuuuuu] (1.809016994374947,0.5877852522924729)-- ++(-2.5pt,-2.5pt) -- ++(5pt,5pt) ++(-5pt,0) -- ++(5pt,-5pt);
\draw [fill=black] (2.118033988749895,1.5388417685876261) circle (2.5pt);
\draw [color=uuuuuu] (1.8090169943749475,2.4898982848827793)-- ++(-2.5pt,-2.5pt) -- ++(5pt,5pt) ++(-5pt,0) -- ++(5pt,-5pt);
\draw [fill=black] (1,3.0776835371752527) circle (2.5pt);
\draw [color=uuuuuu] (0,3.077683537175253)-- ++(-2.5pt,-2.5pt) -- ++(5pt,5pt) ++(-5pt,0) -- ++(5pt,-5pt);
\draw [fill=black] (-0.809016994374947,2.4898982848827798) circle (2.5pt);
\draw [fill=qqwuqq] (-0.94,2.1) circle (2.5pt);
\draw [color=uuuuuu] (-1.1180339887498945,1.5388417685876268)-- ++(-2.5pt,-2.5pt) -- ++(5pt,5pt) ++(-5pt,0) -- ++(5pt,-5pt);
\draw [fill=black] (-0.8090169943749475,0.5877852522924734) circle (2.5pt);
\end{scriptsize}
\end{tikzpicture}
\caption{The short horizontal cylinder, and the two possible combinatorics for the first segment $\alpha_{i_0}$ of a maximal trip through the short horizontal cylinder for a saddle connection $\alpha$ in sector $\Sigma_0$: $\alpha_{i_0}$ can either start at $A$, or on the segment $BC$ - possibly at $B$).}
\label{fig:second_shortest_cylinder}
\end{figure}
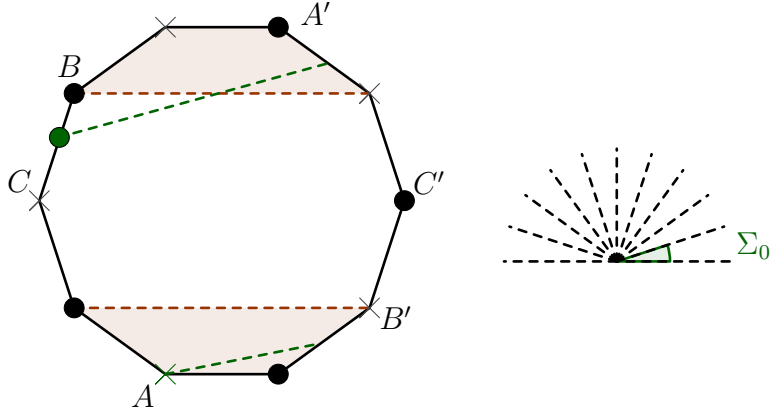

\item Next, an additional non-sandwiched segment accounts for a horizontal length $1+\cos\left(\frac{2\pi}{n}\right)$ while it does not add any vertical length, see the bottom-left of Figure~\ref{fig:cases_additional_length}.
\item An additional sandwiched segment adds a horizontal length $1+\cos\left(\frac{2\pi}{n}\right)$ but also adds a vertical length of $\sin\left(\frac{2\pi}{n}\right)$, see the top-right of Figure~\ref{fig:cases_additional_length}.
\item Finally, the right endpoint $\alpha_{i_0+p-1}^+$ of the last segment could be either $A'$ or lie on the segment $B'C'$ (possibly on $B'$). In the first case, we count the length up to $A'$, which adds a horizontal length $1$ (and no vertical length), whereas in the second case we count the length up to $B'$, and $\alpha_{i_0+p-1}$ accounts for a horizontal length $1+\cos\left(\frac{2\pi}{n}\right)$ (and no vertical length). See the bottom-right of Figure~\ref{fig:cases_additional_length}.
\end{itemize}

\begin{figure}[h]
\center
\definecolor{ccqqqq}{rgb}{0.8,0,0}
\definecolor{zzttqq}{rgb}{0.6,0.2,0}
\definecolor{uuuuuu}{rgb}{0.26666666666666666,0.26666666666666666,0.26666666666666666}
\definecolor{qqwuqq}{rgb}{0,0.39215686274509803,0}
\begin{tikzpicture}[line cap=round,line join=round,>=triangle 45,x=1cm,y=1cm, scale = 1.5]
\clip(-1.208876616646281,-0.4326521835539764) rectangle (2.8678695368710505,3.2854912990521097);
\fill[line width=1pt,color=zzttqq,fill=zzttqq,fill opacity=0.10000000149011612] (-0.8090169943749475,0.5877852522924734) -- (1.809016994374947,0.5877852522924729) -- (1,0) -- (0,0) -- cycle;
\fill[line width=1pt,color=zzttqq,fill=zzttqq,fill opacity=0.10000000149011612] (-0.809016994374947,2.4898982848827798) -- (0,3.077683537175253) -- (1,3.0776835371752527) -- (1.8090169943749475,2.4898982848827793) -- cycle;
\draw [line width=1pt,dash pattern=on 3pt off 3pt,color=zzttqq] (-0.8090169943749475,0.5877852522924734)-- (1.809016994374947,0.5877852522924729);
\draw [line width=1pt,dash pattern=on 3pt off 3pt,color=zzttqq] (1.8090169943749475,2.4898982848827793)-- (-0.809016994374947,2.4898982848827798);
\draw [color=qqwuqq](4.962864087984124,1.4075457329642642) node[anchor=north west] {$\Sigma_0$};
\draw (-0.4,0) node[anchor=north west] {$A$};
\draw (-1.05,1.85) node[anchor=north west] {$C$};
\draw (1.1,3.4) node[anchor=north west] {$A'$};
\draw (-1.06,2.92) node[anchor=north west] {$B$};
\draw [line width=1pt,dash pattern=on 3pt off 3pt] (1.809016994374947,0.5877852522924729)-- (2.118033988749895,1.5388417685876261);
\draw [line width=1pt,dash pattern=on 3pt off 3pt] (2.118033988749895,1.5388417685876261)-- (1.8090169943749475,2.4898982848827793);
\draw [line width=1pt,dash pattern=on 3pt off 3pt] (-0.809016994374947,2.4898982848827798)-- (-1.1180339887498945,1.5388417685876268);
\draw [line width=1pt,dash pattern=on 3pt off 3pt] (-1.1180339887498945,1.5388417685876268)-- (-0.8090169943749475,0.5877852522924734);
\draw [line width=1pt,dash pattern=on 3pt off 3pt,color=qqwuqq] (0,0)-- (1.3579635540868538,0.26567933348371486);
\draw [line width=1pt,dash pattern=on 3pt off 3pt,color=qqwuqq] (-0.809016994374947,2.4898982848827798)-- (1.3485266416944526,2.8325195042168505);
\draw [line width=1pt] (-0.809016994374947,2.4898982848827798)-- (0,3.077683537175253);
\draw [line width=1pt] (0,3.077683537175253)-- (1,3.0776835371752527);
\draw [line width=1pt] (1,3.0776835371752527)-- (1.8090169943749475,2.4898982848827793);
\draw [line width=1pt] (-0.8090169943749475,0.5877852522924734)-- (0,0);
\draw [line width=1pt] (0,0)-- (1,0);
\draw [line width=1pt] (1,0)-- (1.809016994374947,0.5877852522924729);
\draw [line width=1pt, to-to] (-0.0009518304189236203,0.8)-- (1.810935348922113,0.8);
\draw [line width=1pt, to-to] (-0.8502739457350345,2.2285571111031715)-- (1.8392460860993165,2.2285571111031715);
\draw [line width=1pt, to-to] (2,0.5676605300405544)-- (2,-0.045738775465525836);
\draw [line width=1pt, to-to] (2,3.0495684892420787)-- (2,2.502227570482807);
\draw (-0.2,2.256867848280375) node[anchor=north west] {$1 + 2 \cos(\frac{\pi}{5})$};
\draw (0.2,1.3) node[anchor=north west] {$1 + \cos(\frac{\pi}{5})$};
\draw (2,0.5) node[anchor=north west] {$\sin(\frac{\pi}{5})$};
\draw (2,3) node[anchor=north west] {$\sin(\frac{\pi}{5})$};
\begin{scriptsize}
\draw [fill=black] (2.118033988749895,1.5388417685876261) circle (2.5pt);
\draw [color=uuuuuu] (-1.1180339887498945,1.5388417685876268)-- ++(-2.5pt,-2.5pt) -- ++(5pt,5pt) ++(-5pt,0) -- ++(5pt,-5pt);
\draw [color=qqwuqq] (0,0)-- ++(-2.5pt,-2.5pt) -- ++(5pt,5pt) ++(-5pt,0) -- ++(5pt,-5pt);
\draw [fill=black] (1,0) circle (2.5pt);
\draw [color=uuuuuu] (1.809016994374947,0.5877852522924729)-- ++(-2.5pt,-2.5pt) -- ++(5pt,5pt) ++(-5pt,0) -- ++(5pt,-5pt);
\draw [color=uuuuuu] (1.8090169943749475,2.4898982848827793)-- ++(-2.5pt,-2.5pt) -- ++(5pt,5pt) ++(-5pt,0) -- ++(5pt,-5pt);
\draw [fill=black] (1,3.0776835371752527) circle (2.5pt);
\draw [color=uuuuuu] (0,3.077683537175253)-- ++(-2.5pt,-2.5pt) -- ++(5pt,5pt) ++(-5pt,0) -- ++(5pt,-5pt);
\draw [fill=qqwuqq] (-0.809016994374947,2.4898982848827798) circle (2.5pt);
\draw [fill=black] (-0.8090169943749475,0.5877852522924734) circle (2.5pt);
\end{scriptsize}
\end{tikzpicture}
\begin{tikzpicture}[line cap=round,line join=round,>=triangle 45,x=1cm,y=1cm, scale = 1.5]
\clip(-1.5,-2) rectangle (3,2);
\fill[line width=1pt,color=zzttqq,fill=zzttqq,fill opacity=0.10000000149011612] (-0.8090169943749475,0.5877852522924734) -- (0,0) -- (1,0) -- (1.809016994374947,0.5877852522924729) -- cycle;
\fill[line width=1pt,color=zzttqq,fill=zzttqq,fill opacity=0.10000000149011612] (0,0) -- (-0.809016994374947,-0.5877852522924729) -- (1.8090169943749475,-0.5877852522924734) -- (1,0) -- cycle;
\draw [line width=1pt] (-0.8090169943749475,0.5877852522924734)-- (0,0);
\draw [line width=1pt] (0,0)-- (1,0);
\draw [line width=1pt] (1,0)-- (1.809016994374947,0.5877852522924729);
\draw [line width=1pt,dash pattern=on 3pt off 3pt,color=zzttqq] (1.809016994374947,0.5877852522924729)-- (-0.8090169943749475,0.5877852522924734);
\draw [line width=1pt] (0,0)-- (-0.809016994374947,-0.5877852522924729);
\draw [line width=1pt,dash pattern=on 3pt off 3pt,color=zzttqq] (-0.809016994374947,-0.5877852522924729)-- (1.8090169943749475,-0.5877852522924734);
\draw [line width=1pt] (1.8090169943749475,-0.5877852522924734)-- (1,0);
\draw [line width=1pt,dash pattern=on 3pt off 3pt] (-1.1180339887498947,-1.5388417685876261)-- (-0.809016994374947,-0.5877852522924729);
\draw [line width=1pt,dash pattern=on 3pt off 3pt] (1.8090169943749475,-0.5877852522924734)-- (2.1180339887498945,-1.5388417685876268);
\draw [line width=1pt,dash pattern=on 3pt off 3pt] (-0.8090169943749475,0.5877852522924734)-- (-1.1180339887498945,1.5388417685876268);
\draw [line width=1pt,dash pattern=on 3pt off 3pt] (1.809016994374947,0.5877852522924729)-- (2.118033988749895,1.5388417685876261);
\draw [line width=1pt,color=qqwuqq] (-0.29910863284599076,-0.21731514225615348)-- (1.236363240860828,0.17172794654256607);
\draw [line width=1pt, to-to] (0,0.8)-- (1.8,0.8);
\draw [line width=1pt, to-to] (2,0.5918931279152921)-- (2,0);
\draw (0.15,1.35) node[anchor=north west] {$\left(1 + \cos\left( \frac{\pi}{5}\right)\right)$};
\draw (2.0374581975300234,0.5382606894675265) node[anchor=north west] {$ \sin \left( \frac{\pi}{5} \right)$};
\begin{scriptsize}
\draw [fill=uuuuuu] (0,0) circle (2pt);
\draw [color=black] (1,0)-- ++(-2pt,-2pt) -- ++(4pt,4pt) ++(-4pt,0) -- ++(4pt,-4pt);
\draw [fill=uuuuuu] (1.809016994374947,0.5877852522924729) circle (2pt);
\draw [color=black] (-0.8090169943749475,0.5877852522924734)-- ++(-2pt,-2pt) -- ++(4pt,4pt) ++(-4pt,0) -- ++(4pt,-4pt);
\draw [color=black] (-0.809016994374947,-0.5877852522924729)-- ++(-2pt,-2pt) -- ++(4pt,4pt) ++(-4pt,0) -- ++(4pt,-4pt);
\draw [fill=uuuuuu] (1.8090169943749475,-0.5877852522924734) circle (2pt);
\draw [fill=qqwuqq] (-0.29910863284599076,-0.21731514225615348) circle (2.5pt);
\draw [color=qqwuqq] (1.236363240860828,0.17172794654256607)-- ++(-2.5pt,-2.5pt) -- ++(5pt,5pt) ++(-5pt,0) -- ++(5pt,-5pt);
\end{scriptsize}
\end{tikzpicture}
\begin{tikzpicture}[line cap=round,line join=round,>=triangle 45,x=1cm,y=1cm, scale=1.5]
\clip(-1.2,-0.2) rectangle (2.95,3.35);
\fill[line width=1pt,color=zzttqq,fill=zzttqq,fill opacity=0.10000000149011612] (-0.8090169943749475,0.5877852522924734) -- (1.809016994374947,0.5877852522924729) -- (1,0) -- (0,0) -- cycle;
\fill[line width=1pt,color=zzttqq,fill=zzttqq,fill opacity=0.10000000149011612] (-0.809016994374947,2.4898982848827798) -- (0,3.077683537175253) -- (1,3.0776835371752527) -- (1.8090169943749475,2.4898982848827793) -- cycle;
\draw [line width=1pt,dash pattern=on 3pt off 3pt,color=zzttqq] (-0.8090169943749475,0.5877852522924734)-- (1.809016994374947,0.5877852522924729);
\draw [line width=1pt,dash pattern=on 3pt off 3pt,color=zzttqq] (1.8090169943749475,2.4898982848827793)-- (-0.809016994374947,2.4898982848827798);
\draw [line width=1pt] (-0.8090169943749475,0.5877852522924734)-- (0,0);
\draw [line width=1pt] (0,0)-- (1,0);
\draw [line width=1pt] (1,0)-- (1.809016994374947,0.5877852522924729);
\draw [line width=1pt,dash pattern=on 3pt off 3pt] (1.809016994374947,0.5877852522924729)-- (2.118033988749895,1.5388417685876261);
\draw [line width=1pt,dash pattern=on 3pt off 3pt] (2.118033988749895,1.5388417685876261)-- (1.8090169943749475,2.4898982848827793);
\draw [line width=1pt] (1.8090169943749475,2.4898982848827793)-- (1,3.0776835371752527);
\draw [line width=1pt] (1,3.0776835371752527)-- (0,3.077683537175253);
\draw [line width=1pt] (0,3.077683537175253)-- (-0.809016994374947,2.4898982848827798);
\draw [line width=1pt,dash pattern=on 3pt off 3pt] (-0.809016994374947,2.4898982848827798)-- (-1.1180339887498945,1.5388417685876268);
\draw [line width=1pt,dash pattern=on 3pt off 3pt] (-1.1180339887498945,1.5388417685876268)-- (-0.8090169943749475,0.5877852522924734);
\draw [line width=1pt,color=qqwuqq] (-0.1708162534821452,0.12412564759769638)-- (1.6104386789590934,0.4435096610031927);
\draw [line width=1pt,color=qqwuqq] (-0.49841221424334836,2.715565867050141)-- (1.3722007291069984,2.8072638785244157);
\draw [line width=1pt,to-to] (0.008485081973478146,2.2946154978499806)-- (1.8014984365297118,2.2946154978499806);
\draw [line width=1pt, to-to] (0.0009351737515659861,0.7574962032562996)-- (1.7995249624077834,0.7651335271996381);
\draw (0.3,2.3420083753595704) node[anchor=north west] {$1+\cos\left(\frac{\pi}{5}\right)$};
\draw (0.3,1.3) node[anchor=north west] {$1+\cos\left(\frac{\pi}{5}\right)$};
\begin{scriptsize}
\draw [color=black] (0,0)-- ++(-2.5pt,-2.5pt) -- ++(5pt,5pt) ++(-5pt,0) -- ++(5pt,-5pt);
\draw [fill=black] (1,0) circle (2.5pt);
\draw [color=uuuuuu] (1.809016994374947,0.5877852522924729)-- ++(-2.5pt,-2.5pt) -- ++(5pt,5pt) ++(-5pt,0) -- ++(5pt,-5pt);
\draw [fill=black] (2.118033988749895,1.5388417685876261) circle (2.5pt);
\draw [color=uuuuuu] (-1.1180339887498945,1.5388417685876268)-- ++(-2.5pt,-2.5pt) -- ++(5pt,5pt) ++(-5pt,0) -- ++(5pt,-5pt);
\draw [color=uuuuuu] (1.8090169943749475,2.4898982848827793)-- ++(-2.5pt,-2.5pt) -- ++(5pt,5pt) ++(-5pt,0) -- ++(5pt,-5pt);
\draw [fill=black] (1,3.0776835371752527) circle (2.5pt);
\draw [color=uuuuuu] (0,3.077683537175253)-- ++(-2.5pt,-2.5pt) -- ++(5pt,5pt) ++(-5pt,0) -- ++(5pt,-5pt);
\draw [fill=black] (-0.809016994374947,2.4898982848827798) circle (2.5pt);
\draw [fill=black] (-0.8090169943749475,0.5877852522924734) circle (2.5pt);
\end{scriptsize}
\end{tikzpicture}
\begin{tikzpicture}[line cap=round,line join=round,>=triangle 45,x=1cm,y=1cm, scale=1.5]
\clip(-1.5,-0.2) rectangle (2.9,3.35);
\fill[line width=1pt,color=zzttqq,fill=zzttqq,fill opacity=0.10000000149011612] (-0.8090169943749475,0.5877852522924734) -- (1.809016994374947,0.5877852522924729) -- (1,0) -- (0,0) -- cycle;
\fill[line width=1pt,color=zzttqq,fill=zzttqq,fill opacity=0.10000000149011612] (-0.809016994374947,2.4898982848827798) -- (0,3.077683537175253) -- (1,3.0776835371752527) -- (1.8090169943749475,2.4898982848827793) -- cycle;
\draw [line width=1pt,dash pattern=on 3pt off 3pt,color=zzttqq] (-0.8090169943749475,0.5877852522924734)-- (1.809016994374947,0.5877852522924729);
\draw [line width=1pt,dash pattern=on 3pt off 3pt,color=zzttqq] (1.8090169943749475,2.4898982848827793)-- (-0.809016994374947,2.4898982848827798);
\draw [line width=1pt] (-0.8090169943749475,0.5877852522924734)-- (0,0);
\draw [line width=1pt] (0,0)-- (1,0);
\draw [line width=1pt] (1,0)-- (1.809016994374947,0.5877852522924729);
\draw [line width=1pt,dash pattern=on 3pt off 3pt] (1.809016994374947,0.5877852522924729)-- (2.118033988749895,1.5388417685876261);
\draw [line width=1pt,dash pattern=on 3pt off 3pt] (2.118033988749895,1.5388417685876261)-- (1.8090169943749475,2.4898982848827793);
\draw [line width=1pt] (1.8090169943749475,2.4898982848827793)-- (1,3.0776835371752527);
\draw [line width=1pt] (1,3.0776835371752527)-- (0,3.077683537175253);
\draw [line width=1pt] (0,3.077683537175253)-- (-0.809016994374947,2.4898982848827798);
\draw [line width=1pt,dash pattern=on 3pt off 3pt] (-0.809016994374947,2.4898982848827798)-- (-1.1180339887498945,1.5388417685876268);
\draw [line width=1pt,dash pattern=on 3pt off 3pt] (-1.1180339887498945,1.5388417685876268)-- (-0.8090169943749475,0.5877852522924734);
\draw [line width=1pt,to-to] (0.008485081973478146,2.2946154978499806)-- (0.999760016134804,2.2966375325872597);
\draw [line width=1pt,to-to] (0.0009351737515659861,0.7574962032562996)-- (1.7995249624077834,0.7651335271996381);
\draw (0.3,1.3) node[anchor=north west] {$1+\cos\left(\frac{\pi}{5}\right)$};
\draw [line width=1pt,color=qqwuqq] (-0.1708162534821452,0.12412564759769638)-- (1.809016994374947,0.5877852522924729);
\draw [line width=1pt,color=qqwuqq] (-0.27226650235220895,2.8798703442651012)-- (1,3.0776835371752527);
\draw (0.3,2.2) node[anchor=north west] {$1$};
\draw (1.1,3.4) node[anchor=north west] {$A'$};
\draw (2.1,1.9) node[anchor=north west] {$C'$};
\draw(1.8,0.7) node[anchor=north west] {$B'$};
\begin{scriptsize}
\draw [color=black] (0,0)-- ++(-2.5pt,-2.5pt) -- ++(5pt,5pt) ++(-5pt,0) -- ++(5pt,-5pt);
\draw [fill=black] (1,0) circle (2.5pt);
\draw [color=uuuuuu] (1.809016994374947,0.5877852522924729)-- ++(-2.5pt,-2.5pt) -- ++(5pt,5pt) ++(-5pt,0) -- ++(5pt,-5pt);
\draw [fill=black] (2.118033988749895,1.5388417685876261) circle (2.5pt);
\draw [color=uuuuuu] (1.8090169943749475,2.4898982848827793)-- ++(-2.5pt,-2.5pt) -- ++(5pt,5pt) ++(-5pt,0) -- ++(5pt,-5pt);
\draw [fill=black] (1,3.0776835371752527) circle (2.5pt);
\draw [color=uuuuuu] (0,3.077683537175253)-- ++(-2.5pt,-2.5pt) -- ++(5pt,5pt) ++(-5pt,0) -- ++(5pt,-5pt);
\draw [fill=black] (-0.809016994374947,2.4898982848827798) circle (2.5pt);
\draw [color=uuuuuu] (-1.1180339887498945,1.5388417685876268)-- ++(-2.5pt,-2.5pt) -- ++(5pt,5pt) ++(-5pt,0) -- ++(5pt,-5pt);
\draw [fill=black] (-0.8090169943749475,0.5877852522924734) circle (2.5pt);
\end{scriptsize}
\end{tikzpicture}
\caption{Virtual length of segments intersecting the short horizontal cylinder. Top-left: the two cases for the first segment of a maximal trip through the horizontal cylinder; top-right: a sandwiched segment; bottom-left: an intermediate non-sandwiched segment; bottom-right: the last segment of a maximal trip. 
}
\label{fig:cases_additional_length}
\end{figure}
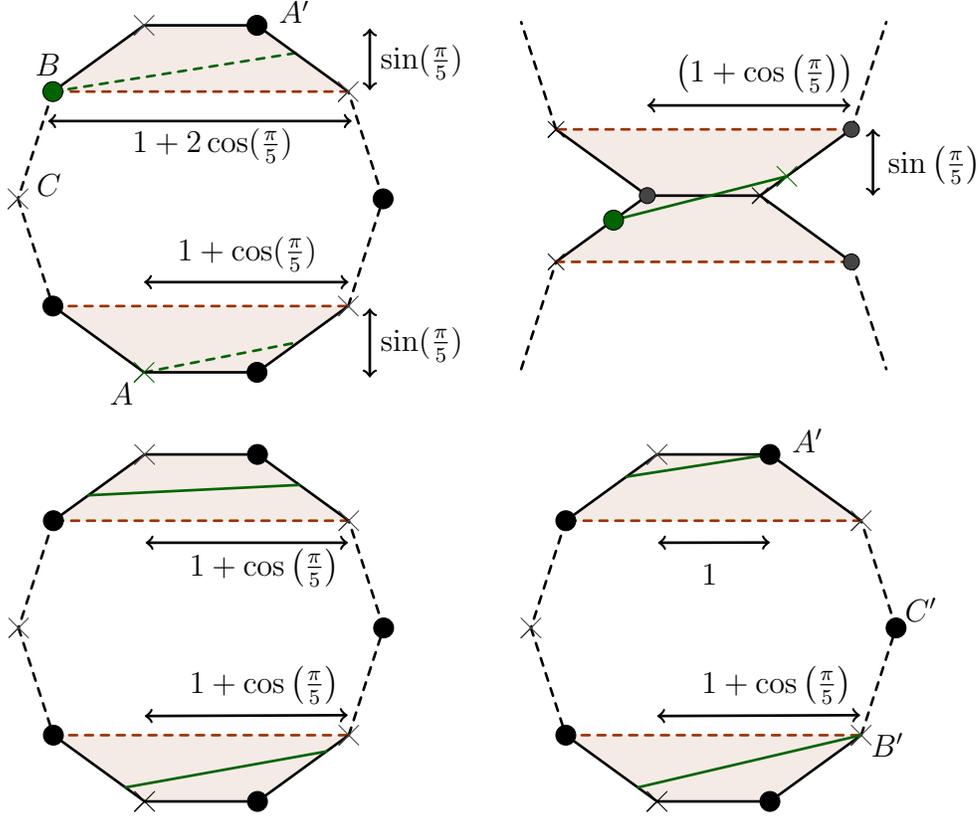

As a conclusion, adding up the "virtual" length of each segment, we obtain the following estimate on the length:
\small
\begin{equation*}
 l(\alpha_{i_0} \cup \cdots \cup \alpha_{i_0+p-1}) \geq \sqrt{\left(p+(p-1)\cos\left(\frac{2\pi}{n}\right)\right)^2 + \left((q+1) \sin\left(\frac{2\pi}{n}\right)\right)^2}
\end{equation*}
\normalsize
Which is the bound of Equation \eqref{eq:second_shortest}. In fact, we obtain slightly more information as from the above discussion this lower bound can be an equality (if and) only if $\alpha_{i_0}^- = A$ and $\alpha_{i_0+p-1} = A'$, that is $\alpha = \alpha_{i_0} \cup \cdots \cup \alpha_{i_0+p-1}$ is a saddle connection staying inside the short horizontal cylinder with endpoints on $A$ and $A'$. If this is not the case, we have a slightly better inequality as we can add $\cos\left(\frac{2\pi}{n}\right)$ to the length estimate of either the first or the last segment, giving the bound of Equation \eqref{eq:second_shortest_2}:
\small
\begin{equation*}\label{eq:case_slightly_better}
l(\alpha_{i_0} \cup \cdots \cup \alpha_{i_0+p-1}) \geq \sqrt{\left(p+p\cos\left(\frac{2\pi}{n}\right)\right)^2 + \left((q+1) \sin\left(\frac{2\pi}{n}\right)\right)^2}
\end{equation*}
\normalsize
This concludes the proof of Lemma \ref{lem:length_consecutive_short}.
\end{proof}
We can now deduce the  bound $l(\alpha) \geq \eta_0 n_\alpha + \varepsilon_0$ from Lemmas \ref{lem:long_segments} and \ref{lem:length_consecutive_short}. Let $\alpha$ be a saddle connection on the regular $n-$gon which is neither a side nor a saddle connection of type (2). Then the length of $\alpha$ is the sum of the lengths of its maximal trips through small cylinders (which cannot overlap) with the length of the remaining long segments.
Since long segments already have length at least $\eta_0 + \varepsilon_0$ it suffices to check that a maximal trip through a short cylinder made of $p$ segments has length at least $\eta_0 p+ \varepsilon_0$. 

\begin{itemize}
\item If $p \geq 4$, it suffices to consider the horizontal length to obtain the required inequality. With \eqref{eq:second_shortest}, we have:
\begin{align*}
l(\alpha) & \geq \sqrt{\left(p+(p-1)\cos\left(\frac{2\pi}{n}\right)\right)^2 + \left((q+1) \sin\left(\frac{2\pi}{n}\right)\right)^2}\\
&\geq p+(p-1)\cos\left(\frac{2\pi}{n}\right) \\
&\geq p+(p-1)\cos\left(\frac{\pi}{5}\right) \\
&\geq \eta_0 p + (1+\cos\left(\frac{\pi}{5}\right)-\eta_0)p - \cos\left(\frac{\pi}{5}\right)\\
&> \eta_0p + \varepsilon_0
\end{align*}
where the last inequality comes from the fact that $p \geq 4$ and:
\[ 4(1+\cos\left(\frac{\pi}{5}\right)-\eta_0) - \cos\left(\frac{\pi}{5}\right) \simeq 0.74 \dots > \varepsilon_0 \]
\item If $p = 3$, and $q=1$, then we have from \eqref{eq:second_shortest}:
\begin{align*}
l(\alpha) & \geq \sqrt{\left(3+2\cos\left(\frac{2\pi}{n}\right)\right)^2 + 2\sin\left(\frac{2\pi}{n}\right)^2}\\
& \geq \sqrt{13+12\cos\left(\frac{2\pi}{n}\right)}\\
& \geq \sqrt{13+12\cos\left(\frac{\pi}{5}\right)}\\
& \geq 4.765\\
& > 3\eta_0 + \varepsilon_0 \qquad (\simeq 4,743 \dots)
\end{align*}
\item If $p=3$ and $q=0$, then we should notice that the first and the last segment of the maximal sequence lie either both on top of the $n$-gon or both on the bottom part of the $n$-gon, and in particular $\alpha$ cannot both start at $A$ and end at $A'$, and therefore we can use Equation~\eqref{eq:second_shortest_2}. In particular:
\begin{align*}
l(\alpha) & \geq \sqrt{\left(3+3\cos\left(\frac{2\pi}{n}\right)\right)^2 + \sin\left(\frac{2\pi}{n}\right)^2}\\
& \geq \sqrt{10+18\cos\left(\frac{2\pi}{n}\right)+8\cos\left(\frac{2\pi}{n}\right)^2}\\
& \geq \sqrt{10+18\cos\left(\frac{\pi}{5}\right)+8\cos\left(\frac{\pi}{5}\right)^2}\\
& \geq 5.45\\
& > 3\eta_0 + \varepsilon_0
\end{align*}
\item If $p = 2$, then if $\alpha$ starts at $A$ and ends at $A'$, we have $\alpha = \Delta$. Otherwise\footnote{In fact if $\alpha \neq \Delta$ and $p=2$ then one can notice that the maximal sequence must start on $BC$ \textit{and} end on $B'C'$, because of the assumption on the slope of $\alpha$.}, we can use \eqref{eq:second_shortest_2} to obtain:
\begin{align*}
l(\alpha) & \geq \sqrt{\left(2+2\cos\left(\frac{2\pi}{n}\right)\right)^2 + \sin\left(\frac{2\pi}{n}\right)^2}\\
& \geq \sqrt{5+8\cos\left(\frac{2\pi}{n}\right)+3\cos\left(\frac{2\pi}{n}\right)^2}\\
& \geq \sqrt{5+8\cos\left(\frac{\pi}{5}\right)+3\cos\left(\frac{\pi}{5}\right)^2}\\
& \geq 3.66\\
& > 2\eta_0 + \varepsilon_0  \qquad (\simeq 3.326 \dots)
\end{align*}
\item Finally, if $p=1$, since sides are excluded that means $\alpha$ is a diagonal, and its length is least $\varphi_{10}= \eta_0 + \varepsilon_0$, with equality for short diagonals of the decagon.
\end{itemize}
This completes the proof of Proposition \ref{prop:study_lengths}.

\subsection{Study of the intersections}\label{sec:study_intersections}
We now prove Proposition \ref{prop:study_intersections} giving the number of non-singular intersection points between two saddle connections $\alpha$ and $\beta$ depending on their type. Which we recall here for convenience:
\begin{center}
\begin{tabular}{ |c|c||c|c|c|c| } 
\hline
& & \multicolumn{4}{|c|}{Type of $\alpha$} \\ 
\hline
& & (1) & (2) & (3) & (4)\\
 \hline
\multirow{4}{4em}{Type of $\beta$}& (1) & 0 & 1 & $n_{\alpha}-1$ & $n_{\alpha}-1$ \\ 

& (2) & $\star$ & $2$ & $2n_{\alpha}$ & $2n_{\alpha}-1$\\ 
& (3) & $\star$ & $\star$ & $n_{\alpha} n_{\beta}$ & $2n_{\alpha}n_{\beta} - 1$ \\ 
& (4) & $\star$ & $\star$ & $\star$ & $n_{\alpha} n_{\beta}$ \\
\hline
\end{tabular}
\end{center}
For this, let us first notice that the cases (2/2), (3/2), (3/3) directly come from the fact that two non-sandwiched segments can only intersect once. Also, cases (3/4) and (4/4) are a consequence of Proposition \ref{prop:intersection_BOU23}. Further, 
\begin{itemize}
\item[(1/1)] Distinct sides do not intersect, and by definition $|\alpha \cap \alpha| = 0$.
\item[(2/1)] The saddle connection $\Delta$ can only intersect one side on their interior.
\item[(3/1)] Since a saddle connection of type 3 does not contain sandwiched segments, the intersection of such a saddle connection with the interior of a given side is at most the number of times the saddle connection is cut into pieces in the subdivision, that is $n_\alpha - 1$.
\item[(4/1)] Similarly, intersection between a saddle connection $\alpha$ of type 4 and the interior of a side can only occur either on subdivision points if the side is not $\sigma_\alpha(1)$, giving at most $n_{\alpha}-1$ intersections, or on the interior of sandwiched segments if the considered side is $\sigma_\alpha(1)$, in which case there are at most $n_{\alpha}-2$ intersections (recall that the first and last segments of $\alpha$ are never sandwiched).
\item[(2/4)] Up to symmetry we can assume that $\alpha = \Delta$. Now, a given segment of a saddle connection $\beta$ can only intersect twice $\Delta$ on its interior (whether it be sandwiched or non-sandwiched), and in particular:
\[ |\Delta \cap \beta| \leq 2 n_{\beta} \]
However,
\begin{itemize}
\item The only segments that can intersect (in its interior) both segments of $\Delta$ are those with endpoints on the sides of label $\sigma_{\Delta}(1) = \sigma_0(1)$ and $\sigma_{\Delta}(2) = \sigma_0(2)$, and sandwiched segments.
\item Further, if the initial (resp. terminal) segment of $\beta$ does not have its singular endpoint among the common vertices $A, A'$ of sides with label $\sigma_0(1)$ and $\sigma_0(2)$, 
 and its non-singular endpoint either on $\sigma_0(2)$ or $\sigma_0(1)$, then the initial (resp. terminal) segment of $\beta$ intersects $\Delta$ at most once.
\end{itemize}
Since the only saddle connections satisfying both conditions are saddle connections of type 3, we conclude that if $\beta$ has type 4, then
\[ |\Delta \cap \beta| \leq 2 n_{\beta}-1 \]

\end{itemize}
\subsection{Conclusion: analyzing the cases}\label{sec:conclusion}
In this section, we finally prove Theorem \ref{theo:decagon}. We consider two closed curves $\gamma$ and $\delta$ which are made of either one or two saddle connections. Recall that we want to show
\[ \frac{\Int(\gamma,\delta)}{l(\gamma)l(\delta)} < \frac{1}{2} \]
or equivalently
\[ \Int(\gamma,\delta) < \frac{1}{2} l(\gamma)l(\delta) \]
with equality if and only if $\gamma$ and $\delta$ are both made of two sides, and they intersect twice (thus, at both singularities with the same sign).

\begin{Rema}\label{rema:geometric_intersection_1}
In fact we will prove the slightly stronger result that for two closed curves $\gamma = \cup_i^k \gamma_i$ and $\delta = \cup_j^l \delta_j$ made of respectively $k$ and $l$ saddle connections, with $k,l \in \{ 1,2\}$ and under the condition that it passes at most once through each singularity, we have
\[ 
\frac{\left(\sum_{i,j} |\gamma_i \cap \delta_j|\right) + s_{\gamma, \delta}}{l(\gamma) l(\delta)} \leq \frac{1}{2}
\]
where $s_{\gamma, \delta} \in \{0, 1,2\}$ is the number of singularities contained in both $\gamma$ and $\delta$.
In particular, a subdivision argument shows that the same result holds for any value of $k$ and $l$ without the assumption on the curves passing at most once through each singularity. This slightly stronger version allows to extend the result of Theorem \ref{theo:decagon} to the case where the algebraic intersection is replaced by the geometric intersection.
\end{Rema}

We subdivide the study into two main cases:
\begin{enumerate}
\item[(I)] At least one saddle connections is a side.
\item[(II)] None of the saddle connections are sides.
\end{enumerate}

\subsection*{I - One of the saddle connections is a side}
Up to a permutation, we can assume that $\gamma = \gamma_1 \cup \gamma_2$ and $\gamma_1$ is a side. This first case is actually the longer, and we will need to subdivide it into several sub-cases according to the type of $\gamma_2$.
\begin{enumerate}
\item[(a)] $\gamma_2$ is also a side,
\item[(b)] $\gamma_2$ has type (2),
\item[(c)] $\gamma_2$ has type (3),
\item[(d)] $\gamma_2$ has type (4).
\end{enumerate}

\paragraph{(a) If $\gamma_2$ is also a side,} then we can assume that $\gamma_1$ and $\gamma_2$ are not adjacent sides, as otherwise $\gamma$ would be homologous to a short diagonal, and therefore $\gamma$ would not minimize the length in its homology class. In particular, any saddle connection $\beta$ decomposed into $n_{\beta}$ segments will intersect (outside the singularities) the union of $\gamma_1 \cup \gamma_2$ at most $n_{\beta}-1$ times. This is because the intersections can only occur when $\beta$ crosses a side, and this side either subdivide two segments or is in the interior of a sandwiched segment. Furthermore, in the case of a sandwiched segment, since $\gamma_1$ and $\gamma_2$ are not adjacent, a sandwiched segment can only intersect one of these sides (or, more precisely, if it intersects one of the $\gamma_i$, $i \in \{ 1,2\}$ on its interior, its extremities do not lie on the other side $\gamma_{3-i}$). Using this argument, we easily obtain the following upper bound:

\begin{Lem}\label{lem:caseIa}
Assume $\gamma_1$ and $\gamma_2$ are two non-adjacent sides of the regular $n-$gon. For any saddle connection $\beta$ we have
\[ \frac{|\gamma_1\cap \beta|+ |\gamma_2\cap \beta|+ 1}{l(\gamma_1 \cup \gamma_2)l(\beta)} \leq \frac{1}{2}, \]
and equality can only occur if $\beta$ is a side.
\end{Lem}
\begin{proof}[Proof of Lemma \ref{lem:caseIa}]
\begin{itemize}
\item If $\beta$ is a side, then $|\gamma_1 \cap \beta| = |\gamma_2 \cap \beta| = 0$ and $l(\gamma_1) = l(\gamma_2) = l(\beta) = 1$.
\item If $\beta$ has type 2, 3 or type 4, then $|\gamma_1 \cap \beta| + |\gamma_2 \cap \beta| \leq n_{\beta}-1$ by the above argument, $l(\gamma_1 \cup \gamma_2) = 2$, and $l(\beta) > n_\beta$, so that
\[ \frac{|\gamma_1\cap \beta|+ |\gamma_2\cap \beta|+ 1}{l(\gamma_1 \cup \gamma_2)l(\beta)} < \frac{n_\beta}{2 n_\beta} = \frac{1}{2}. \]
\end{itemize}
\end{proof}

As a consequence of Lemma \ref{lem:caseIa}, if the second curve $\delta$ is made of a single saddle connection (that is $\delta = \beta$, which is thus not a side since sides are not closed), then there is at most one singular intersection between $\gamma$ and $\delta$ and therefore $\Int(\gamma, \delta) \leq |\gamma_1\cap \beta|+ |\gamma_2\cap \beta|+ 1$, and Lemma \ref{lem:caseIa} directly yields
\[ \frac{\Int(\gamma,\delta)}{l(\gamma)l(\delta)} < \frac{1}{2}. \]
Now, if $\delta = \delta_1 \cup \delta_2$ is made of two (non-closed) saddle connections, then 
\begin{align*}
\Int(\gamma, \delta) &\leq |\gamma_1\cap \delta_1|+ |\gamma_2\cap \delta_1|+ |\gamma_1\cap \delta_2|+ |\gamma_2\cap \delta_2| + 2 \\
& \leq (|\gamma_1\cap \delta_1|+ |\gamma_2\cap \delta_1|+1)+ (|\gamma_1\cap \delta_2|+ |\gamma_2\cap \delta_2| + 1) \\
& \leq \frac{1}{2} l(\gamma_1 \cup \gamma_2) l(\delta_1) + \frac{1}{2} l(\gamma_1 \cup \gamma_2) l(\delta_2)\\
& \leq \frac{1}{2} l(\gamma) l(\delta)  \end{align*}
with equality if and only if $\delta = \delta_1 \cup \delta_2$ is made of two sides, and $\gamma$ and $\delta$ intersect at both singularities with the same sign.

\paragraph{(b) If $\gamma_2$ has type (2),} then we can assume that $\gamma_1 \neq \sigma_{\gamma_2}(1)$, otherwise $\gamma_1$ and $\gamma_2$ would share a common endpoint on the regular $n$-gon and we could find a shorter curve in the same homology class.\newline
Now, if $\beta$ is a side, then $|\gamma_1 \cap \beta|+|\gamma_2 \cap \beta| \leq 1$ by Proposition \ref{prop:study_intersections}  and we have:
\[ \frac{|\gamma_1\cap \beta|+ |\gamma_2\cap \beta|}{l(\gamma_1 \cup \gamma_2)l(\beta)} \leq \frac{1}{2\sqrt{2} + 1} < \frac{1}{2}. \]
Now, we show that in all the other cases, we have:
\begin{Lem}\label{lem:caseIb}
Assume $\gamma = \gamma_1 \cup \gamma_2$ where $\gamma_2$ has type (2) and $\gamma_1$ is a side which is neither $\sigma_{\gamma_2}(1)$ nor $\sigma_{\gamma_2}(2)$. Assume the saddle connection $\beta$ is not a side, and define $s_{\beta} := 1$ if $\beta$ is closed and $s_{\beta} :=2$ otherwise (the number $s_{\beta}$ is an upper bound on the number of singular intersections between $\gamma$ and $\delta$). Then, we have
\[ \frac{|\gamma_1\cap \beta|+ |\gamma_2\cap \beta| + s_{\beta}}{l(\gamma_1 \cup \gamma_2)l(\beta)} < \frac{1}{2}. \]
\end{Lem}
As a consequence, since we already dealt in (a) with the case where one of the curves was made of two sides, we can assume that the curve $\delta$ is not made of two sides so that at least one saddle connection of $\delta$ is in case (2), (3) or (4). This allows to compensate for the possible singular intersections (at most $s_{\beta}$) and obtain
\[ \frac{\Int(\gamma,\delta)}{l(\gamma)l(\delta)} < \frac{1}{2} \]
as required.
\begin{proof}[Proof of Lemma \ref{lem:caseIb}]
We deal with three cases according to the type of $\beta$.
\begin{itemize}
\item If $\beta$ has type 2, then we deduce from Proposition~\ref{prop:study_intersections} that $|\gamma_1 \cap \beta| + |\gamma_2 \cap \beta| \leq 3$ and hence
\[ \frac{|\gamma_1 \cap \beta| + |\gamma_2 \cap \beta| + 2}{l(\gamma_1 \cap \gamma_2) l(\beta)} \leq \frac{5}{(1 + 2\sqrt{2})2\sqrt{2}} < \frac{1}{2} \]
as required (here $s_{\beta} = 2$).
\item If $\beta$ has type 3, then we deduce from Proposition~\ref{prop:study_intersections} that $|\gamma_1 \cap \beta| + |\gamma_2 \cap \beta| \leq 3 n_{\beta} - 1$ and we deduce from Proposition \ref{prop:study_lengths} that $l(\beta) > 2\sqrt2 n_\beta$. Therefore
\[ \frac{|\gamma_1 \cap \beta| + |\gamma_2 \cap \beta| + 2}{l(\gamma_1 \cap \gamma_2) l(\beta)} \leq \frac{3 n_\beta +1}{(1 + 2\sqrt{2})2\sqrt{2} n_{\beta}} \leq \frac{3 n_\beta +1}{(2\sqrt{2}+8) n_{\beta}} < \frac{1}{2} \]
as required (here $s_{\beta} = 2$).
\item If $\beta$ has type 4, then either $\beta$ does not intersect the side $\gamma_1$ and from Proposition \ref{prop:study_intersections} we have
\[ |(\gamma_1 \cup \gamma_2) \cap \beta| \leq 2 n_{\beta} - 1 \]
or $\beta$ intersects the side $\gamma_1$. In the latter case, we obtain from the assumption $\gamma_1  \neq \sigma_{\gamma_2}(1)$ and the proof of (2/4), Proposition~\ref{prop:study_intersections} that given two consecutive segments $\beta_{j}$ and $\beta_{j+1}$ that share an endpoint on the interior of $\gamma_1$, only one of them can intersect $\gamma_2$ twice. In particular, we lose an intersection with $\gamma_2$ while gaining one with $\gamma_1$. Hence, we also have
\[ |(\gamma_1 \cup \gamma_2) \cap \beta| \leq 2 n_{\beta} - 1 \]
In particular, using that $l(\beta) \geq \eta_0 n_\beta + \varepsilon_0$, we have
\[ \frac{|\gamma_1 \cap \beta| + |\gamma_2 \cap \beta| + 2}{l(\gamma_1 \cap \gamma_2) l(\beta)} \leq \frac{2 n_{\beta} + 1}{(1 + 2\sqrt{2})(\eta_0 n_\beta + \varepsilon_0)} \]
Now since $(1 + 2\sqrt{2}) \eta_0 
> 5$ and $(1 + 2\sqrt{2})\varepsilon_0 > 1$ we have
\[
(1 + 2\sqrt{2})(\eta_0 n_\beta + \varepsilon_0) > 5 n_\beta  + 1 \geq 4 n_\beta + 2 = 2(2n_\beta +1)
\]
thus providing the required bound.
This completes the proof of Lemma~\ref{lem:caseIb}.
\end{itemize}
\end{proof}

\paragraph{(c) If $\gamma_2$ has type (3),} then for every saddle connection $\beta$ of type (2), (3) or (4), we use Propositions \ref{prop:study_lengths} and \ref{prop:study_intersections} to obtain:

\begin{Lem}\label{lem:caseIc}
Assume $\gamma = \gamma_1 \cup \gamma_2$, where $\gamma_1$ is a side and $\gamma_2$ has type (3). For any saddle connection $\beta$, we have:
\[ \frac{|\gamma_1 \cap \beta| + |\gamma_2 \cap \beta| + 1}{l(\gamma_1 \cup \gamma_2) l(\beta)} < \frac{1}{2} \]
\end{Lem}
In particular, distinguishing the cases where $\delta$ is made of one or two saddle connections as in case (a), we obtain
\[ \frac{\Int(\gamma,\delta)}{l(\gamma)l(\delta)} < \frac{1}{2} \]
as required.

\begin{proof}[Proof of Lemma~\ref{lem:caseIc}]
By Proposition \ref{prop:study_lengths}, we have $$l(\gamma_1 \cup \gamma_2) \geq 1 + 2\sqrt{2}n_{\gamma_2} + \varepsilon_1 > 2\sqrt{2} n_{\gamma_2}$$ and
\begin{enumerate}
\item If $\beta$ is a side, then $|\gamma_1 \cap \beta| = 0$ and $|\gamma_2 \cap \beta| \leq n_{\gamma_2}-1$ by Proposition \ref{prop:study_intersections}, and $l(\beta) = 1$ so that
\[ \frac{|\gamma_1 \cap \beta| + |\gamma_2 \cap \beta| + 1}{l(\gamma_1 \cup \gamma_2) l(\beta)} <\frac{n_{\gamma_2}}{2 \sqrt{2} n_{\gamma_2}} < \frac{1}{2} \]
as required.
\item If $\beta$ has type (2), then $|\gamma_1 \cap \beta| \leq 1$ and $|\gamma_2 \cap \beta| \leq 2 n_{\gamma_2}$ by Proposition \ref{prop:study_intersections}, and $l(\beta) > 2 \sqrt{2}$ by Proposition \ref{prop:study_lengths}. Hence we obtain
\[ \frac{|\gamma_1 \cap \beta| + |\gamma_2 \cap \beta| + 1}{l(\gamma_1 \cup \gamma_2) l(\beta)} < \frac{2n_{\gamma_2}+2}{2\sqrt{2} n_{\gamma_2} \cdot 2\sqrt{2} } = \frac{2n_{\gamma_2}+2}{8n_{\gamma_2}} < \frac{1}{2} \]
where the last inequality come from $n_{\gamma_2} \geq 1$.
\item If $\beta$ has type (3), then $|\gamma_1 \cap \beta| \leq n_{\beta}-1$ and $|\gamma_2 \cap \beta| \leq n_{\gamma_2} n_{\beta}$ by Proposition \ref{prop:study_intersections}, and $l(\beta) > 2 \sqrt{2} n_\beta$ as a consequence of Proposition \ref{prop:study_lengths} (we are not using the additional $\varepsilon_1$ here). Hence we obtain
\[ \frac{|\gamma_1 \cap \beta| + |\gamma_2 \cap \beta| + 1}{l(\gamma_1 \cup \gamma_2) l(\beta)} < \frac{(n_{\gamma_2} +1) n_{\beta}}{8n_{\gamma_2}n_{\beta}} \leq \frac{(n_{\gamma_2} +1) n_{\beta}}{2(2n_{\gamma_2}+2)n_{\beta}} < \frac{1}{2} \]
again using that $n_{\gamma_2} \geq 1$.
\item Finally, if $\beta$ has type (4), then $|\gamma_1 \cap \beta| \leq n_{\beta}-1$ and $|\gamma_2 \cap \beta| \leq 2 n_{\gamma_2}n_{\beta} - 1$ by Proposition \ref{prop:study_intersections}. Further, $l(\beta) > \eta_0 n_{\beta}$, so that
\[ \frac{|\gamma_1 \cap \beta| + |\gamma_2 \cap \beta| + 1}{l(\gamma_1 \cup \gamma_2) l(\beta)} < \frac{(2n_{\gamma_2}+1)n_{\beta} - 1}{ (2\sqrt{2}n_{\gamma_2} + 1 + \varepsilon_1) \eta_0n_{\beta} } \]
Notice that in this step we are really using $\eta_0$ and not $\sqrt{2}$, and we are also using the additional $1+\varepsilon_1$ length, which was not necessary in the previous cases. Indeed, the denominator satisfies
\[ (2\sqrt{2}n_{\gamma_2} + 1 + \varepsilon_1) \eta_0 n_{\beta} \geq \left( 2\sqrt{2}\eta_0 \cdot n_{\gamma_2} + (1 + \varepsilon_1)\eta_0 \right) n_{\beta}\]
and since $2\sqrt{2}\eta_0 > 4$ as well as $(1 + \varepsilon_1)\eta_0 = 2$ we obtain
\[ (2\sqrt{2}n_{\gamma_2} + 1 + \varepsilon_1) \eta_0 n_{\beta} > \left( 4 n_{\gamma_2} + 2 \right) n_{\beta} \]
and therefore 
\[ \frac{|\gamma_1 \cap \beta| + |\gamma_2 \cap \beta| + 1}{l(\gamma_1 \cup \gamma_2) l(\beta)} < \frac{1}{2} \]
as required.
\end{enumerate}
\end{proof}

\paragraph{(d) Finally, if $\gamma_2$ has type (4),} then we proceed similarly: 

\begin{itemize}
    \item First, if $\delta$ is made of two saddle connections, and one of them (say, $\delta_1$) is a side, then we can assume that $\delta_2$ has type 4; because we already dealt with the case where one of the saddle connections was made of two sides, or a side and a saddle connection of type 2 or 3. In particular, we have from our length and intersections estimates:
\begin{align*}
\Int(\gamma,\delta) & \leq 
|\gamma_1 \cap \delta_1| + |\gamma_1 \cap \delta_2| + |\gamma_2 \cap \delta_1| + |\gamma_2 \cap \delta_2| +2 \\
& \leq 0 + (n_{\delta_2}-1) + (n_{\gamma_2}-1) + n_{\gamma_2}n_{\delta_2} +2 \\
& = n_{\gamma_2}n_{\delta_2} + n_{\gamma_2} + n_{\delta_2}\\
& < (n_{\gamma_2}+1)(n_{\delta_2}+1)
\end{align*}
and, using that $1+\varepsilon_0 > \sqrt2$,  
\[ l(\gamma) = l(\gamma_1) + l(\gamma_2) \geq 1 + \sqrt{2} n_{\gamma_2} + \varepsilon_0 > \sqrt{2} (n_{\gamma_2} +1) \]
\[ l(\delta) = l(\gamma_1) + l(\gamma_2) \geq 1 + \sqrt{2} n_{\delta_2} + \varepsilon_0 > \sqrt{2} (n_{\delta_2} +1) \]
so that
\[ \frac{\Int(\gamma,\delta)}{l(\gamma)l(\delta)} < \frac{(n_{\gamma_2}+1)(n_{\delta_2}+1)}{2 (n_{\gamma_2}+1)(n_{\delta_2}+1)}= \frac{1}{2} \]
as required.

\item We can now deal with the case where no side appear in $\delta$. For this, we show that for every saddle connection $\beta$ whose type is (2), (3) or (4), we have
\end{itemize}
\begin{Lem}\label{lem:caseId}
If $\beta$ is not a side, then
\[ \frac{|\gamma_1 \cap \beta| + |\gamma_2 \cap \beta| + 1}{l(\gamma_1 \cup \gamma_2) l(\beta)} < \frac{1}{2} \]
\end{Lem}
As in the previous cases, using Lemma \ref{lem:caseId} and distinguishing the cases where $\delta$ is made of one or two saddle connections yields the required result:
\[ \frac{\Int(\gamma,\delta)}{l(\gamma)l(\delta)} < \frac{1}{2}. \]

\begin{proof}[Proof of Lemma \ref{lem:caseId}]
First, we have
\[ l(\gamma_1 \cup \gamma_2) = l(\gamma_1) + l(\gamma_2) \geq 1 + \sqrt{2} n_{\gamma_2} + \varepsilon_0 > \sqrt{2} (n_{\gamma_2} +1) \]
(again using that $1+\varepsilon_0 > \sqrt{2}$). Then, 
\begin{itemize}
\item If $\beta$ has type 2, then $|\gamma_1 \cap \beta| \leq 1$ and $|\gamma_2 \cap \beta|\leq 2 n_{\gamma_2}-1$, but also $l(\beta) > 2 \sqrt{2}$, so that
\[ \frac{|\gamma_1 \cap \beta| + |\gamma_2 \cap \beta| + 1}{l(\gamma_1 \cup \gamma_2) l(\beta)} < \frac{2 n_{\gamma_2} +1}{4  (n_{\gamma_2} +1)} < \frac{1}{2} \]
as required.
\item If $\beta$ has type 3, then $|\gamma_1 \cap \beta| \leq n_{\beta} - 1$ and $|\gamma_2 \cap \beta|\leq 2 n_{\gamma_2}n_{\beta}-1$, but also $l(\beta) > 2 \sqrt{2} n_{\beta}$, so that
\[ \frac{|\gamma_1 \cap \beta| + |\gamma_2 \cap \beta| + 1}{l(\gamma_1 \cup \gamma_2) l(\beta)} < \frac{(2 n_{\gamma_2} +1)n_{\beta}-1}{4  (n_{\gamma_2} +1)n_{\beta}} < \frac{1}{2} \]
as required.
\item If $\beta$ has type 4, then $|\gamma_1 \cap \beta| \leq n_{\beta} - 1$ and $|\gamma_2 \cap \beta|\leq n_{\gamma_2}n_{\beta}$, but also $l(\beta) > \sqrt{2} n_{\beta}$, so that
\[ \frac{|\gamma_1 \cap \beta| + |\gamma_2 \cap \beta| + 1}{l(\gamma_1 \cup \gamma_2) l(\beta)} < \frac{(n_{\gamma_2} +1)n_{\beta}}{2 (n_{\gamma_2} +1)n_{\beta}} = \frac{1}{2} \]
as required.
\end{itemize}
\end{proof}

This concludes the case where one of the saddle connections is a side. We now deal with the case where no side appear in neither $\gamma$ nor $\delta$.

\subsection*{(II) All saddle connections have type 2, 3 or 4}

This second case will be subdivided into two cases:
\begin{enumerate}
\item[(a)] $\gamma$ and $\delta$ are both short diagonals,
\item[(b)] $\gamma$ and $\delta$ are not both short diagonals.
\end{enumerate}

\paragraph{Case (a).} If $\gamma$ and $\delta$ are both short diagonals (which we recall are closed curves), then we have $\Int(\gamma,\delta) \leq 1$ as $\gamma$ and $\delta$ can either intersect at a singularity, or intersect on the interior of the regular $n$-gon, in which case the endpoints of $\gamma$ and the endpoints of $\delta$ do not represent the same singularity, so they do not intersect at the singularity. Further, $l(\gamma) = l(\delta) = 2 \cos \left(\frac{\pi}{n}\right)$ and hence
\[ \frac{\Int(\gamma,\delta)}{l(\gamma)l(\delta)} \leq \frac{1}{\left(2 \cos \left(\frac{\pi}{n}\right)\right)^2} < \frac{1}{2} \]
as required.\newline

\paragraph{Case (b).} Otherwise we show

\begin{Lem}\label{lem:caseIIb}
Assume $\alpha$ and $\beta$ are two saddle connections which are not sides, and at least one of them is not a short diagonal. Then:
\[ \frac{|\alpha \cap \beta| + 1}{l(\alpha) l(\beta)} < \frac{1}{2}. \]
\end{Lem}
As a consequence, if $\gamma = \bigcup_{i=1}^{k} \gamma_i$ and $\delta = \bigcup_{j=1}^{l} \delta_j$ are unions of $k$ (resp. $l$) saddle connections, we have
\[
\Int(\gamma,\delta) \leq \left(\sum_{i,j} |\gamma_i \cap \delta_j|\right) + s_{\gamma, \delta} \leq \left(\sum_{i,j} |\gamma_i \cap \delta_j|+1\right) \]
as $s_{\gamma, \delta} \in \{0, 1,2\}$ (defined in Remark \ref{rema:geometric_intersection_1}) is an upper bound on the number of singular intersections. In particular, from Lemma \ref{lem:caseIIb}, we obtain:
\begin{align*}
\Int(\gamma,\delta) \leq \left(\sum_{i,j} |\gamma_i \cap \delta_j|+1\right) \leq \frac{1}{2} \left(\sum_{i,j} l(\gamma_i) l(\delta_j)\right) = \frac{1}{2} l(\gamma) l(\delta)
\end{align*}
as required.\newline

We are left to prove Lemma \ref{lem:caseIIb}.
\begin{proof}[Proof of Lemma \ref{lem:caseIIb}]
We distinguish cases according to the type of $\alpha$ and $\beta$.

\begin{itemize}
\item[(2/2)] In this case $|\alpha \cap \beta| \leq 2$ and $l(\alpha), l(\beta) \geq 2 \sqrt{2}$ from Propositions \ref{prop:study_lengths} and \ref{prop:study_intersections}, so that
\[ \frac{|\alpha \cap \beta|+1}{l(\alpha) l(\beta)} < \frac{3}{8} < \frac{1}{2}. \]
\item[(2/3)] In this case $|\alpha \cap \beta| \leq 2 n_{\beta}$ and $l(\alpha) > 2 \sqrt{2}$ and $l(\beta) > 2\sqrt{2} n_{\beta}$, so that
\[ \frac{|\alpha \cap \beta|+1}{l(\alpha) l(\beta)} < \frac{2 n_{\beta}+1}{8n_{\beta}} < \frac{1}{2}. \]
\item[(2/4)] In this case $|\alpha \cap \beta| \leq 2 n_{\beta}-1$ and $l(\alpha) > 2 \sqrt{2}$ and $l(\beta) \geq \eta_0 n_\beta + \varepsilon_0 > \sqrt{2} n_{\beta}$, so that
\[ \frac{|\alpha \cap \beta|+1}{l(\alpha) l(\beta)} < \frac{2 n_{\beta}}{4n_{\beta}} = \frac{1}{2}. \]
\item[(3/3)] In this case $|\alpha \cap \beta| \leq n_{\alpha}n_{\beta}$ and $l(\alpha) > 2 \sqrt{2} n_{\alpha}$ and $l(\beta) > 2\sqrt{2} n_{\beta}$, so that
\[ \frac{|\alpha \cap \beta|+1}{l(\alpha) l(\beta)} < \frac{n_{\alpha} n_{\beta}+1}{8 n_{\alpha}n_{\beta}} < \frac{1}{2}. \]
\item[(3/4)] In this case $|\alpha \cap \beta| \leq 2 n_{\alpha}n_{\beta}-1$ and $l(\alpha) > \sqrt{2} n_{\alpha}$ and $l(\beta) > 2\sqrt{2} n_{\beta}$, so that
\[ \frac{|\alpha \cap \beta|+1}{l(\alpha) l(\beta)} < \frac{2n_{\alpha} n_{\beta}}{4 n_{\alpha}n_{\beta}} = \frac{1}{2}. \]
\item[(4/4)] In this case, we have from Propositions \ref{prop:study_lengths} and \ref{prop:study_intersections} $|\alpha \cap \beta| \leq n_{\alpha}n_{\beta}$ and $l(\alpha) \geq \eta_0 n_\alpha + \varepsilon_0 > \sqrt{2} n_{\alpha} + \varepsilon_0$ and similarly $l(\beta) > \sqrt{2} n_{\beta} + \varepsilon_0$, so that
 \begin{align*}
\frac{|\alpha \cap \beta| + 1}{l(\alpha)l(\beta)} &\leq \frac{n_{\alpha} n_{\beta} + 1}{(\sqrt{2}n_{\alpha} + \varepsilon_0)(\sqrt{2}n_{\beta} + \varepsilon_0)}\\
& \leq \frac{n_{\alpha} n_{\beta} + 1}{2n_{\alpha}n_\beta + \sqrt{2} \varepsilon_0 (n_{\alpha} + n_{\beta}) + \varepsilon_0^2}
\end{align*}
Now,
\begin{itemize}
\item as soon as $n_{\alpha} + n_\beta \geq 3$ we have 
\[ \sqrt{2} \varepsilon_0 (n_{\alpha} + n_{\beta}) + \varepsilon_0^2 \geq 3\sqrt{2} \varepsilon_0+ \varepsilon_0^2 > 2 \]
and hence
\[
\frac{|\alpha \cap \beta| + 1}{l(\alpha)l(\beta)} < \frac{n_{\alpha} n_{\beta} + 1}{2n_\alpha n_\beta + 2} = \frac{1}{2}
\]
\item Otherwise, $n_\alpha + n_\beta = 2$ and then $\alpha$ and $\beta$ are both diagonals, and since we assumed that at least one of them is not a short diagonal (say $\beta$) we have $l(\alpha) \geq 2 \cos \left(\frac{\pi}{n}\right)$ and $l(\beta) \geq \left(2\cos \left(\frac{\pi}{n}\right) \right)^2 - 1$ (which is the length of the second shortest diagonal) and hence
\begin{align*}
\frac{|\alpha \cap \beta| + 1}{l(\alpha)l(\beta)} &\leq \frac{2}{2\cos \left(\frac{\pi}{n}\right) \left(\left(2\cos \left( \frac{\pi}{n}\right)\right)^2 - 1\right)}\\
& \leq \frac{2}{2\cos \left(\frac{\pi}{n}\right) \left(\left(2\cos \left( \frac{\pi}{10}\right)\right)^2 - 1\right)}< \frac{1}{2}
\end{align*}
as required.
\end{itemize}
\end{itemize}
\end{proof}
This concludes the proof of Theorem \ref{theo:decagon}.

\section{Bouw-M\"oller surfaces}\label{sec:Bouw-Moller}
In this section, we generalize the method developed in the previous sections to a family of Bouw-M\"oller surfaces, and we prove Theorem \ref{theo:BM}. 
Given $m,n \geq 2$ with $(m,n) \neq (2,2)$, we recall that the Bouw-M\"oller surface $S_{m,n}$ is made of $m$ semi-regular polygons $P(0), P(1),\dots, P(m-1)$, that are equiangular $2n$-gons and where the sides of the polygon $P(i), 0 \leq i \leq m-1$ have alternating length $\sin \left(\frac{i\pi}{m}\right)$ and $\sin\left(\frac{(i+1)\pi}{m}\right)$. For the extremal polygons $P(0)$ and $P(m-1)$, there is a degenerate length and these polygons are in fact regular $n$-gons. Then, the sides of $P(i)$ are identified alternatively with sides of $P(i-1)$ and $P(i+1)$. See \cite{Hooper} or \cite{BP24} for a description of these surfaces. We will need here the two following facts:

\begin{Prop}{\cite[Proposition 24]{Hooper}}\label{prop:singularities_BM})
The Bouw-M\"oller surface $S_{m,n}$ has $d:=\gcd(m,n)$ singularities.
\end{Prop}

\begin{Prop}{\cite[Proposition 28]{Hooper}}\label{prop:rotation_BM}
The rotation by angle $\frac{2\pi}{n}$ on each of the polygons is an affine diffeomorphism of the surface.
\end{Prop}

We prove Theorem \ref{theo:BM} as follows. We first show in Section \ref{sec:lower_bound} that under the assumption $1 < d <n$ there exist two pairs of curves, each of them made of two sides of $P(0)$ (resp. $P(m-1)$), and intersecting twice. This will give the lower bound on KVol. Then, we prove in Section \ref{sec:upper_bound} that this lower bound is in fact sharp, using refinements of the length estimates of \cite{BP24} (which hold for a much larger class of surfaces). Our length estimates are similar to those obtained for the $(4m+2)-$gon in Proposition \ref{prop:study_lengths}. Better, contrary to the proof of Theorem \ref{theo:decagon} where we had to consider four different types of segments, only two are sufficient under the assumption $m,n \geq 8$ -- sides and non-sides -- and we restrict to this case for simplicity.

\subsection{Intersection of pairs of sides: Lower bound}\label{sec:lower_bound}
We first study the intersections of closed curves made of sides of the polygons on Bouw-M\"oller surfaces $S_{m,n}$. Given two closed curves $\alpha = \alpha_1 \cup \alpha_2 \cup \cdots \cup \alpha_r$ and $\beta = \beta_1 \cup \cdots \cup \beta_s$ where each of the $\alpha_i$ and $\beta_j$ is the side of a polygon, we have $\Int(\alpha,\beta) \leq \min(r,s)$ as $\alpha$ and $\beta$ can only intersect at the singularities. We show:

\begin{Prop}
Let $m \geq 2$, $n \geq 4$ and $d := gcd(m,n)$. We consider the Bouw-M\"oller surface $S_{m,n}$. Assume $d \neq 1$, then the maximal possible ratio obtained with curves made of sides is $\frac{1}{2l_0^2}$. Further, if we also assume $d \neq n$, then there is a pair of two closed curves, each of them made of two sides of $P(0)$ and/or $P(m-1)$, intersecting twice.

\end{Prop}
\begin{proof}
The first part of the statement comes from the fact that closed curves $\alpha$ and $\beta$ respectively made of $r$ and $s$ sides can intersect at most $\min(r,s)$ times. The assumption on $d$ ensures that $r, s \geq 2$ since in this case sides are not closed curves. This gives
\[
\frac{\Int(\alpha, \beta)}{l(\alpha) l(\beta)} \leq \frac{ \min(r,s)}{r s l_0^2} \leq \frac{1}{\max(r,s) l_0^2} \leq \frac{1}{2 l_0^2}
\]
and therefore equality holds only if $r = s = 2$.\newline

Let us now assume $d \neq 1, n$ in order to prove the second part of the statement. Recall that $d$ is the number of singularities of the Bouw-M\"oller surface $S_{m,n}$ (Proposition \ref{prop:singularities_BM}), and that $S_{m,n}$ is invariant under rotation of each of the polygons by angle $\frac{2\pi}{n}$ (Proposition \ref{prop:rotation_BM}).
In particular, the singularities of $S_{m,n}$ must be disposed symmetrically on the vertices of $P(0)$ (resp. $P(m-1)$). Moreover, if we denote by $z_1, \dots, z_d$ the singularities of $S_{m,n}$ and we label the vertices of $P(0)$ (resp. $P(m-1)$) in cyclic (say, anti-clockwise) order while choosing the vertex $0$ to represent $z_1$, then the vertices representing $z_1$ will be exactly those of label $dk$ for $k \in \{0, \dots, \frac{n}{d}-1\}$. Now, if say the vertex of label $1$ represent $z_2$, then the vertices representing $z_2$ will be exactly those of label $dk+1$ (again $k \in \{0, \dots, \frac{n}{d}-1\}$). It will also be the case for $P(m-1)$.

Then, let us consider the closed curve which is made of two (distinct) sides going from $z_1$ to $z_2$ in $P(0)$ (or, to be more precise, one, named $\alpha_1$ oriented from $z_1$ to $z_2$ and the other, named $\alpha_2$, from $z_2$ to $z_1$, so that the resulting curve $\alpha = \alpha_1 \cup \alpha_2$ is closed). This is made possible by $d < n$, so that there are at least two sides from a vertex representing the singularity $z_1$ to a vertex representing the singularity $z_2$ in $P(0)$.\newline

Now, turn around the singularity $z_1$, in the anti-clockwise order, starting at the vertex of $\alpha_1$: we first reach a side $\tilde{\alpha_1}$ of $P(0)$, which connects the singularities $z_1$ and $z_d$ (see the left part of Figure \ref{fig:intersections_sides_BM}). For later use, we will orient this side from $z_d$ to $z_1$. Then, we cross an angular sector in every polygon $P(i)$ until reaching a side $\beta_1$ of $P(m-1)$. As $z_2$ comes after $z_1$ in the anti-clockwise order, this side connects $z_1$ to $z_2$. (This is again because the rotation by angle $\frac{2\pi}{n}$ is a symmetry of the surface). We name this side $\beta_1$, and again for later use, we will denote by $\tilde{\beta}_1$ the companion side which comes right after in the anti-clockwise order, see the right part of Figure \ref{fig:intersections_sides_BM}. The side $\tilde{\beta}_1$ is a side of $P(m-1)$ connecting $z_1$ to $z_d$, but which we will orient from $z_d$ to $z_1$. Now, continue turning around $z_1$ until intersecting $\alpha_2$, then its companion side $\tilde{\alpha_2}$, and continue a bit further until reaching another side $\beta_2$ of $P(m-1)$, which again connects $z_1$ to $z_2$, and then its companion side $\tilde{\beta}_2$. The side $\alpha_2$ and $\beta_2$ will be oriented from $z_2$ to $z_1$, whereas $\tilde{\alpha_2}$ and $\tilde{\beta}_2$ are oriented from $z_1$ to $z_d$, see Figure \ref{fig:intersections_sides_BM}.\newline

Now $\alpha = \alpha_1 \cup \alpha_2$ and $\beta = \beta_1 \cup \beta_2$, intersect at both singularities $z_1$ and $z_2$ with the same sign. This is because
\begin{itemize}
\item At the singularity $z_1$, the cyclic order is by construction $\alpha_1, \beta_1, \alpha_2, \beta_2$, that is $\Int_{z_1}(\alpha,\beta) = + 1$. 
\item At the singularity $z_2$, using the symmetry by rotation of angle $\frac{2\pi}{n}$ we obtain that $\Int_{z_2}(\alpha,\beta) = \Int_{z_1}(\tilde{\alpha_1} \cup \tilde{\alpha_2},\tilde{\beta}_1 \cup \tilde{\beta}_2)$, and the cyclic order at $z_1$ between $\tilde{\beta}$ and $\tilde{\alpha}$ gives $\Int_{z_1}(\tilde{\alpha}, \tilde{\beta}) = +1$, as we first see $\tilde{\alpha_1}$, then $\tilde{\beta}_1$, $\tilde{\alpha_2}$ and finally $\tilde{\beta}_2$.
\end{itemize}
\end{proof}

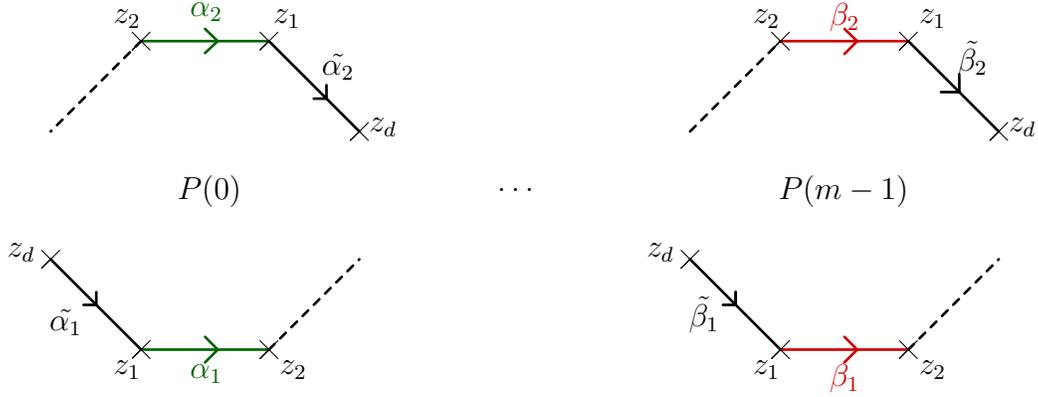
\begin{figure}
\center
\definecolor{ccqqqq}{rgb}{0.8,0,0}
\definecolor{qqwuqq}{rgb}{0,0.39215686274509803,0}
\begin{tikzpicture}[line cap=round,line join=round,>=triangle 45,x=1cm,y=1cm, scale=1.7]
\clip(-1.5,-0.5) rectangle (7.4,3);
\draw [line width=1pt,color=qqwuqq] (0,0)-- (1,0);
\draw [line width=1pt,dash pattern=on 3pt off 3pt] (1,0)-- (1.7071067811865475,0.7071067811865475);
\draw [line width=1pt] (0,0)-- (-0.7071067811865477,0.7071067811865478);
\draw [line width=1pt,color=qqwuqq] (0,2.414213562373095)-- (1,2.414213562373095);
\draw [line width=1pt] (1,2.414213562373095)-- (1.7071067811865475,1.7071067811865472);
\draw [line width=1pt,dash pattern=on 3pt off 3pt] (0,2.414213562373095)-- (-0.7071067811865475,1.7071067811865477);
\draw [line width=1pt,color=ccqqqq] (5,2.414213562373095)-- (6,2.414213562373095);
\draw [line width=1pt] (6,2.414213562373095)-- (6.707106781186548,1.7071067811865472);
\draw [line width=1pt,color=ccqqqq] (5,0)-- (6,0);
\draw [line width=1pt] (5,0)-- (4.292893218813452,0.7071067811865478);
\draw [line width=1pt,dash pattern=on 3pt off 3pt] (6,0)-- (6.707106781186548,0.7071067811865475);
\draw [line width=1pt,dash pattern=on 3pt off 3pt] (4.292893218813452,1.7071067811865477)-- (5,2.414213562373095);
\draw [line width=1pt,color=qqwuqq] (0.6,0)-- (0.5,0.1);
\draw [line width=1pt,color=qqwuqq] (0.6,0)-- (0.5,-0.1);
\draw [line width=1pt] (-0.35,0.35)-- (-0.35,0.45);
\draw [line width=1pt] (-0.35,0.35)-- (-0.45,0.35);
\draw [line width=1pt] (1.4501067811865473,1.9641067811865474)-- (1.45,2.07);
\draw [line width=1pt] (1.4501067811865473,1.9641067811865474)-- (1.3501504201393928,1.964041319887102);
\draw [line width=1pt,color=qqwuqq] (0.5996055080027693,2.414213562373095)-- (0.499322624585732,2.3200604481088187);
\draw [line width=1pt,color=qqwuqq] (0.5996055080027693,2.414213562373095)-- (0.5000654607591916,2.5102265085144606);
\draw [line width=1pt,color=ccqqqq] (5.6,0)-- (5.5,0.1);
\draw [line width=1pt,color=ccqqqq] (5.6,0)-- (5.5,-0.1);
\draw [line width=1pt,color=ccqqqq] (5.599651286446135,2.414213562373095)-- (5.5000621257707545,2.313116672085849);
\draw [line width=1pt,color=ccqqqq] (5.599651286446135,2.414213562373095)-- (5.500700517826366,2.509741425214165);
\draw [line width=1pt] (6.3992634811040086,2.0149500812690864)-- (6.398188745637021,2.155552467807562);
\draw [line width=1pt] (6.3992634811040086,2.0149500812690864)-- (6.255057195384485,2.0156016186717842);
\draw [line width=1pt] (4.65,0.35)-- (4.65,0.45);
\draw [line width=1pt] (4.65,0.35)-- (4.55,0.35);
\draw [line width=1pt] (4.65,0.35)-- (4.65,0.45);
\draw (0.2,1.45) node[anchor=north west] {$P(0)$};
\draw (4.9,1.45) node[anchor=north west] {$P(m-1)$};
\draw [color=qqwuqq](0.3,-0.03) node[anchor=north west] {$\alpha_1$};
\draw (-0.8,0.4) node[anchor=north west] {$\tilde{\alpha_1}$};
\draw [color=qqwuqq](0.3,2.8) node[anchor=north west] {$\alpha_2$};
\draw (1.33,2.4) node[anchor=north west] {$\tilde{\alpha_2}$};
\draw [color=ccqqqq](5.3,-0.03) node[anchor=north west] {$\beta_1$};
\draw (4.2,0.5) node[anchor=north west] {$\tilde{\beta_1}$};
\draw [color=ccqqqq](5.3,2.8) node[anchor=north west] {$\beta_2$};
\draw (6.3,2.5) node[anchor=north west] {$\tilde{\beta_2}$};
\draw (-0.3,0) node[anchor=north west] {$z_1$};
\draw (1,0) node[anchor=north west] {$z_2$};
\draw (-1.13,0.92) node[anchor=north west] {$z_d$};
\draw (0.95,2.75) node[anchor=north west] {$z_1$};
\draw (-0.3,2.75) node[anchor=north west] {$z_2$};
\draw (1.7,1.9) node[anchor=north west] {$z_d$};
\draw (4.7,0) node[anchor=north west] {$z_1$};
\draw (6,0) node[anchor=north west] {$z_2$};
\draw (3.87,0.92) node[anchor=north west] {$z_d$};
\draw (6,2.75) node[anchor=north west] {$z_1$};
\draw (4.7,2.75) node[anchor=north west] {$z_2$};
\draw (6.7,1.9) node[anchor=north west] {$z_d$};
\draw (2.7,1.35) node[anchor=north west] {$\dots$};
\begin{scriptsize}
\draw [color=black] (0,0)-- ++(-2pt,-2pt) -- ++(4pt,4pt) ++(-4pt,0) -- ++(4pt,-4pt);
\draw [color=black] (1,0)-- ++(-2pt,-2pt) -- ++(4pt,4pt) ++(-4pt,0) -- ++(4pt,-4pt);
\draw [color=black] (1.7071067811865475,1.7071067811865472)-- ++(-2pt,-2pt) -- ++(4pt,4pt) ++(-4pt,0) -- ++(4pt,-4pt);
\draw [color=black] (1,2.414213562373095)-- ++(-2pt,-2pt) -- ++(4pt,4pt) ++(-4pt,0) -- ++(4pt,-4pt);
\draw [color=black] (0,2.414213562373095)-- ++(-2pt,-2pt) -- ++(4pt,4pt) ++(-4pt,0) -- ++(4pt,-4pt);
\draw [color=black] (-0.7071067811865477,0.7071067811865478)-- ++(-2pt,-2pt) -- ++(4pt,4pt) ++(-4pt,0) -- ++(4pt,-4pt);
\draw [color=black] (5,0)-- ++(-2pt,-2pt) -- ++(4pt,4pt) ++(-4pt,0) -- ++(4pt,-4pt);
\draw [color=black] (6,0)-- ++(-2pt,-2pt) -- ++(4pt,4pt) ++(-4pt,0) -- ++(4pt,-4pt);
\draw [color=black] (6.707106781186548,1.7071067811865472)-- ++(-2pt,-2pt) -- ++(4pt,4pt) ++(-4pt,0) -- ++(4pt,-4pt);
\draw [color=black] (6,2.414213562373095)-- ++(-2pt,-2pt) -- ++(4pt,4pt) ++(-4pt,0) -- ++(4pt,-4pt);
\draw [color=black] (5,2.414213562373095)-- ++(-2pt,-2pt) -- ++(4pt,4pt) ++(-4pt,0) -- ++(4pt,-4pt);
\draw [color=black] (4.292893218813452,0.7071067811865478)-- ++(-2pt,-2pt) -- ++(4pt,4pt) ++(-4pt,0) -- ++(4pt,-4pt);
\end{scriptsize}
\end{tikzpicture}
\caption{The two pairs of curves $\alpha_1 \cup \alpha_2$ and $\beta_1 \cup \beta_2$ intersect at both singularities $z_1$ and $z_2$ with the same sign, for Bouw-M\"oller surfaces with $1 < \gcd(m,n) < n$.}
\label{fig:intersections_sides_BM}
\end{figure}

\subsection{Polygonal subdivision: Upper bound}\label{sec:upper_bound}
Contrary to the regular $2n-$gon, adjacent segments do not necesarily come in pairs and there could be sequences of consecutive adjacent segments containing an odd number of segments, see \cite{BP24}. Still, in order to estimate the lengths of a saddle connection it is convenient to group adjacent segments in pairs when it is possible, and therefore we define:

\begin{Nota}
Let $\alpha$ be a saddle connection on $S_{m,n}$. We let
\begin{itemize}
\item $p_{\alpha}$ be the number of non-adjacent segments in the polygonal decomposition of $\alpha$;
\item $q_{\alpha}$ be the maximal number of (distinct) pairs of consecutives adjacent segments that can be formed with the polygonal decomposition of $\alpha$.
\item $n_{\alpha} = p_{\alpha} + q_{\alpha}$.
\end{itemize}
\end{Nota}

\begin{Rema}
For the regular $n-$gons, the definition of $n_\alpha$ is consistent with Notation \ref{nota:nalpha1}, because sandwiched segments are pairs of adjacent segments, and in the case of the regular $n-$gons adjacent segments always come in pairs. Since this is not the case anymore here we will not use the terminology of sandwiched/non-sandwiched segments and we will work directly with that of adjacent/non-adjacent segments.
\end{Rema}
It was already noticed in \cite{BP24} that the worst case for the count of intersections is when the non-adjacent segments are always separated by an \emph{odd} number of adjacent segments. A saddle connection satisfying this property will be called an \emph{odd saddle connection}. We have:

\begin{Prop}{\cite[Proposition 3.18]{BP24}}\label{prop:intersections_BM}
Let $m \geq 2, n \geq 3$ and $(m,n) \neq (2,3)$. For any two saddle connections $\alpha$ and $\beta$ on the Bouw-M\"oller surface $S_{m,n}$, we have:
\[ |\alpha \cap \beta| \leq n_{\alpha} n_{\beta} \]
and equality can occur only if both $\alpha$ and $\beta$ are odd saddle connections.
\end{Prop}

This proposition is proven under more general assumptions on the considered translation surface, but we will only use it for Bouw-M\"oller surfaces here. 


Roughly speaking, the idea of \cite{BP24} was to compensate the loss in the count of intersections for odd saddle connections by an adequate length estimate. In order to obtain the upper bound for Theorem \ref{theo:BM}, we provide the following length estimate, whose proof is similar in spirit to that of Proposition \ref{prop:study_lengths}.

\begin{Prop}\label{prop:study_lengths_BM}
Let $m,n \geq 8$ 
Let $\alpha$ be a saddle connection on the Bouw-M\"oller surface $S_{m,n}$ which is not a side of $P(0)$ or $P(m-1)$, we have
\begin{equation}\label{eq:length_estimate_BM}
l(\alpha) \geq \left(\sqrt{2} n_{\alpha} + (\sqrt{2} - 1)\right) l_0.
\end{equation}
\end{Prop}
\begin{Rema}
This proposition is stated for $m \geq 8$ and $n \geq 8$ for simplicity. For $3 \leq m \leq 7$, a similar Proposition holds but there are more cases to consider and we will omit these cases for simplicity. In the case $n < 8$, apart from the cases $(m,n)=(2,4)$ or $(3,4)$ the same proposition should also hold up to small modifications: for example, the short diagonals of $P(0)$ (resp. $P(m-1)$) do not satisfy \eqref{eq:length_estimate_BM} and would have to be considered separately in the whole argument, as we did with the curves of type (2) for the regular $n-$gons.
\end{Rema}
\begin{Rema}
When $m=2$ there are saddle connections similar to the type (2) saddle connections on the regular $n-$gon which do not satisfy Equation \eqref{eq:length_estimate_BM}. If fact, for $m=2$ and $n \geq 10$, the proof of Proposition \ref{prop:study_lengths} extends directly and shows that these saddle connections are the only additional exception up to symmetry (there is no need to consider the analog of saddle connections of type $(3)$ on the double regular $n-$gon). Furthermore, an analysis of the possible cases as done in Section \ref{sec:conclusion} for the $(4m+2)-$gon directly yields the conclusion of Theorem \ref{theo:BM} for $m=2$ (assuming $n \geq 10$ is even).
\end{Rema}

We leave the proof of Proposition \ref{prop:study_lengths_BM} for the final section, and we directly explain how to conclude the proof of Theorem \ref{theo:BM} using Propositions \ref{prop:intersections_BM} and \ref{prop:study_lengths_BM}. 

Given $\gamma = \gamma_1 \cup \cdots \cup \gamma_k$ and $\delta = \delta_1 \cup \cdots \cup \delta_l$ two closed curves made of saddle connections $\gamma_1, \dots, \gamma_k$ (resp. $\delta_1, \dots, \delta_l)$, we have
\begin{align*}
\Int(\gamma,\delta) &\leq \left( \sum_{i,j} |\gamma_i \cap \delta_j| \right) + \min(k,l) \\
&\leq \left( \sum_{i} \left( \sum_{j} |\gamma_i \cap \delta_j| \right) +1 \right).
\end{align*}
Now, we can remark that for every $i,j$ such that neither $\gamma_i$ nor $\delta_j$ is a side of $P(0)$ or $P(m-1)$, we have
\begin{equation}\label{eq:ineq_intersection_BM}
|\gamma_i \cap \delta_j| +1 \leq \frac{1}{2 l_0^2} l(\gamma_i) l(\delta_j).
\end{equation}
Where we get the additional $+1$ from the fact that $2(\sqrt{2}-1)+(\sqrt{2}-1)^2 = 1$.
In particular, as soon as none of the saddle connections of $\gamma$ and $\delta$ are sides of $P(0)$ or $P(m-1)$, we obtain
\begin{align*}
\Int(\gamma,\delta)
&\leq \left( \sum_{i} \left( \sum_{j} |\gamma_i \cap \delta_j| \right) +1 \right)\\
&\leq  \sum_{i} \left( \sum_{j} |\gamma_i \cap \delta_j| +1 \right)\\
&\leq \sum_{i} \left( \sum_{j} \frac{1}{2 l_0^2} l(\gamma_i) l(\delta_j) \right)\\
&\leq \frac{1}{2l_0^2} l(\gamma) l(\delta)
\end{align*}
as required.\newline

We will now assume that one of the curves, say $\gamma$, is made of at least one short side. We will proceed as in the study of the regular $(4m+2)-$gon by grouping all the sides together. 
More precisely, we introduce the following

\begin{Nota}\label{nota:GyD}
\begin{itemize}
\item[(i)] If $\gamma$ (resp. $\delta$) contains at least two saddle connections which are sides of the polygons $P(0), \dots, P(m-1)$, we denote by $\mathfrak{G} = \gamma_{i_1} \cup \cdots \cup \gamma_{i_p}$ (resp. $\mathfrak{D} = \delta_{j_1} \cup \cdots \cup \delta_{j_q}$) the collection of sides that appear in $\gamma$ (resp. $\delta$).
\item[(ii)] If $\gamma$ contains a single short side $\gamma_i$, this side cannot be closed because we assumed $S_{m,n}$ had at least two singularities. In particular, $\gamma$ cannot be reduced to $\gamma_i$ and we can chose another arbitrary saddle connection $\gamma_{i'}$, which is not a side. We define $\mathfrak{G} := \gamma_{i} \cup \gamma_{i'}$. We proceed similarly for $\delta$ if it contains a single short side.
\end{itemize}
\end{Nota}
Further, we will assume that none of the sides of $\mathfrak{G}$ (resp. $\mathfrak{D}$) are adjacent, because otherwise the curve $\gamma$ (resp. $\delta$) would not minimize the length in its homology class. We prove the following:



\begin{Lem}
We have
\begin{equation}\label{eq:mathfrakAB}
|\mathfrak{G} \cap \mathfrak{D}| + \min(\#\mathfrak{G}, \#\mathfrak{D}) \leq \frac{1}{2l_0^2} l(\mathfrak{G}) l(\mathfrak{D}) 
\end{equation}
and, for a saddle connection $\delta_j \in \delta \backslash \mathfrak{D}$ which is not in $\mathfrak{D}$, we have
\begin{equation}\label{eq:mathfrakApasB}
|\mathfrak{G} \cap \delta_j| + 1 \leq \frac{1}{2l_0^2} l(\mathfrak{G}) l(\delta_j) 
\end{equation}
\end{Lem}

\begin{proof}
We have to distinguish three cases according to whether $\mathfrak{G}$ and $\mathfrak{D}$ contain one side (case (ii) of Notation \ref{nota:GyD}) or not (case (i)).
\begin{itemize}
    \item If $\mathfrak{G}$ and $\mathfrak{D}$ are only made of sides, then $|\mathfrak{G} \cap \mathfrak{D}|=0$ and $l(\mathfrak{G}) = \#\mathfrak{G} \cdot l_0$ as well as $l(\mathfrak{D}) = \#\mathfrak{D} \cdot l_0$. Therefore 
    \begin{align*}
    |\mathfrak{G} \cap \mathfrak{D}| + \min(\#\mathfrak{G}, \#\mathfrak{D})  &=  \min(\#\mathfrak{G}, \#\mathfrak{D}) \\
    & =\frac{\#\mathfrak{G}\cdot  \#\mathfrak{D}}{\max(\#\mathfrak{G}, \#\mathfrak{D})} \\
    & = \frac{l(\mathfrak{G}) l(\mathfrak{D})}{\max(\#\mathfrak{G}, \#\mathfrak{D}) l_0^2} \\
    &\leq \frac{1}{2l_0^2} l(\mathfrak{G}) l(\mathfrak{D})  
    \end{align*}
    This gives \eqref{eq:mathfrakAB}. Now, given a saddle connection $\delta_j$, one can obtain \eqref{eq:mathfrakApasB} by noticing that when $\mathfrak{G}$ is made only of sides, the intersections between $\mathfrak{G}$ and $\gamma_j$ only occur between two segments in the polygonal decomposition of $\gamma_j$. However, the assumption on the sides of $\mathfrak{G}$ not being adjacent gives that, along a sequence of consecutive adjacent segments, an intersection with a side of $\mathfrak{G}$ can only occur every two sides crossed. This gives 
    \[ |\mathfrak{G} \cap \delta_j| \leq n_{\delta_j}-1 \]
    and, since $l(\delta_j) \geq n_{\delta_j} l_0$, we conclude that
    \[ |\mathfrak{G} \cap \delta_j| +1 \leq n_{\delta_j} \leq \frac{l(\mathfrak{G}) l(\delta_j)}{\# \mathfrak{G} \cdot l_0^2} \leq \frac{l(\mathfrak{G}) l(\delta_j)}{2 l_0^2}. \]
    which proves \eqref{eq:mathfrakApasB}.
    \item Now, assume that, say $\mathfrak{G} = \gamma_i \cup \gamma_{i'}$ is as in case (ii) of Notation \ref{nota:GyD}, whereas $\mathfrak{D}$ is only made of sides. In this case, we get from Proposition \ref{prop:study_lengths_BM}
    \[ l(\mathfrak{G}) \geq \left( \sqrt2 n_{\gamma_{i'}} + (\sqrt2 -1) + 1 \right) l_0= \sqrt{2} l_0(n_{\gamma_{i'}}+1) \]
    Further, the above argument for the intersection implies that
    \[
    |\mathfrak{G} \cap \mathfrak{D}| = 0 + |\gamma_{i'} \cap \mathfrak{D}| \leq n_{\gamma_{i'}} -1
    \]
    and then 
    \begin{align*}
    |\mathfrak{G} \cap \mathfrak{D}| + \min(\#\mathfrak{G}, \#\mathfrak{D}) &\leq n_{\gamma_{i'}} +1\\
    &\leq  \frac{l(\mathfrak{G}) l(\mathfrak{D})}{\sqrt{2} l_0 \cdot \#\mathfrak{D} \cdot  l_0 }\\
    & \leq \frac{l(\mathfrak{G}) l(\delta_j)}{2 \sqrt2 l_0^2} < \frac{l(\mathfrak{G}) l(\delta_j)}{2 l_0^2} 
    \end{align*}
    Which proves \eqref{eq:mathfrakAB} in this case. Further, the length estimate $l(\delta_j) > \sqrt{2} n_{\delta_j}$, which holds for any $\delta_j \in \delta \setminus \mathfrak{D}$ by Proposition \ref{prop:study_lengths_BM}, gives
    \begin{align*}
    |\mathfrak{G} \cap \delta_j| +1 &\leq \left(n_{\delta_j}-1 \right) + n_{\gamma_{i'}} n_{\delta_j} +1 = \left(n_{\gamma_{i'}}+1 \right) n_{\delta_j}  \\
    & < \frac{l(\mathfrak{G}) l(\delta_j)}{ (\sqrt2 l_0)^2} = \frac{l(\mathfrak{G}) l(\delta_j)}{2 l_0^2}
    \end{align*}
    which is exactly \eqref{eq:mathfrakApasB}.
    \item It remains to prove \eqref{eq:mathfrakAB} when both $\mathfrak{G} = \gamma_i \cup \gamma_{i'}$ and $\mathfrak{D} = \delta_j \cup \delta_{j'}$ are in case (ii). This comes from the series of inequalities:
    \begin{align*}
        |\mathfrak{G} \cap \mathfrak{D}| + \min(\#\mathfrak{G}, \#\mathfrak{D}) &\leq 0+ (n_{\delta_{j'}}-1) + (n_{\gamma_{i'}}-1) + n_{\gamma_{i'}} n_{\delta_{j'}} + 2 \\
        & = (n_{\gamma_{i'}}+1) (n_{\delta_{j'}}+1) -1\\
        & \leq  \frac{l(\mathfrak{G}) l(\mathfrak{D})}{ (\sqrt2 l_0)^2} -1 \\
        & < \frac{1}{2 l_0^2} l(\mathfrak{G}) l(\delta_j)
        \end{align*}
        as required.
\end{itemize}
\end{proof}

As a consequence, we can partition $\gamma$ (resp. $\delta$) using $\mathfrak{G}$ and its complement, and get
\begin{multline*}
\Int(\gamma,\delta)
\leq  |\mathfrak{G} \cap \mathfrak{D}| + \min(\#\mathfrak{G}, \#\mathfrak{D})+
\sum_{\delta_j \notin \mathfrak{D}}  \left(|\mathfrak{G} \cap \delta_j|+1 \right) +\\ \sum_{\gamma_i \notin \mathfrak{G}} \left( |\gamma_i \cap \mathfrak{D}| +1 \right) +
\sum_{\gamma_i \notin \mathfrak{G}} \sum_{\delta_j \notin \mathfrak{D}}  \left( |\gamma_i \cap \delta_j| +1 \right)
\end{multline*}
which, using Equations \eqref{eq:ineq_intersection_BM} \eqref{eq:mathfrakAB} and \eqref{eq:mathfrakApasB}, gives
\begin{align*}
\Int(\gamma,\delta)
&\leq \frac{1}{2l_0^2} \left(l(\mathfrak{G}) l(\mathfrak{D})+ l(\mathfrak{G}) l(\delta \backslash \mathfrak{D}) + l(\gamma \backslash \mathfrak{G}) l(\mathfrak{D}) + l(\gamma \backslash \mathfrak{G})l(\delta \backslash \mathfrak{D}) \right) \\
&\leq \frac{1}{2l_0^2} l(\gamma) l(\delta),
\end{align*}
as required. Therefore it only remains to prove Proposition \ref{prop:study_lengths_BM} in order to complete the proof of Theorem \ref{theo:BM}. This is the purpose of the next section.

\begin{Rema}
    As in Remark \ref{rema:geometric_intersection_1}, one should notice that this argument proves
    \[ 
    \frac{\sum_{i,j} |\gamma_i \cup \delta_j| + s_{\gamma, \delta}}{l(\gamma) l(\delta)} \leq \frac{1}{2l_0^2}
    \]
    when $\gamma$ and $\delta$ are closed curves which are unions of saddle connections that pass at most once through each singularity\footnote{This property is implicitly used in the assumption that none of the sides of $\mathfrak{G}$ (resp. $\mathfrak{D}$) are adjacent.} and $s_{\gamma, \delta}$ is the number of singularities contained in both $\gamma$ and $\delta$. This property extends if we remove the assumption on the curves passing through each singularity at most once, proving Theorem \ref{theo:BM} for the case where the algebraic intersection is replaced by the geometric intersection.
\end{Rema}

\subsection{Proof of Proposition \ref{prop:study_lengths_BM}}
In this final section, we analyze the polygonal decomposition of a saddle connection $\alpha$ into both adjacent and non-adjacent segments. While some of these segments may have length less than $\left(\sqrt{2} + (\sqrt{2} - 1)\right)l_0$, we will demonstrate that grouping them together in a suitable way yield better estimates. This will be achieved by distinguishing various types of segments based on the sides to which their endpoints belong.\newline

For simplicity, we will normalize the length of the sides of the polygons defining $S_{m,n}$ in order to have $l_0=1$. That is, we will consider sides of length $\frac{\sin k \pi/m}{\sin \pi/m}$ instead of $\sin \frac{k \pi}{m}$, for $1 \leq k \leq m-1$. In particular:
\begin{itemize}
\item Every side which is not a side of $P(0)$ or $P(m-1)$ (or identified to such a side) has length at least $\frac{\sin 2 \pi/m}{\sin \pi/m} = 2\cos \frac{\pi}{m}$ which is greater than $2\sqrt{2}-1$ because $m \geq 8$;
\item The length of the short diagonals of $P(0)$ and $P(m-1)$ is $2\cos \frac{\pi}{n}$ which is greater than $2\sqrt{2}-1$ because $n \geq 8$.
\item We will also need the length of the third shortest sides, that is the length of the long sides of $P(2)$ (resp. $P(m-3)$), which is given by
\[ \frac{\sin 3 \pi/m}{\sin \pi/m} = 4\cos^2 \frac{\pi}{m} - 1.\] 
\end{itemize}
\subsubsection{Short and long segments.}
From the above estimates on the length of sides and diagonals of the polygons $P(i)$ we distinguish two types of segments. The following segments are \emph{long} non-adjacent segments (resp. long pairs of adjacent segments), whose length is constraint to be more than $2 \sqrt{2} - 1$ because of their position within the polygons.
\begin{itemize}
\item[(a)] Non-adjacent segments of $P(i)$ for $2 \leq i \leq m-3$;
\item[(b)] Pairs of adjacent segments which share an endpoint on a side of $P(i)$ for $2 \leq i \leq m-3$;
\item[(c)] Non-adjacent segments of $P(1)$ (resp. $P(m-2)$) whose endpoints lie on two sides of $P(1)$  (resp. $P(m-2)$) which are separated by at least a long side of $P(1)$;
\item[(d)] Non-adjacent segments of $P(0)$ (resp. $P(m-1)$) whose endpoints lie on two sides of $P(0)$  (resp. $P(m-1)$) which are separated by at least two sides (when $n \geq 8$ the length of a short diagonal has length greater than $2\sqrt{2}-1$).
\end{itemize}
These estimates come from the fact that the lengths of segments (a) to (d) are greater than the lengths of the segments connecting the endpoints of the sides of the polygon to which their endpoints belong. This argument relies on the fact that each of the polygons is convex with obtuse angles.\newline

Then, the remaining non-adjacent segments (resp. pairs of adjacent segments) may a priori have length less than $2\sqrt{2}-1$. These are:
\begin{itemize}
\item[(e)] Non-adjacent segments of $P(1)$ (resp. $P(m-2)$) (when $m \geq 3$) whose endpoints lie on two sides located on either side of a short side of $P(1)$  (resp. $P(m-2)$) ; 
\item[(f)] Non-adjacent segments of $P(0)$ (resp. $P(m-1)$) whose endpoints lie on two sides located on either side of a short side of $P(0)$ (resp. $P(m-1)$) ; 
\item[(g)] Pairs of adjacent segments with a segment inside $P(0)$ (resp. $P(m-1)$). 
\end{itemize}

\begin{Rema}
Single adjacent segments are left alone since they do not contribute to the count of $n_{\alpha}$.
\end{Rema}

\subsubsection{Grouping the segments.}
For every segment (resp. pair of adjacent segments) of type (e), (f) and (g), we will construct a pair (resp. a triple) with either the following or preceding segment and show that the length of the pair is at least $3\sqrt{2} = 2\sqrt{2} + (\sqrt{2}-1)$. 
We start with a definition:
\begin{Def}
Let $\alpha_i$ be a segment of type (e) or (f). By definition the endpoints of $\alpha_i$ lie on two sides of a polygon $P(i)$ which are located on either side of a side $s_0$ of $P(i)$. The side $s_0$ will be called the supporting side of the segment $\alpha_i$.
\end{Def}
Then, we construct pairs of segments as follows:
\begin{itemize}
\item If $\alpha_i$ is a segment of type (e), we pair $\alpha_i$ with the segment $\alpha_{i \pm 1}$ which is on the side of the endpoint of $\alpha_i$ closest to its supporting side $s_0$.
\item Conversely, if $\alpha_i$ is a segment of type (f), we pair $\alpha_i$ with the segment $\alpha_{i \pm 1}$ which is on the side of the endpoint of $\alpha_i$ furthest to its supporting side $s_0$.
\item Finally, if $\alpha_i \cup \alpha_{i+1}$ is a pair of adjacent segments of type (g), we group $\alpha_{i} \cup \alpha_{i+1}$ with the segment $\alpha_{i-1}$ or $\alpha_{i+2}$ which is directly follows (or precedes) the adjacent segment of $P(0)$. For example, if $\alpha_{i+1}$ belongs to $P(0)$, we group $\alpha_{i} \cup \alpha_{i+1}$ with the segment $\alpha_{i+2}$.
\end{itemize}

With this construction, it is easily shown that:
\begin{Lem}\label{lem:cases_ghi}
For $m,n \geq 8$, the segment paired with a segment of type (e) or (f) 
(resp. a pair of adjacent segments of type (g)) is a long segment. Further, the combined length of the two (resp. three) segments is greater than $3\sqrt{2} -1$.
\end{Lem}
\begin{proof}
We deal with the cases separately. We first consider a segment $\alpha_i$ of type (e). Up to symmetry, we can assume that $\alpha_i$ is contained in $P(1)$ and has supporting side denoted $s_0$, and that the paired segment is $\alpha_{i+1}$. Therefore, we are in the setting of Figure \ref{fig:case_f_setting}. The direction of $\alpha_{i}$ lies between the the directions of the two segments $AB$ and $AC$, and it is easily shown that the slope of $CE$ is greater than the slope of $AC$, using that
\begin{align*}
\vec{AC}= & \left(1 + 2\cos \frac{\pi}{m} \cos\frac{\pi}{n},2\cos \frac{\pi}{m} \sin\frac{\pi}{n}\right) \\
\vec{CE}= & \bigg(2\cos \frac{\pi}{m} \left( 2\cos\frac{\pi}{n}+\cos \frac{3\pi}{n} \right) + \left(4\cos^2 \frac{\pi}{m}-1\right)\left(1 + \cos \frac{2\pi}{n} \right),\\
& \left(4\cos^2 \frac{\pi}{m}-1\right) \sin\frac{2\pi}{n} + 2\cos \frac{\pi}{m} \sin \frac{3\pi}{n} \bigg).
\end{align*}
Further, the length of $\alpha_i \cup \alpha_{i+1}$ is greater than the length of $AD$ given by
\[ 4 \cos \frac{\pi}{m} \cos \frac{\pi}{n} + 1 + \left(4\cos^2 \frac{\pi}{m} - 1 \right)\]
which is greater than $3 \sqrt{2} - 1$, as required.\newline

We now consider a segment $\alpha_i$ of type (f). Up to symmetry, we can assume that $\alpha_i$ is contained in $P(0)$ and has supporting side denoted $s_0$, and that the paired segment is $\alpha_{i+1}$. In the setting of Figure \ref{fig:short_cylinder_BM}, $\alpha_i$ has its two endpoints on the sides $s_3$ and $s_1$, and its slope is positive. Therefore, the segment $\alpha_{i+1}$ which is paired with $\alpha_i$ is a long non-adjacent segment of $P(1)$ with an endpoint on the side $s_1$ and its other endpoint on one of the sides $s_0, s_2$ or $s_3$. The length of $\alpha_i \cup \alpha_{i+1}$ is therefore greater than the length of $AC$, which is by construction 
\[2\cos \frac{\pi}{n} + 2 \cos \frac{\pi}{m}. \]
This is greater than $3 \sqrt{2} - 1$, as required.\newline

The study is similar in case (g), see Figure \ref{fig:short_cylinder_BM}. 
\end{proof}

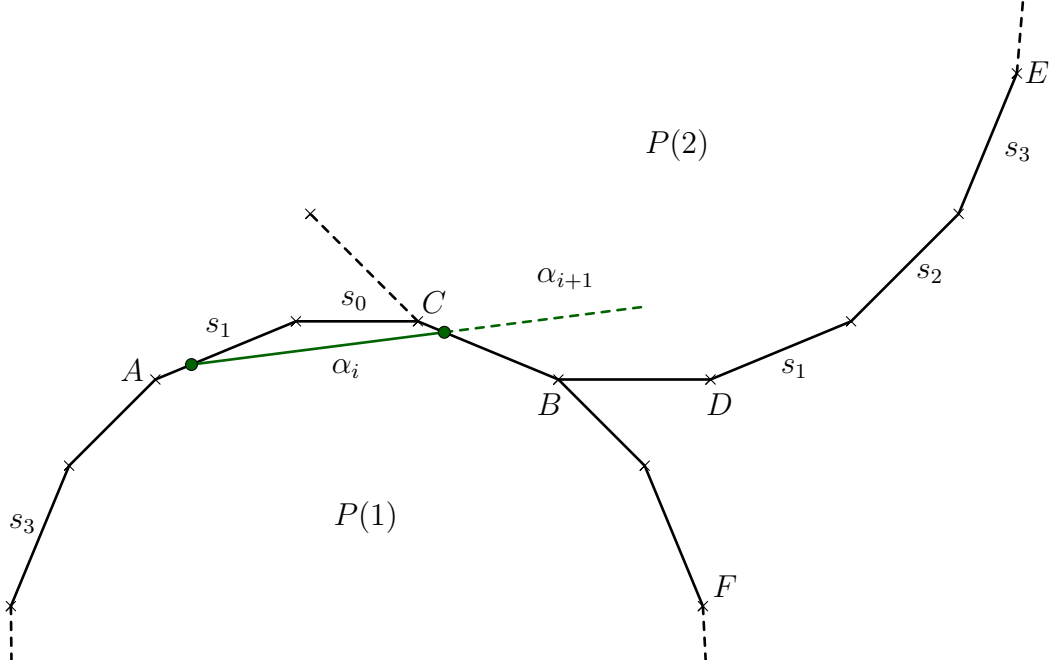
\begin{figure}[h]
\center
\definecolor{qqwuqq}{rgb}{0,0.39215686274509803,0}
\begin{tikzpicture}[line cap=round,line join=round,>=triangle 45,x=1cm,y=1cm,scale=0.9]
\clip(-4.5,-5) rectangle (12,5.5);
\draw [line width=1pt] (-2.07593846584489,-0.8598818672049496)-- (0,0);
\draw [line width=1pt] (0,0)-- (1.8019377358048383,0);
\draw [line width=1pt] (1.8019377358048383,0)-- (3.877876201649728,-0.85988186720495);
\draw [line width=1pt] (3.877876201649728,-0.85988186720495)-- (6.124855805367196,-0.85988186720495);
\draw [line width=1pt] (6.124855805367196,-0.85988186720495)-- (8.200794271212086,0);
\draw [line width=1pt] (8.200794271212086,0)-- (9.789648786188568,1.588854514976482);
\draw [line width=1pt] (9.789648786188568,1.588854514976482)-- (10.649530653393517,3.664792980821373);
\draw [line width=1pt] (-3.3501008581084246,-2.1340442594684843)-- (-2.07593846584489,-0.8598818672049496);
\draw [line width=1pt] (3.877876201649728,-0.85988186720495)-- (5.152038593913263,-2.1340442594684847);
\draw [line width=1pt,dash pattern=on 3pt off 3pt] (1.8019377358048383,0)-- (0.21308322082835573,1.5888545149764826);
\draw [line width=1pt,color=qqwuqq] (-1.5425808197306707,-0.6389578965890503)-- (2.1923461392282944,-0.1617124555624223);
\draw [line width=1pt,dash pattern=on 3pt off 3pt,color=qqwuqq] (2.1923461392282944,-0.1617124555624223)-- (5.107493004626339,0.21078225898537228);
\draw (0.370098616820362,-0.4) node[anchor=north west] {$\alpha_i$};
\draw (3.4,0.9449970126053061) node[anchor=north west] {$\alpha_{i+1}$};
\draw (0.4,-2.5) node[anchor=north west] {$P(1)$};
\draw (5,3) node[anchor=north west] {$P(2)$};
\draw (-2.75,-0.4) node[anchor=north west] {$A$};
\draw (3.4,-0.9) node[anchor=north west] {$B$};
\draw [line width=1pt] (-4.209982725313374,-4.209982725313376)-- (-3.3501008581084246,-2.1340442594684843);
\draw (1.7,0.6) node[anchor=north west] {$C$};
\draw (-1.5,0.2) node[anchor=north west] {$s_1$};
\draw (0.5,0.6) node[anchor=north west] {$s_0$};
\draw (7,-0.4) node[anchor=north west] {$s_1$};
\draw (9,1) node[anchor=north west] {$s_2$};
\draw (10.3,2.8) node[anchor=north west] {$s_3$};
\draw (5.9,-0.9) node[anchor=north west] {$D$};
\draw (10.6,4) node[anchor=north west] {$E$};
\draw (6,-3.6) node[anchor=north west] {$F$};
\draw [line width=1pt] (5.152038593913263,-2.1340442594684847)-- (6.011920461118211,-4.209982725313377);
\draw (-4.4,-2.7) node[anchor=north west] {$s_3$};
\draw [line width=1pt,dash pattern=on 3pt off 3pt] (10.649530653393517,3.664792980821373)-- (10.740823589989088,4.750767644960777);
\draw [line width=1pt,dash pattern=on 3pt off 3pt] (-4.209982725313374,-4.209982725313376)-- (-4.196826142006233,-5.501026995946773);
\draw [line width=1pt,dash pattern=on 3pt off 3pt] (6.011920461118211,-4.209982725313377)-- (6.078754565353606,-5.524813062398994);
\begin{scriptsize}
\draw [color=black] (0,0)-- ++(-2pt,-2pt) -- ++(4pt,4pt) ++(-4pt,0) -- ++(4pt,-4pt);
\draw [color=black] (1.8019377358048383,0)-- ++(-2pt,-2pt) -- ++(4pt,4pt) ++(-4pt,0) -- ++(4pt,-4pt);
\draw [color=black] (-2.07593846584489,-0.8598818672049496)-- ++(-2pt,-2pt) -- ++(4pt,4pt) ++(-4pt,0) -- ++(4pt,-4pt);
\draw [color=black] (3.877876201649728,-0.85988186720495)-- ++(-2pt,-2pt) -- ++(4pt,4pt) ++(-4pt,0) -- ++(4pt,-4pt);
\draw [color=black] (6.124855805367196,-0.85988186720495)-- ++(-2pt,-2pt) -- ++(4pt,4pt) ++(-4pt,0) -- ++(4pt,-4pt);
\draw [color=black] (8.200794271212086,0)-- ++(-2pt,-2pt) -- ++(4pt,4pt) ++(-4pt,0) -- ++(4pt,-4pt);
\draw [color=black] (9.789648786188568,1.588854514976482)-- ++(-2pt,-2pt) -- ++(4pt,4pt) ++(-4pt,0) -- ++(4pt,-4pt);
\draw [color=black] (10.649530653393517,3.664792980821373)-- ++(-2pt,-2pt) -- ++(4pt,4pt) ++(-4pt,0) -- ++(4pt,-4pt);
\draw [color=black] (-3.3501008581084246,-2.1340442594684843)-- ++(-2pt,-2pt) -- ++(4pt,4pt) ++(-4pt,0) -- ++(4pt,-4pt);
\draw [color=black] (5.152038593913263,-2.1340442594684847)-- ++(-2pt,-2pt) -- ++(4pt,4pt) ++(-4pt,0) -- ++(4pt,-4pt);
\draw [color=black] (0.21308322082835573,1.5888545149764826)-- ++(-2pt,-2pt) -- ++(4pt,4pt) ++(-4pt,0) -- ++(4pt,-4pt);
\draw [fill=qqwuqq] (-1.5425808197306707,-0.6389578965890503) circle (2.5pt);
\draw [fill=qqwuqq] (2.1923461392282944,-0.1617124555624223) circle (2.5pt);
\draw [color=black] (-4.209982725313374,-4.209982725313376)-- ++(-2pt,-2pt) -- ++(4pt,4pt) ++(-4pt,0) -- ++(4pt,-4pt);
\draw [color=black] (6.011920461118211,-4.209982725313377)-- ++(-2pt,-2pt) -- ++(4pt,4pt) ++(-4pt,0) -- ++(4pt,-4pt);
\end{scriptsize}
\end{tikzpicture}
\caption{A segment $\alpha_i$ of type (e). The segment $\alpha_{i+1}$ can have an endpoint $\alpha_{i+1}^+$ on either side $s_1,s_2$ or $s_3$.}
\label{fig:case_f_setting}
\end{figure}

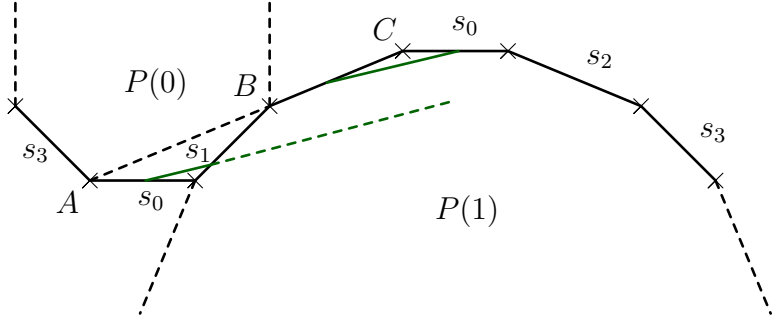
\begin{figure}[h]
\center
\definecolor{qqwuqq}{rgb}{0,0.39215686274509803,0}
\begin{tikzpicture}[line cap=round,line join=round,>=triangle 45,x=1cm,y=1cm, scale = 1.4]
\clip(-1,-1.5) rectangle (7,1.7);
\draw [line width=1pt] (0,0)-- (1,0);
\draw [line width=1pt] (1,0)-- (1.7071067811865475,0.7071067811865476);
\draw [line width=1pt] (0,0)-- (-0.7071067811865475,0.7071067811865475);
\draw [line width=1pt] (1.7071067811865475,0.7071067811865476)-- (2.9707814551021867,1.2305379695498038);
\draw [line width=1pt] (2.9707814551021867,1.2305379695498038)-- (3.9707814551021867,1.2305379695498038);
\draw [line width=1pt] (3.9707814551021867,1.2305379695498038)-- (5.2344561290178255,0.7071067811865474);
\draw [line width=1pt] (5.2344561290178255,0.7071067811865474)-- (5.941562910204373,0);
\draw [line width=1pt,dash pattern=on 3pt off 3pt] (-0.7071067811865475,0.7071067811865475)-- (-0.7071067811865475,1.7071067811865475);
\draw [line width=1pt,dash pattern=on 3pt off 3pt] (1.7071067811865475,0.7071067811865476)-- (1.7071067811865475,1.7071067811865475);
\draw [line width=1pt,dash pattern=on 3pt off 3pt] (0,0)-- (1.7071067811865475,0.7071067811865476);
\draw [line width=1pt,color=qqwuqq] (2.25,0.93)--(3.5,1.23);
\draw [line width=1pt,color=qqwuqq] (0.53,0)--(1.15,0.15);
\draw [line width=1pt,color=qqwuqq,dash pattern=on 3pt off 3pt] (1.15,0.15)--(3.43,0.75);
\draw (0.2192317991809927,1.1692332815314035) node[anchor=north west] {$P(0)$};
\draw (0.35,0) node[anchor=north west] {$s_0$};
\draw (0.8,0.45) node[anchor=north west] {$s_1$};
\draw (3.3308632387419945,1.663604631742028) node[anchor=north west] {$s_0$};
\draw (0,0) node[anchor=north east] {$A$};
\draw (1.7,0.7) node[anchor=south east] {$B$};
\draw (2.8,1.23) node[above] {$C$};
\draw (4.610412615757733,1.3437172874880945) node[anchor=north west] {$s_2$};
\draw (5.686397319157332,0.6457812636613304) node[anchor=north west] {$s_3$};
\draw (-0.75,0.45) node[anchor=north west] {$s_3$};
\draw (3.1757663445582685,-0.032767648392468034) node[anchor=north west] {$P(1)$};
\draw [line width=1pt,dash pattern=on 3pt off 3pt] (1,0)-- (0.4765688116367429,-1.2636746739156395);
\draw [line width=1pt,dash pattern=on 3pt off 3pt] (5.941562910204373,0)-- (6.46499409856763,-1.2636746739156395);
\begin{scriptsize}
\draw [color=black] (0,0)-- ++(-2pt,-2pt) -- ++(4pt,4pt) ++(-4pt,0) -- ++(4pt,-4pt);
\draw [color=black] (1,0)-- ++(-2pt,-2pt) -- ++(4pt,4pt) ++(-4pt,0) -- ++(4pt,-4pt);
\draw [color=black] (1.7071067811865475,0.7071067811865476)-- ++(-2pt,-2pt) -- ++(4pt,4pt) ++(-4pt,0) -- ++(4pt,-4pt);
\draw [color=black] (-0.7071067811865475,0.7071067811865475)-- ++(-2pt,-2pt) -- ++(4pt,4pt) ++(-4pt,0) -- ++(4pt,-4pt);
\draw [color=black] (2.9707814551021867,1.2305379695498038)-- ++(-2pt,-2pt) -- ++(4pt,4pt) ++(-4pt,0) -- ++(4pt,-4pt);
\draw [color=black] (3.9707814551021867,1.2305379695498038)-- ++(-2pt,-2pt) -- ++(4pt,4pt) ++(-4pt,0) -- ++(4pt,-4pt);
\draw [color=black] (5.2344561290178255,0.7071067811865474)-- ++(-2pt,-2pt) -- ++(4pt,4pt) ++(-4pt,0) -- ++(4pt,-4pt);
\draw [color=black] (5.941562910204373,0)-- ++(-2pt,-2pt) -- ++(4pt,4pt) ++(-4pt,0) -- ++(4pt,-4pt);
\end{scriptsize}
\end{tikzpicture}
\caption{The segment paired with a pair of adjacent segments of type (g) is a long segment contained in $P(1)$ which has one endpoint on $s_1$ and one endpoint on either of the sides $s_0, s_2$ or $s_3$. This is similar for the segment paired with a non-adjacent segment of type (f) whose supporting side is $s_0$. The main reason this holds is because the vertices $A, B$ and $C$ are aligned (the segments $AB$ and $BC$ both make an angle $\frac{\pi}{n}$ with the horizontal).}
\label{fig:short_cylinder_BM}
\end{figure}

\begin{Rema}
Contrary to the case of the $(4m+2)-$gon, there is no need to consider trips through short cylinders. In fact, trips through short cylinders have to be considered if one wants to prove Proposition \ref{prop:study_lengths_BM} for $3 \leq m \leq 7$. This is one of the reasons we choose to assume $m \geq 8$.
\end{Rema}

\subsubsection{There are no overlaps.}
We now show that, for $m \geq 8$, the pairs cannot overlap. For this, first remark that:
\begin{itemize}
\item A segment of type (e) is always grouped to a long segment of $P(2)$ (or $P(m-3)$) which has at least one of its endpoints on a short segment of $P(2)$ (or $P(m-3)$), namely the endpoint shared with the segment of type (e).
\item The segments of type (f) and the pairs of adjacent segments of type (g) are grouped to a long segment of $P(1)$ (or $P(m-2)$).
\end{itemize}
In particular, the only possible overlaps are among
\begin{enumerate}
\item Two pairs formed by segments of type (e);
\item Pairs formed by segments of type (f) and/or by triples with two adjacent segments of type (g);
\end{enumerate}
We deal with the two cases separately:
\begin{enumerate}
\item In the first case we consider a pair of segments containing a segment of type (e) denoted $\alpha_{i}$. As in the proof of Lemma \ref{lem:cases_ghi}, we can assume up to symmetry that we are in the setting of Figure \ref{fig:case_f_setting}. Also recall from the proof of Lemma \ref{lem:cases_ghi} that the endpoint $\alpha_{i+1}^+$ of $\alpha_{i+1}$ which is not shared with $\alpha_i$ belong to either side $s_1$, $s_2$ or $s_3$. Now,
\begin{itemize}
\item If $\alpha_{i+1}^+$ of $\alpha_{i+1}$ is a vertex of $P(2)$, the segment $\alpha_{i+1}$ is obviously only in the pair with $\alpha_i$.
\item If $\alpha_{i+1}^+$ belongs to the (interior of the) side $s_1$, the next segment $\alpha_{i+2}$ is either an adjacent segment of $P(1)$ or a segment of type (e). In the first case, the adjacent segment $\alpha_{i+2}$ is followed by another adjacent segment $\alpha_{i+3}$, contained in $P(0)$, and they form a pair of type (g). This pair is grouped with the next segment $\alpha_{i+4}$, and therefore there is no overlap. In the second case, considering the slope of $\alpha_{i+2}$ we deduce that the endpoint of $\alpha_{i+2}$ closer to its supporting side (which is also $E$) is the one shared with $\alpha_{i+3}$, and therefore the pairs formed with the two segments of type (e) $\alpha_i$ and $\alpha_{i+2}$ do not overlap.
\item If the endpoint $\alpha_{i+1}^+$ belongs to $s_2$, the segment $\alpha_{i+2}$ is a long segment of $P(3)$ (or an adjacent segment which is part of a long pair of adjacent segments) and we conclude that $\alpha_{i+1}$ only belongs to a pair with $\alpha_i$.

\item Finally, if the endpoint $\alpha_{i+1}^+$ belongs to $s_3$, it is immediately seen from slope considerations that $\alpha_ {i+2}$ must be a long segment of $P(1)$, and therefore $\alpha_{i+1}$ only belongs to the pair with $\alpha_{i}$. More precisely, since the direction of $\alpha_{i}$ (and therefore of $\alpha_{i+2}$) lies between the horizontal and the direction of the segment $AC$, see Figure \ref{fig:case_f_setting}), we obtain that $\alpha_{i+2}$ has at its second endpoint between $C$ and $F$ and there are at least three separating sides between its endpoints. 
\end{itemize}
\item In the second case we consider a pair containing a segment of type (f) (resp. a pair of adjacent segment of type (g) along with its paired non-adjacent segment), say contained in $P(0)$ and $P(1)$, and we let $s_0$ be the supporting side of the segment of type (f) (resp. the side of $P(0)$ --and $P(1)$-- which lies between the two adjacent segments, denoted $s_0$ in the setting of Figure \ref{fig:short_cylinder_BM}). Therefore, only three cases can occur: the long segment $\alpha_i$ which is grouped to the segment of type (f) (resp. to the pair of adjacent segments of type (g)) can
\begin{itemize}
\item have an endpoint on the side $s_0$, in which case the next segment $\alpha_{i+1}$ is an adjacent segment of $P(0)$, followed and preceded by non-adjacent segments, and therefore it does not contribute to the count of pairs of adjacent segments and is left alone. Therefore, there is no overlap in this case;
\item Have an endpoint on the side $s_2$ (endpoints included) and in this case either there is no next segment if the endpoint is a vertex or the next segment belong to $P(2)$ and therefore it is not grouped to any other segment. This gives that $\alpha_i$ do not belong to any other pair (resp. triple);
\item Or have an endpoint on the side $s_3$ (endpoints excluded), in which case the next segment is either a long segment of $P(0)$ (because $n \geq 6$), and this gives that $\alpha_i$ do not belong to any other pair (resp. group of segment). Or the next segment has type (f) with supporting side $s_0$. In this case, since the slope of $\alpha$ is positive the endpoint of this segment further to $s_0$ is the one which is not shared with $\alpha_i$, and therefore this segment of type (f) is not paired with $\alpha_i$. This shows that $\alpha_i$ is not paired to any other segment (resp. pair of adjacent segments).
\end{itemize}
\end{enumerate}

As a conclusion, we managed to subdivide the saddle connection $\alpha$ into groups of segments which have length at least $\sqrt{2}k + (\sqrt{2}-1)$ while they contain $k=1$ or $2$ non-adjacent segments and pairs of adjacent segments. This gives the required result.

\begin{Rema}
For $3 \leq m \leq 7$, there are other types of short segments, and some of the pairs of segments constructed may actually overlap. In this case, one should group the overlapping pairs of segments into triples and obtain adequate length estimates. This is the second reason why we assume $m \geq 8$. 
\end{Rema}

\bibliographystyle{plain}
\bibliography{KVol_bibli}
\end{document}